\documentclass{amsart}
\usepackage{longsalch, xy, eufrak}
\xyoption{all}
\title{Ravenel's Global Conjecture is true.}
\date{April 2015.}
\begin{document}
\maketitle
\begin{abstract}
I prove Ravenel's 1983 ``Global Conjecture'' on $\Ext^1$ over the classifying Hopf algebroid of formal $A$-modules,
equivalently, the first flat cohomology group $H^1_{fl}$ of the moduli stack $\mathcal{M}_{fmA}$ of formal $A$-modules. I then
show that the Hecke $L$-functions of certain Gro\ss{e}ncharakters of Galois extensions $K/\mathbb{Q}$ can be 
computed from $H^1_{fl}(\mathcal{M}_{fmA})$, and vice versa; as a consequence I show that, for a large class of Galois extensions 
of $\mathbb{Q}$, two extensions $K,L$ are arithmetically equivalent (i.e., they have the same Dedekind zeta-function) if and only if
the flat cohomology groups $H^1_{fl}(\mathcal{M}_{fm\mathcal{O}_K})$ and $H^1_{fl}(\mathcal{M}_{fm\mathcal{O}_L})$ agree.
\end{abstract}
\tableofcontents

\section{Introduction.}

In this paper\footnote{This paper is the sixth in a series of papers on formal groups with complex multiplication and their applications in homotopy theory, but it can be read without reference to the other papers in the series.} I prove the ``Global Conjecture'' from Ravenel's 
1983 paper~\cite{MR745362}. This conjecture is about the cohomology groups of the moduli stack of one-dimensional formal $A$-modules, or equivalently,
$\Ext$ over the classifying Hopf algebroid $(L^A,L^AB)$ of formal $A$-modules;
I will now explain a bit about what this means.

A ``formal $A$-module'' is a formal group law $F$ over an $A$-algebra $R$
which is equipped with a ring homomorphism $\rho: A \rightarrow \End(F)$
such that the endomorphism $\rho(a)\in \End(F) \subseteq R[[X]]$
is congruent to $aX$ modulo $X^2$. Morally, $F$ is a ``formal group law with complex multiplication by $A$.''
An excellent introductory reference for formal $A$-modules is chapter 21 of~\cite{MR2987372}. Formal $A$-modules arise
in algebraic geometry, for example, in Lubin and Tate's famous theorem (in~\cite{MR0172878}) on the abelian closure of a $p$-adic number field, and for another example, in Drinfeld's $p$-adic symmetric domains, which are (rigid analytic) deformation
spaces of certain formal modules; see~\cite{MR0422290} and~\cite{MR1393439}. 
Formal $A$-modules also arise in algebraic topology, by using the natural map from the moduli stack of formal $A$-modules to the moduli stack of formal group laws to detect certain classes in the cohomology of the latter, particularly in order to resolve certain differentials in spectral sequences used to compute the Adams-Novikov $E_2$-term and stable homotopy groups of spheres; in a paper currently in preparation, for example, I use formal modules to compute the homotopy groups of the $K(4)$-local Smith-Toda complex $V(3)$ at primes $p>5$.

Now I want to give a little bit of background on the moduli theory of formal modules before I describe the new results.
Recall (from e.g. chapter~21 of~\cite{MR2987372} or \cite{MR745362}) that, for every finite extension $K/\mathbb{Q}$ with ring of integers $A$, there exists
a classifying Hopf algebroid $(L^A,L^AB)$ for one-dimensional formal $A$-modules, equivalently,
a moduli stack $\mathcal{M}_{fmA}$ of one-dimensional formal $A$-modules. 
In the special case $K = \mathbb{Q}$, we have an isomorphism of Hopf algebroids
$(L^{\mathbb{Z}}, L^{\mathbb{Z}}B) \cong (L,LB)$, where $(L,LB)$ is the classifying Hopf algebroid for one-dimensional formal group laws studied by Lazard (see~\cite{MR0393050}) and which is very familiar to topologists due to Quillen proving (see~\cite{MR0253350}) that it is isomorphic to the Hopf algebroid $(MU_*,MU_*MU)$ of
stable co-operations in complex cobordism, and consequently that the bigraded $\Ext$-algebra $\Ext^{*,*}_{(L^{\mathbb{Z}},L^{\mathbb{Z}}B)}(L^{\mathbb{Z}},L^{\mathbb{Z}})$, i.e., the flat cohomology $H^{*,*}_{fl}(\mathcal{M}_{fm{\mathbb{Z}}}; \mathcal{O})$,
is the $E_2$-term of the Adams-Novikov spectral sequence converging to the stable homotopy groups of spheres. 

There is also a $p$-adic version of the above: for every finite extension $K/\mathbb{Q}_p$ with ring of integers $A$, there exists a classifying Hopf algebroid
$(L^A,L^AB)$ for one-dimensional formal $A$-modules, and a classifying Hopf algebroid $(V^A,V^AT)$ for $A$-typical one-dimensional formal $A$-modules, and the two Hopf algebroids are equivalent (but not isomorphic) by a $A$-module version of
Cartier's ``$p$-typicalization'' operation on formal group laws. 
See e.g.~\cite{MR745362} for this material.
Again, the base case $K = \mathbb{Q}_p$ is very familiar to topologists, since
$(V^{\hat{\mathbb{Z}}_p}, V^{\hat{\mathbb{Z}}_p}T)$ is isomorphic to
$(\hat{V}_p, \hat{VT}_p)$, the $p$-adic completion of the classifying
Hopf algebroid $(V,VT)$ of $p$-typical one-dimensional formal group laws over $\mathbb{Z}_{(p)}$, which Quillen proved in~\cite{MR0253350}
to be isomorphic to the Hopf algebroid $(BP_*, BP_*BP)$
of stable co-operations in Brown-Peterson homology.

Now the following conjectures were made in Ravenel's 1983 paper~\cite{MR745362}:
\begin{conjecture} {\bf (Ravenel's Local-Global Conjecture.)} \label{local-global conj}
Let $K/\mathbb{Q}$ be a finite field extension with ring of integers $A$.
Then, for all $s,t$ and all prime ideals $\mathfrak{p}$ in $A$
and all graded $(L^A,L^AB)$-comodules $M$,
there exists an isomorphism of $A_{\mathfrak{p}}$-modules
\[ A_{\mathfrak{p}} \otimes_A \Ext_{(L^A,L^AB)}^{s,t}(L^A, M) \cong
 \Ext_{(V^{A_{\mathfrak{p}}},V^{A_{\mathfrak{p}}}T)}^{s,t}(V^{A_{\mathfrak{p}}}, V^{A_{\mathfrak{p}}}\otimes_{L^{A_{\mathfrak{p}}}} M), \]
where $A_{\mathfrak{p}}$ is the localization of $A$ at the prime ideal $\mathfrak{p}$.
\end{conjecture}
This statement of Conjecture~\ref{local-global conj} is slightly paraphrased
from Ravenel's original statement in~\cite{MR745362}, which has a small
error 
(one of the terms $L^A$ and one of the terms $V^{A_{\mathfrak{p}}}$ are replaced
by $A$ and $A_{\mathfrak{p}}$, respectively, in Ravenel's statement of the conjecture; in context it is clear that this is a typo).

I will quote Conjecture~\ref{global conj} verbatim from \cite{MR745362}, because
it is loosely stated, and reasonably so: one of the tasks involved in proving
Conjecture~\ref{global conj} is to find a way to make its statement precise
in such a way that the conjecture is actually true.
\begin{conjecture} {\bf (Ravenel's Global Conjecture.)} \label{global conj}
``For global $A$, $\Ext_A^{1,2m} = A/J_m^A$, where $J_m^A$ is, up to some small factor, the ideal generated by $a^N(a^m-1)$ for $a\in A$ and $N$ sufficiently large.''
\end{conjecture}
Since the $\Ext$ groups $\Ext_{(L^A,L^AB)}^{1,i}(L^A,L^A)$
are easily shown to be trivial if $i$ is odd,
Conjecture~\ref{global conj} is a complete description
of $\Ext_{(L^A,L^AB)}^{1}(L^A,L^A)$.

\begin{conjecture} {\bf (Ravenel's Local Conjecture.)} \label{local conj}
Let $K/\mathbb{Q}_p$ be a finite field extension with ring of integers $A$.
Suppose $A$ has uniformizer $\pi$ and residue field $\mathbb{F}_q$.
Then we have an isomorphism of $A$-modules
\[ \Ext_{(V^A,V^AT)}^{1,2n(q-1)}(V^A,V^A) \cong A/I_{n},\]
where $I_{n}$ is the ideal in $A$ generated by all elements of the 
form $a^n-1$ with $a\in A$ congruent to $1$ modulo $\pi$.
\end{conjecture}
Since the $\Ext$ groups $\Ext_{(V^A,V^AT)}^{1,m}(V^A,V^A)$
are easily shown to be trivial if $2(q-1)$ does not divide $m$,
Conjecture~\ref{local conj} is a complete description
of $\Ext_{(V^A,V^AT)}^{1}(V^A,V^A)$.

There is a fourth remark in Ravenel's paper~\cite{MR745362}
which is not phrased as a conjecture (indeed, Ravenel includes it as a reason for not making a fourth conjecture!) but which I regard as being of equal importance as Conjectures~\ref{local-global conj},~\ref{global conj}, and~\ref{local conj}, and addressing (and solving) the problem posed by Ravenel's remark is one
of the main topics of the present paper. I include this remark verbatim from Ravenel's paper~\cite{MR745362}:
\begin{remark} {\bf (Ravenel's remark on the connection to the Dedekind $\zeta$-function.)} \label{ravenel remark on dedekind zeta}
``The numbers $j_m$ of 3.8 are also related to Bernoulli numbers and the values
of the Riemann zeta function at negative integers, but these properties
do not appear to generalize to other number fields. For example if the field
is not totally real its Dedekind zeta function vanishes at all negative integers.''\end{remark}
The ``numbers $j_m$'' refers to the order of the group
$\Ext^{1,2m}_{(L^{\mathbb{Z}},L^{\mathbb{Z}}B)}(L^{\mathbb{Z}},L^{\mathbb{Z}})$,
which Adams proved, in~\cite{MR0198468} and~\cite{MR0198470},
is equal to the denominator of the rational number $\zeta(1-2m)$, up to multiplication
by a power of $2$, where $\zeta$ is the Riemann $\zeta$-function;
since the Riemann $\zeta$-function is the special case $K=\mathbb{Q}$
of the Dedekind $\zeta$-function of a number field $K$, Ravenel's 
Remark~\ref{ravenel remark on dedekind zeta} appears to exclude the possibility
of a generalization of Adams's result to other number fields.

The state of all of these conjectures, as well as the progress made toward their resolution in the present paper, is as follows:
\begin{itemize}
\item Conjecture~\ref{local-global conj}, the Local-Global Conjecture,
was known classically in the $K=\mathbb{Q}$ case at the time that Ravenel
made the conjecture. 
The conjecture was proven in full generality (using Cartier typicalization and a relatively straightforward Hopf algebroid change-of-rings argument) 
by A. Pearlman in his (unpublished) thesis,~\cite{pearlmanthesis},
shortly after Ravenel made the conjecture (indeed, Pearlman's proof came so soon after Ravenel made the conjecture that Ravenel notes, in the published version of~\cite{MR745362}, that
the conjecture had already been proven by Pearlman!).
\item 
Conjecture~\ref{local conj}, the Local Conjecture,
was classically known only in the case $K = \mathbb{Q}_p$ when Ravenel made the conjecture. When Ravenel posed the conjecture in~\cite{MR745362},
he also offered a proof of the conjecture for all extensions $K/\mathbb{Q}_p$
with ramification degree $e$ satisfying $e<p-1$. (There is a 
flaw in the argument Ravenel presents in that paper: he relies on a formula for $p$-adic valuations of binomial coefficients, $\nu_p(\binom{jp^i}{k}) = i-\nu_p(k)$ for $j$ prime to $p$, a formula which is often
not true unless $j=1$. For example, let $j=p+1$, let $i=1$, and let $k=p$ for a counterexample. An argument similar to Ravenel's does prove the Local Conjecture in those cases, however.)

In K. Johnson's 1987 paper~\cite{MR887512}, Johnson proves the Local Conjecture for all extensions $K/\mathbb{Q}_p$ which are not a totally ramified extension of $\mathbb{Q}(\zeta_p)$ of degree a power of $p$, where $\zeta_p$ is a primitive $p$th root of unity.
In the present paper (see Corollary~\ref{local computation}) 
I prove Conjecture~\ref{local conj}
for all extensions $K/\mathbb{Q}_p$ such that 
$\log_p(\frac{e}{p-1})$ is not an integer, where 
$e$ is the ramification degree of $K/\mathbb{Q}_p$. The computation I provide is no
more general than Johnson's, but I give the computation anyway because it is self-contained and because it may be useful to have more than one proof of the Local Conjecture in these cases in the literature (and because I worked it out before I was aware of Johnson's paper).
\item 
Conjecture~\ref{global conj}, the Global Conjecture,
was classically known only in the case $K = \mathbb{Q}$ when Ravenel made the conjecture, and this is still the only case of the Global Conjecture which has
been proven before the results of the present paper: no progress has been made on the Global Conjecture in the 32 years since~\cite{MR745362} was written. (The ``small factor''
in the statement of Conjecture~\ref{global conj} is $2$
in the case $K = \mathbb{Q}$.)

In the present paper I address the issue of making the statement of the Global Conjecture precise by defining an {\em $n$-congruing ideal} (Definition~\ref{def of n-congruing ideal}) to be an ideal $I$ of $A$ with the property that, for each $a\in A$, there exists $N\in\mathbb{N}$ such that $a^N(a^n-1)\in I$. I then prove (in Theorem~\ref{hasse principle}) a Hasse principle for $n$-congruing ideals, that is, an ideal in $I$ is $n$-congruing if and only if its $\mathfrak{p}$-adic completion is $n$-congruing,
for all maximal ideals $\mathfrak{p}$ of $A$.
As a consequence of Theorem~\ref{hasse principle} we find that there exists a unique minimal $n$-congruing ideal in $A$.
Consequently, a rigorous statement of the Global Conjecture is:
\begin{conjecture}\label{weaker rigorous global conj} {\bf (Ravenel's Global Conjecture, precise form.)}
Let $K/\mathbb{Q}$ be a finite Galois extension with ring of integers $A$.
Then there exists some number $c\in \mathbb{N}$ such that,
for all $n\in \mathbb{N}$,
\[ \Ext_{(L^A,L^AB)}^{1,2n}(L^A,L^A)[c^{-1}] \cong A/(J_n)[c^{-1}],\]
where $J_n$ is the minimal $n$-congruing ideal of $A$.
\end{conjecture}
A stronger version of the conjecture is the following:
\begin{conjecture}\label{rigorous global conj} {\bf (Ravenel's Global Conjecture, stronger precise form.)}
Let $K/\mathbb{Q}$ be a finite Galois extension with ring of integers $A$.
Then there exists some ``correcting factor'' 
$c\in A$ such that, for all $n\in \mathbb{N}$,
\[\Ext_{(L^A,L^AB)}^{1,2n}(L^A,L^A) \cong A/((c_n)J_n),\]
where $c_n$ is some factor of $c$, and $J_n$ is the minimal $n$-congruing ideal
of $A$.
\end{conjecture}

In the present paper (see Corollary~\ref{weak form of global conj holds}) 
I prove Conjecture~\ref{weaker rigorous global conj}
in full generality (that is, for all finite Galois extensions $K/\mathbb{Q})$).
The number $c\in \mathbb{N}$ in the statement of Conjecture~\ref{weaker rigorous global conj} can be taken to be 
the ``prime-power-ramification discriminant'' $\underline{\Delta}_{K/\mathbb{Q}}$,
defined in Definition~\ref{def of prime power disc},
which is a certain divisor of two times the classical discriminant, $2\Delta_{K/\mathbb{Q}}$. 

In Corollary~\ref{strong form of global conj holds}
I also prove many cases of Conjecture~\ref{rigorous global conj}.
Specifically, I prove Conjecture~\ref{rigorous global conj}, with correcting factor $c=1$ (i.e.,
no correcting factor is necessary, unlike the $K=\mathbb{Q}$ case!), for all
finite Galois extensions $K/\mathbb{Q}$ with the property that,
for all primes $\mathfrak{p}$ of the ring of integers $A$ of $K$,
the number $\log_p(\frac{e_{\mathfrak{p}}}{p-1})$ is not an integer, where 
$p$ is the prime of $\mathbb{Z}$ under $\mathfrak{p}$, and
$e_{\mathfrak{p}}$ is the ramification degree of $\mathfrak{p}$. (In these cases the ``prime-power-ramification discriminant'' $\underline{\Delta}_{K/\mathbb{Q}}$
is $1$.)
This includes, for example, all Galois extensions $K/\mathbb{Q}$ of odd prime degree in which $2$ ramifies.
\item Finally, 
as for Remark~\ref{ravenel remark on dedekind zeta}, Ravenel is indeed
correct to point out that $\zeta_K(-m) = 0$ for all positive integers $m$
and all non-totally-real number fields $K$. It is also not the case that
the denominator of $\zeta_K(1-m)$ coincides with the order 
of $\Ext_{(L^A,L^AB)}^{1,2m}(L^A,L^A)$, even up to multiplication by a constant correcting factor, for totally real number fields $K/\mathbb{Q}$ with ring of integers $A$ (specifically, I checked this by explicit computation for several totally real quadratic extensions $K$ of $\mathbb{Q}$). 

However, there is more to say on the subject of Remark~\ref{ravenel remark on dedekind zeta}. In Definition~\ref{def of unramified straightening transform} I define a certain ``unramified straightening transform'' $\mathbb{S}$ on Euler products, with the following desirable properties:
\begin{enumerate}
\item Proposition~\ref{riemann special value denominators}: $\mathbb{S}(\zeta(s))(2n)$ is the denominator of $\zeta(1-2n)$ for all positive integers $n$.
\item Theorem~\ref{main thm for unramified straightening}:
Let $K/\mathbb{Q}$ be a finite Galois extension with ring of integers $A$.
Let $\chi_{2\Delta_{K/\mathbb{Q}}}$ be the trivial Gro\ss{e}ncharakter of $K$ of conductor equal to $2\Delta_{K/\mathbb{Q}}$, two times the classical discriminant
of $K/\mathbb{Q}$, and let
$L(s,\chi_{2\Delta_{K/\mathbb{Q}}})$ be its associated Hecke $L$-function.

Then, for all positive $n\in\mathbb{N}$, the following numbers are all equal:
\begin{itemize}
\item the order of the abelian group $\Ext^{1,2n}_{(L^A,L^AB)}(L^A,L^A)[(2\Delta_{K/\mathbb{Q}})^{-1}]$,
\item the order of the abelian group $H^{1,2n}_{fl}(\mathcal{M}_{fmA}; \mathcal{O})[(2\Delta_{K/\mathbb{Q}})^{-1}]$,
\item the order of the abelian group $A[(2\Delta_{K/\mathbb{Q}})^{-1}]/J_{n}$, where $J_n$ is the minimal $n$-congruing ideal in $A[(2\Delta_{K/\mathbb{Q}})^{-1}]$, and
\item the number $\mathbb{S}(L(s,\chi_{2\Delta_{K/\mathbb{Q}}}))(n)$.
\end{itemize}
\item The definition of $\mathbb{S}$ is very, very simple. 
\end{enumerate}

In Definition~\ref{def of g-d transform} I define a certain ``Galois-Dedekind straightening transform'' $\mathbb{S}_{GD}$ on Euler products.
The Galois-Dedekind straightening transform has a more complicated and less natural-seeming definition than the unramified straightening
transform, but it has following desirable properties:
\begin{enumerate}
\item Example~\ref{g-d transform of riemann zeta}: $\mathbb{S}_{GD}(\zeta(s))(2n)$ is the denominator of $\zeta(1-2n)$ for all positive integers $n$.
\item Theorem~\ref{main thm for g-d straightening}:
Let $K/\mathbb{Q}$ be a finite Galois extension with ring of integers $A$. 
Let $\chi_{\underline{\Delta}_{K/\mathbb{Q}}}$ be the trivial Gro\ss{e}ncharakter of $K$ of conductor equal to the prime-power-ramification discriminant $\underline{\Delta}_{K/\mathbb{Q}}$, and let 
$L(s,\chi_{\underline{\Delta}_{K/\mathbb{Q}}})$ be its associated Hecke $L$-function.

Then, for all positive $n\in\mathbb{N}$, the following numbers are all equal:
\begin{itemize}
\item the order of the abelian group $\Ext^{1,2n}_{(L^A,L^AB)}(L^A,L^A)[{\underline{\Delta}_{K/\mathbb{Q}}}^{-1}]$,
\item the order of the abelian group $H^{1,2n}_{fl}(\mathcal{M}_{fmA}; \mathcal{O})[{\underline{\Delta}_{K/\mathbb{Q}}}^{-1}]$,
\item the order of the abelian group $A[{\underline{\Delta}_{K/\mathbb{Q}}}^{-1}]/J_{n}$, where $J_n$ is the minimal $n$-congruing ideal in $A[{\underline{\Delta}_{K/\mathbb{Q}}}^{-1}]$, and
\item the number $\mathbb{S}_{GD}(L(s,\chi_{\underline{\Delta}_{K/\mathbb{Q}}}))(n)$.
\end{itemize}
\item Corollary~\ref{global corollary on dedekind zeta}: 
Let $K/\mathbb{Q}$ be a finite field extension with ring of integers $A$. Suppose $K/\mathbb{Q}$ is Galois
and suppose that the prime-power-ramification discriminant $\underline{\Delta}_{K/\mathbb{Q}}$ is equal to one.

Then, for all positive $n\in\mathbb{N}$, the following numbers are all equal:
\begin{itemize}
\item the order of the abelian group $\Ext^{1,2n}_{(L^A,L^AB)}(L^A,L^A)$,
\item the order of the abelian group $H^{1,2n}_{fl}(\mathcal{M}_{fmA}; \mathcal{O})$,
\item the order of the abelian group $A/J_{n}$, where $J_n$ is the minimal $n$-congruing ideal in $A$, and
\item the number $\mathbb{S}_{GD}(\zeta_K(s))(n)$, where $\zeta_K(s)$ is the Dedekind $\zeta$-function of $K$.
\end{itemize}
\item Proposition~\ref{inverse g-d s-transform really is inverse}: there exists an ``inverse'' transform $\mathbb{S}_{GD}^{-1}$ such that
$\mathbb{S}_{GD}^{-1}(\mathbb{S}_{GD}(\zeta_K(s))) = \zeta_K(s)$ for every finite Galois extension $K/\mathbb{Q}$.
\end{enumerate}
\end{itemize}

Finally, as a consequence of these facts about $\mathbb{S}_{GD}$, I prove Theorem~\ref{main equivalence thm after localization},
which states that, if $K_1/\mathbb{Q}$ and $K_2/\mathbb{Q}$ are finite Galois extensions with ring of integers $A_1$ and $A_2$, respectively,
and $m$ is any integer which is divisible by both $\underline{\Delta}_{K_1/\mathbb{Q}}$ and $\underline{\Delta}_{K_2/\mathbb{Q}}$,
then the following statements are all equivalent:
\begin{enumerate}
\item 
The Hecke $L$-function of the trivial Gro{\ss}encharakter on $K_1$ with conductor $m$ is equal to the 
the Hecke $L$-function of the trivial Gro{\ss}encharakter on $K_2$ with conductor $m$.
\item 
For all positive $n\in\mathbb{N}$, 
the order of the abelian group
\[ \Ext^{1,2n}_{(L^{A_1},L^{A_1}B)}(L^{A_1},L^{A_1})[m^{-1}]\]
is equal to the order of the abelian group
\[ \Ext^{1,2n}_{(L^{A_2},L^{A_2}B)}(L^{A_2},L^{A_2})[m^{-1}].\]
\item 
For all positive $n\in\mathbb{N}$, 
the order of the abelian group
\[ H^{1,2n}_{fl}(\mathcal{M}_{fmA_1}; \mathcal{O})[m^{-1}]\]
is equal to the order of the abelian group
\[ H^{1,2n}_{fl}(\mathcal{M}_{fmA_2}; \mathcal{O})[m^{-1}].\]
\item 
For all positive $n\in\mathbb{N}$, 
the abelian group
\[ \Ext^{1,2n}_{(L^{A_1},L^{A_1}B)}(L^{A_1},L^{A_1})[m^{-1}]\]
is isomorphic to the abelian group
\[ \Ext^{1,2n}_{(L^{A_2},L^{A_2}B)}(L^{A_2},L^{A_2})[m^{-1}].\]
\item 
For all positive $n\in\mathbb{N}$, 
the abelian group
\[ H^{1,2n}_{fl}(\mathcal{M}_{fmA_1}; \mathcal{O})[m^{-1}]\]
is isomorphic to the abelian group
\[ H^{1,2n}_{fl}(\mathcal{M}_{fmA_2}; \mathcal{O})[m^{-1}].\]
\item 
For all $n\in\mathbb{N}$, the order of the abelian group $A_1/(J_{n,1})[m^{-1}]$
is equal to the order of the abelian group $A_2/(J_{n,2})[m^{-1}]$
where $J_{n,1}$ is the minimal $n$-congruing ideal of $A_1[m^{-1}]$.
and $J_{n,2}$ is the minimal $n$-congruing ideal of $A_2[m^{-1}]$.
\item 
For all $n\in\mathbb{N}$, the abelian group $A_1/(J_{n,1})[m^{-1}]$
is isomorphic to the abelian group $A_2/(J_{n,2})[m^{-1}]$
where $J_{n,2}$ is the minimal $n$-congruing ideal of $A_1[m^{-1}]$.
and $J_{n,2}$ is the minimal $n$-congruing ideal of $A_2[m^{-1}]$.
\end{enumerate}

Corollary~\ref{main equivalence thm when ppr disc is trivial} follows as a consequence:
if $K_1/\mathbb{Q}$ and $K_2/\mathbb{Q}$ are finite Galois extensions with rings of integers $A_1$ and $A_2$, respectively,
and the prime-power-ramification discriminants
$\underline{\Delta}_{K_1/\mathbb{Q}}$ and $\underline{\Delta}_{K_2/\mathbb{Q}}$ are both equal to $1$,
then the following are equivalent:
\begin{enumerate}
\item  The Dedekind $\zeta$-functions of $K_1$ and of $K_2$ are equal.
\item  For all positive $n\in\mathbb{N}$, 
the order of the abelian group
\[ \Ext^{1,2n}_{(L^{A_1},L^{A_1}B)}(L^{A_1},L^{A_1})\]
is equal to the order of the abelian group
\[ \Ext^{1,2n}_{(L^{A_2},L^{A_2}B)}(L^{A_2},L^{A_2}).\]
\item  For all positive $n\in\mathbb{N}$, 
the order of the abelian group
\[ H^{1,2n}_{fl}(\mathcal{M}_{fmA_1}; \mathcal{O})\]
is equal to the order of the abelian group
\[ H^{1,2n}_{fl}(\mathcal{M}_{fmA_2}; \mathcal{O}).\]
\item  For all positive $n\in\mathbb{N}$, 
the abelian group
\[ \Ext^{1,2n}_{(L^{A_1},L^{A_1}B)}(L^{A_1},L^{A_1})\]
is isomorphic to the abelian group
\[ \Ext^{1,2n}_{(L^{A_2},L^{A_2}B)}(L^{A_2},L^{A_2}).\]
\item  For all positive $n\in\mathbb{N}$, 
the abelian group
\[ H^{1,2n}_{fl}(\mathcal{M}_{fmA_1}; \mathcal{O})\]
is isomorphic to the abelian group
\[ H^{1,2n}_{fl}(\mathcal{M}_{fmA_2}; \mathcal{O}).\]
\item  For all $n\in\mathbb{N}$, the order of the abelian group $A_1/(J_{n,1})$
is equal to the order of the abelian group $A_2/(J_{n,2})$
where $J_{n,1}$ is the minimal $n$-congruing ideal of $A_1$.
and $J_{n,2}$ is the minimal $n$-congruing ideal of $A_2$.
\item For all $n\in\mathbb{N}$, the abelian group $A_1/(J_{n,1})$
is isomorphic to the abelian group $A_2/(J_{n,2})$
where $J_{n,2}$ is the minimal $n$-congruing ideal of $A_1$.
and $J_{n,2}$ is the minimal $n$-congruing ideal of $A_2$.
\end{enumerate}

I see Theorem~\ref{main equivalence thm after localization} and Corollary~\ref{main equivalence thm when ppr disc is trivial} 
as a satisfying resolution to the state of affairs observed by Ravenel in Remark~\ref{ravenel remark on dedekind zeta}: 
for finite Galois extensions of $\mathbb{Q}$,
when the prime-power-ramification discriminant is trivial (i.e., equal to $1$), the orders of the $\Ext^1$-groups 
are determined by the Dedekind $\zeta$-function of $K$ (not just their special values at negative integers!), and vice versa. 

It is worth mentioning that these theorems about $\mathbb{S}(\zeta_K(s))$ do not require any assumptions about
$K/\mathbb{Q}$ being an abelian extension. Computations of special values $\zeta_K(-n)$ for nonabelian $K/\mathbb{Q}$ are often prohibitively hard,
since one cannot factor $\zeta_K$ as a product of Dirichlet $L$-functions, so instead one must use the Artin $L$-functions of the irreducible representations of 
$\Gal(K/\mathbb{Q})$, whose special values are far more mysterious than Dirichlet $L$-function special values. Our substitute $\mathbb{S}(\zeta_K(s))(n)$ 
for the (denominators of) the special values $\zeta_K(-n)$ is, by contrast, highly computable. In Example~\ref{examples of ppr discriminants}
I give an example of a non-abelian Galois extension of $\mathbb{Q}$ with trivial prime-power-ramification discriminant, i.e., an example of
an non-abelian extension of $\mathbb{Q}$ to which Corollary~\ref{main equivalence thm when ppr disc is trivial} applies.

It is worth mentioning that a tremendous amount of work in number theory has gone into the number-theoretic 
properties and applications of the deformation spaces of formal $A$-modules, i.e., the formal neighborhoods of the various points in the
moduli stack of formal $A$-modules. For example, Lubin and Tate's proof, in~\cite{MR0172878}, of the $p$-adic version of Kronecker's ``jugendtraum''
uses the deformation space of a formal $\mathcal{O}_K$-module of $\mathcal{O}_K$-height $1$ to construct the abelian closure of a finite extension $K/\mathbb{Q}_p$,
and Carayol's non-abelian Lubin-Tate theory, as in~\cite{MR1044827}, lays out a program to use 
the deformation spaces of height $>1$ formal $\mathcal{O}_K$-modules 
to produce local Langlands and Jacquet-Langlands correspondences; this program was ultimately successful, as described in~\cite{MR1876802},
where, for a finite extension $K/\mathbb{Q}$, 
a ``globalization'' process of passing from local data ($p$-adic representations and $p$-local Euler factors, described
by appropriate cohomology groups of deformation spaces of formal $(\mathcal{O}_K)^{\hat{}}_{\mathfrak{p}}$-modules, for the various maximal
ideals $\mathfrak{p}$ of $\mathcal{O}_K$) to the global data (complex representations and their $L$-functions) was accomplished
by using certain Shimura varieties with the property that the formal neighborhoods of the various points are
the relevant deformation spaces. The results of the present paper 
suggest that the moduli stack $\mathcal{M}_{fm\mathcal{O}_K}$ of formal $\mathcal{O}_K$-modules itself
can play a role very much like these Shimura varieties:
$\mathcal{M}_{fm\mathcal{O}_K}$
is also a ``globalization'' of the various deformation spaces of formal $(\mathcal{O}_K)^{\hat{}}_{\mathfrak{p}}$-modules, with the
property that its flat cohomology captures $L$-function-theoretic data about the number ring $\mathcal{O}_K$.
The stack $\mathcal{M}_{fm\mathcal{O}_K}$ is less amenable to uniformization techniques (as in~\cite{MR1141456}) which are used for Shimura varieties, 
but $\mathcal{M}_{fm\mathcal{O}_K}$ has the advantage of an explicit presentation by a Hopf algebroid, with important connections to homotopy theory, already mentioned above.

I am grateful to 
the topology groups at University of Rochester and University of Chicago for their patience with me while I prattled on and on in their topology seminars about the contents of this paper. I am especially grateful to D. Ravenel for many useful conversations about formal modules and their moduli theory.

\section{Review of formal modules, their moduli theory, and common notations.}

In this section I give a brief review of some definitions and some results which I regard as classical in the theory of formal modules.
An excellent reference is~\cite{MR2987372}.

\begin{definition} If $A$ is a commutative ring and $R$ is a commutative $A$-algebra, then a {\em (one-dimensional) formal $A$-module over $R$}
is a formal group law $F$ over $R$ equipped with a ring homomorphism $\rho: A \rightarrow \End(F)$
such that, for all $a\in A$, the power series $\rho(a) \in \End(F) \subseteq R[[X]]$ is congruent to $aX$ modulo $X^2$.
A {\em homomorphism of formal $A$-modules} is (as one would guess) just a homomorphism of formal group laws which commutes with
the action of $A$. A {\em strict isomorphism of formal $A$-modules} is an isomorphism (i.e., a homomorphism with an inverse) of formal $A$-modules
which is strict as an isomorphism of the underlying formal group laws, i.e., as a power series, the isomorphism is congruent to $X$ modulo $X^2$.

All formal modules in this paper will be assumed one-dimensional.
\end{definition}
Roughly speaking, a formal $A$-module is a ``formal group law with complex multiplication by $A$'' (this perspective was taken, for example,
in~\cite{MR0172878}).

\begin{definition}\label{def of A-typical fm}
Suppose $A$ is a local commutative ring with finite residue field $\mathbb{F}_q$ and uniformizer $\pi$.
Suppose $R$ is a commutative $A$-algebra which is $\pi$-torsion-free, i.e., multiplication by $\pi$ is injective on $R$.
Suppose further that $R$ is characteristic zero.
We say that a formal $A$-module $F$ over $R$ is {\em $A$-typical} if the logarithm $\log_F$ of the underlying formal group law of $F$ 
has the form
 \[ \log_F(X) = X + \lambda_1 X^q + \lambda_2 X^{q^2} + \lambda_3 X^{q^3} + \dots .\]
\end{definition}
Definition~\ref{def of A-typical fm} admits a natural, canonical extension to a notion of $A$-typicality for arbitrary formal $A$-modules,
not just those defined over certain (characteristic zero, etc.) $A$-algebras; see chapter 21 of~\cite{MR2987372} for this.

See chapter 21 of~\cite{MR2987372} for the following results, which are phrased below as one ``omnibus'' theorem:
\begin{theorem} {\bf Lazard-type theorems for formal $A$-modules.}
\begin{itemize}
\item {\bf (Global case.)} Let $A$ be the ring of integers in a finite extension $K/\mathbb{Q}$. 
\begin{itemize}
\item Then  
there exists a ``formal $A$-module Lazard ring'' $L^A$ and a ``classifying ring for strict formal $A$-module isomorphisms'' $L^AB$,
having the property that, for any commutative $A$-algebra $R$, there is a bijection between $A$-algebra morphisms $L^A\rightarrow R$
and formal $A$-modules over $R$, and a bijection between $A$-algebra morphisms $L^AB\rightarrow R$ and strict isomorphisms of formal $A$-modules over $R$.
These bijections are natural in $R$.
\item The rings $L^A$ and $L^AB$ assemble to form a Hopf algebroid $(L^A,L^AB)$, with left unit, right unit, augmentation, conjugation, and coproduct maps being the ring homomorphisms that express the operations taking the domain of a strict isomorphism,
taking the codomain of a strict isomorphism, taking the identity strict isomorphism on a formal $A$-module, taking the inverse of a strict isomorphism,
and composing a composable pair of strict isomorphisms, respectively.
\item If the class number of $A$ is one, then $L^A$ is isomorphic to a polynomial ring $L^A \cong A[S_1^A, S_2^A, S_3^A, \dots]$, and $L^AB$ is isomorphic to a polynomial
ring $L^AB \cong L^A[b_1^A, b_2^A, b_3^A, \dots ]$. If the class number of $A$ is not one, then $L^A$ and $L^AB$ are subalgebras of polynomial algebras;
see~\cite{cmah2} for explicit computations.
\item Given a commutative $A$-algebra $R$, a unit $r\in R^{\times}$, and a formal $A$-module $F$ over $R$ with logarithm $\log_F(X)$, the formal group law
with logarithm $r^{-1}\log_F(rX)$ also has the natural structure of a formal $A$-module. This establishes an action of the multiplicative group scheme $\mathbb{G}_m$
on the functor $\hom_{comm. A-alg}(L^A, -)$, hence a grading on $L^A$. {\em By convention, we double all grading degrees of gradings arising from $\mathbb{G}_m$-actions, so that all our graded rings are concentrated in even degrees and so there is no question of whether we are using the graded-commutativity sign convention.} With this convention, the generator $S_n^A$ is in grading degree $2n$. By a similar construction, $L^AB$ also has a natural grading, with $b_n^A$ in grading degree $2n$.
\end{itemize}
\item {\bf (Local case.)} Now let $A$ instead be the ring of integers in a finite extension $K/\mathbb{Q}_p$.
\begin{itemize}
\item Then  
there exists a ``formal $A$-module Lazard ring'' $L^A$ and a ``classifying ring for strict formal $A$-module isomorphisms'' $L^AB$,
with the same universal properties as in the global case. 
\item There also exists an ``$A$-typical formal $A$-module Lazard ring'' $V^A$ and a ``$A$-typical classifying ring for strict formal $A$-module isomorphisms'' $V^AT$,
having the property that, for any commutative $A$-algebra $R$, there is a bijection between $A$-algebra morphisms $V^A\rightarrow R$
and $A$-typical formal $A$-modules over $R$, and a bijection between $A$-algebra morphisms $V^AT\rightarrow R$ and strict isomorphisms of $A$-typical 
formal $A$-modules over $R$.
These bijections are natural in $R$.
\item The rings $L^A$ and $L^AB$ assemble to form a Hopf algebroid $(L^A,L^AB)$, just as in the global case. 
The rings $V^A$ and $V^AT$ also assemble to form a Hopf algebroid $(V^A,V^AT)$, for the same reasons.
\item The ring $L^A$ is isomorphic to a polynomial ring $L^A \cong A[S_1^A, S_2^A, S_3^A, \dots]$, and $L^AB$ is isomorphic to a polynomial
ring $L^AB \cong L^A[b_1^A, b_2^A, b_3^A, \dots ]$ (there is no class number condition here, since the class group of $A$ is automatically trivial!). 
These rings are graded just as in the global case.
\item The ring $V^A$ is isomorphic to a polynomial ring $V^A \cong A[v_1^A, v_2^A, v_3^A, \dots]$, and $V^AT$ is isomorphic to a polynomial
ring $V^AT \cong V^A[t_1^A, t_2^A, t_3^A, \dots ]$.
These rings have a natural grading by the same construction as in the global case, and the grading degrees of $v_n^A$ and of $t_n^A$ are each equal to $2(q^n-1)$,
where $q$ is the cardinality of the residue field of $A$.
\end{itemize}
\end{itemize}
\end{theorem}

This paper is about $\Ext$ groups in certain categories of comodules
over certain Hopf algebroids, specifically the Hopf algebroids classifying
formal modules over various number rings.
I had better say a little bit about what this means and why this is worth doing:
\begin{conventions}\label{running conventions}
Suppose that $(A,\Gamma)$ is a Hopf algebroid with $\Gamma$ flat over $A$.
Suppose further that $M,N$ are left $\Gamma$-comodules.
Then I will write 
\[ \Ext_{(A,\Gamma)}^{*}(M,N)\]
for the relative $\Ext$ groups in the category of left $\Gamma$-comodules,
relative to the allowable class generated by the extended $\Gamma$-comodules.
See Appendix 1 of Ravenel's book~\cite{MR860042} for this very standard construction.
It has the following desirable properties:
\begin{itemize}
\item The two-sided cobar complex of $(A,\Gamma)$ with coefficients in $M$ and $N$
computes $\Ext_{(A,\Gamma)}^*(M,N)$ if $M$ is projective as an $A$-module (see A1.2.12 of~\cite{MR860042}).
\item When $\Gamma$ is smooth (respectively, formally smooth) over $A$, then 
the stack associated to the groupoid scheme $(\Spec A, \Spec \Gamma)$ is 
an Artin (respectively, ``formally Artin'') stack $\mathcal{X}$, and the category of left $\Gamma$-comodules is equivalent to the category of quasicoherent modules over the structure sheaf $\mathcal{O}$ of the fpqc site on $\mathcal{X}$. This equivalence preserves cohomology, in the following sense: 
write $\tilde{M}$ for the quasicoherent $\mathcal{O}$-module associated to a
left $\Gamma$-module $M$, and write $\overline{\mathcal{F}}$ for the left
$\Gamma$-comodule associated to a quasicoherent $\mathcal{O}$-module $\mathcal{F}$.
Then for all nonnegative integers $n$, we have isomorphisms
\begin{align*} 
 \Ext_{(A,\Gamma)}^n(A,M) &\cong H_{fl}^n(\mathcal{X}; \tilde{M}) \mbox{\ \ and} \\
 \Ext_{(A,\Gamma)}^n(A,\overline{\mathcal{F}}) &\cong H_{fl}^n(\mathcal{X}; \mathcal{F}),\end{align*}
natural in $\mathcal{F}$ and in $M$. See~\cite{pribblethesis} or~\cite{smithlingthesis} for these results.
\end{itemize}

Now suppose further that $(A,\Gamma)$ is a {\em graded} Hopf algebroid with $\Gamma$ flat over $A$. (I insist here that $A,\Gamma$ both be commutative rings equipped with a
grading, not merely graded-commutative rings!)
Suppose further that $M,N$ are graded left $\Gamma$-comodules.
Then I make the definition 
\[ \Ext_{(A,\Gamma)}^{s,t}(M,N)\coloneqq  \Ext_{gr.\ (A,\Gamma)}^{s}(\Sigma^t M, N) ,\]
that is, I will write $\Ext_{(A,\Gamma)}^{s,t}(M,N)$
for the relative $\Ext$-group $\Ext_{gr.\ (A,\Gamma)}^{s}(\Sigma^t M, N)$
in the category of graded left $\Gamma$-comodules,
relative to the allowable class generated by the extended graded $\Gamma$-comodules;
here (as is usual in topology) I am writing $\Sigma^t M$ for the $\Gamma$-comodule $M$ with all of its grading degrees increased by $t$.
See Appendix 1 of Ravenel's book~\cite{MR860042} for this very standard construction.
It has the following desirable properties:
\begin{itemize}
\item The $\Ext$-group $\Ext^s_{(A,\Gamma)}(M, N)$, ignoring the gradings,
splits as a direct sum
\[ \Ext^s_{(A,\Gamma)}(M, N) \cong \oplus_{t\in\mathbb{Z}} \Ext^{s,t}_{(A,\Gamma)}(M, N).\]
Consequently the two-sided cobar complex of $(A,\Gamma)$ with coefficients in $M$ and $N$
still computes $\Ext_{(A,\Gamma)}^*(M,N)$ if $M$ is projective as an $A$-module (again, see A1.2.12 of~\cite{MR860042}).
\item When $\Gamma$ is smooth (respectively, formally smooth) over $A$, then 
the Artin (respectively, ``formally Artin'') stack $\mathcal{X}$ inherits an action of
the multiplicative group scheme $\mathbb{G}_m$, 
and the category of graded left $\Gamma$-comodules is equivalent to the category of 
$\mathbb{G}_m$-equivariant quasicoherent modules over the structure sheaf $\mathcal{O}_X$ of the fpqc site on $\mathcal{X}$. This equivalence preserves cohomology, in the following sense: 
write $\tilde{M}$ for the $\mathbb{G}_m$-equivariant quasicoherent $\mathcal{O}_X$-module associated to a
graded left $\Gamma$-module $M$, and write $\overline{\mathcal{F}}$ for the graded
left $\Gamma$-comodule associated to a $\mathbb{G}_m$-equivariant quasicoherent $\mathcal{O}_X$-module $\mathcal{F}$.
Then for all nonnegative integers $n$, we have isomorphisms
\begin{align} 
\label{gm-equivariant stack coh 1} \Ext_{(A,\Gamma)}^{s,t}(A,M) &\cong H_{fl,\mathbb{G}_m}^s(\mathcal{X}; \tilde{M}\otimes_{\mathcal{O}_{X}} \mathcal{O}_X(-1)^{\otimes_{\mathcal{O}_X} t}) \mbox{\ \ and} \\
\label{gm-equivariant stack coh 2} \Ext_{(A,\Gamma)}^{s,t}(A,\overline{\mathcal{F}}) &\cong H_{fl,\mathbb{G}_m}^s(\mathcal{X}; \mathcal{F}\otimes_{\mathcal{O}_{X}} \mathcal{O}_X(-1)^{\otimes_{\mathcal{O}_X} t}),\end{align}
natural in $\mathcal{F}$ and in $M$.
Here I am writing $H_{fl,\mathcal{G}_m}^{*}$ for $\mathbb{G}_m$-equivariant cohomology of $\mathbb{G}_m$-equivariant quasicoherent modules over the structure sheaf $\mathcal{O}_X$ of the fpqc site on $X$, and I am writing $\mathcal{O}_X(-1)$ for the 
quasicoherent ``twist'' module associated to the graded left $\Gamma$-comodule
$\Sigma^{-1} A$.

As far as I know, there does not seem to be a standard notation for stack cohomology equipped with the additional grading one gets from a $\mathbb{G}_m$-action on the stack, as in~\ref{gm-equivariant stack coh 1} and~\ref{gm-equivariant stack coh 2}. I will write 
$H_{fl}^{s,t}(\mathcal{X}; \mathcal{F})$ as shorthand for the cohomology group in~\ref{gm-equivariant stack coh 2}.
\item When $E$ is a ring spectrum whose ring of homotopy groups $E_*$ is commutative
and whose ring of stable co-operations $E_*E$
is commutative and flat over $E_*$, then $(E_*,E_*E)$ is a graded Hopf algebroid,
and for any spectrum $X$, the bigraded abelian group
\[ \Ext_{(E_*,E_*E)}^{*,*}(E_*,E_*(X))\]
is the $E_2$-term of the $E$-Adams spectral sequence
converging to the homotopy groups $\pi_*(\hat{X}_E)$
of the $E$-nilpotent completion of $X$. See chapter 2 of~\cite{MR860042} for this material. 

Specifically, when $E = MU$, the complex cobordism spectrum,
then the Hopf algebroid $(MU_*,MU_*MU)$ is isomorphic to Lazard's Hopf algebroid
$(L^{\mathbb{Z}},L^{\mathbb{Z}}B)$ classifying one-dimensional formal $\mathbb{Z}$-modules,
i.e., one-dimensional formal group laws. When $p$ is any prime and $E=BP$,
the $p$-primary Brown-Peterson spectrum,
then the Hopf algebroid $(BP_*,BP_*BP)$ is isomorphic to Lazard's Hopf algebroid
$(V^{\mathbb{Z}_{(p)}},V^{\mathbb{Z}_{(p)}}T)$ classifying $p$-typical one-dimensional formal $\mathbb{Z}_{(p)}$-modules,
i.e., $p$-typical one-dimensional formal group laws over commutative $\mathbb{Z}_{(p)}$-algebras. These are both theorems of Quillen: see~\cite{MR0253350}.

Consequently (and using Bousfield's theorems identifying $MU$-nilpotent and $BP$-nilpotent completions of connective spectra: see chapter 2 of~\cite{MR860042}), 
writing $\mathcal{M}_{fg}$ for the moduli stack of one-dimensional formal groups over $\Spec \mathbb{Z}$, 
\[ \Ext_{(L^{\mathbb{Z}},L^{\mathbb{Z}}B)}^{*,*}(L^{\mathbb{Z}},L^{\mathbb{Z}})\cong H^{*,*}_{fl}(\mathcal{M}_{fg}, \mathcal{O})\]
is the input for the 
Adams-Novikov spectral sequence converging to the stable homotopy
groups of spheres $\pi_*(S)$,
and 
\[ \Ext_{(V^{\mathbb{Z}_{(p)}},V^{\mathbb{Z}_{(p)}}T)}^{*,*}(V^{\mathbb{Z}_{(p)}},V^{\mathbb{Z}_{(p)}})\cong H^{*,*}_{fl}(\mathcal{M}_{fg}\times_{\Spec \mathbb{Z}} \Spec \mathbb{Z}_{(p)}; \mathcal{O})\]
is the input for the 
$p$-primary Adams-Novikov spectral sequence converging to the $p$-local stable homotopy
groups of spheres $\pi_*(S)_{(p)}$.
\end{itemize}
\end{conventions}

\section{Proof of (many cases of) Ravenel's Local Conjecture.}

\begin{lemma}\label{counting valuations lemma}
Let $p$ be a prime number, let $i,j$ be nonnegative integers such that $j\leq i$, and let $\alpha$ a positive integer prime to $p$.
\begin{itemize}
\item In the set of integers $\{ p^i +1, p^i +2, p^i +3, \dots ,\alpha p^i\}$, there exist exactly $(\alpha - 1)p^{i-j}$ integers divisible by $p^j$.
\item If $m$ is a nonnegative integer, then in 
the set of integers $\{ p^i +1+m, p^i +2+m, p^i +3+m, \dots ,\alpha p^i +m\}$, there exist at least $(\alpha - 1)p^{i-j}$ integers divisible by $p^j$.
\end{itemize}
\end{lemma}
\begin{proof}
Elementary.
\end{proof}

\begin{lemma}\label{valuations of binomial coefficients}
Let $p$ be a prime number, let $n,e$ be positive integers, and 
let $u_{p,n,e}: \{ 1, \dots ,n\} \rightarrow \mathbb{N}$
be the function
\[ u_{p,n,e}(k) = e\cdot\nu_p\left( \binom{n}{k}\right) + k .\]
Then the least value taken by the function $u_{p,n,e}$ is
achieved exactly when:
\begin{itemize}
\item $k\in \{ \frac{e}{p-1}, p\frac{e}{p-1}\}$ if $\log_p\frac{e}{p-1}$ is an integer and $\log_p\frac{e}{p-1} < \nu_p(n)$,
\item $k = p^{\nu_p(n)}$ if $\log_p\frac{e}{p-1}$ is an integer and $\log_p\frac{e}{p-1}\geq \nu_p(n)$ {\em or} 
 $\log_p\frac{e}{p-1}$ is not an integer and $\ceiling{\log_p\frac{e}{p-1}} > \nu_p(n)$,
\item and $k = p^{\ceiling{\log_p\frac{e}{p-1}}}$ if $\log_p\frac{e}{p-1}$ is not an integer and $\ceiling{\log_p\frac{e}{p-1}} \leq \nu_p(n)$.
\end{itemize}
That least value taken by $u_{p,n,e}$ is:
\begin{itemize}
\item \begin{align*} u_{p,n,e} (\frac{e}{p-1}) &= u_{p,n,e} (\frac{pe}{p-1}) &= e\left(\nu_p(n)  + \frac{1}{p-1} - \log_p \frac{e}{p-1}\right)\end{align*}
if $\log_p\frac{e}{p-1}$ is an integer and $\log_p\frac{e}{p-1} < \nu_p(n)$,
\item \begin{align*} u_{p,n,e} (\nu_p(n)) &= p^{\nu_p(n)}\end{align*}
if $\log_p\frac{e}{p-1}$ is an integer and $\log_p\frac{e}{p-1} \geq \nu_p(n)$ {\em or} 
 $\log_p\frac{e}{p-1}$ is not an integer and $\ceiling{\log_p\frac{e}{p-1}} > \nu_p(n)$,
\item and \begin{align*} u_{p,n,e} (p^{\ceiling{\log_p\frac{e}{p-1}}}) &= e\left(\nu_p(n) - \ceiling{\log_p \frac{e}{p-1}}\right) + p^{\ceiling{\log_p\frac{e}{p-1}}}\end{align*}
if $\log_p\frac{e}{p-1}$ is not an integer and $\ceiling{\log_p\frac{e}{p-1}} \leq \nu_p(n)$.
\end{itemize}
\end{lemma}
\begin{proof}
I claim that the minimum value of $u_{p,n,e}$ only occurs either at $n$ or in the first half of its domain, i.e., if $k\in \{ 1, 2, \dots , n\}$ is such that
$u_{p,n,e}(k)$ is the minimum value taken by $u_{p,n,e}$, then either $k=n$ or $k\leq \frac{n}{2}$. The argument here is trivially easy: if 
$k > \frac{n}{2}$, then:
\begin{align*} 
 u_{p,n,e}(k) 
  &= e\cdot\nu_p\left(\binom{n}{k}\right) + k \\
  &= e\cdot\nu_p\left(\binom{n}{n-k}\right) + k \\
  &> e\cdot\nu_p\left(\binom{n}{n-k}\right) + n-k \\
  &= u_{p,n,e}(n-k)\mbox{\ \ if\ } k<n.\end{align*} 
As a consequence, if $\alpha>0$ is an integer prime to $p$
and $u_{p,n,e}$ takes its minimum value at $\alpha p^i$, then either $\alpha p^i = n$ or
$\alpha p^i \leq \frac{n}{2}$, and consequently either $\alpha p^i = n$ or
\begin{align*}  (\alpha +1) p^i \leq 2\alpha p^i \leq n ,\end{align*}
hence either $\alpha p^i = n$ or 
\begin{align} \label{inequality 10} n+1-\alpha p^i \geq p^i+1,\end{align}
an inequality we will use shortly.

Now I claim that the minimum value of the function $u_{p,n,e}$
occurs at some power of $p$ (including the possibility of $p^0 = 1$).
Indeed, suppose that $u_{p,n,e}(p^i) \geq u_{p,n,e}(\alpha p^i)$,
with $\alpha>0$ an integer prime to $p$.
It will be convenient to think of the binomial coefficient $\binom{n}{k}$ as
\[ \binom{n}{k} = \frac{n}{1} \frac{n-1}{2} \frac{n-2}{3}  \dots \frac{n+1-k}{k} .\]
Then:
\begin{align*}
 u_{p,n,e}(\alpha p^i) 
  &= e\cdot \nu_p\left( \binom{n}{\alpha p^i}\right) + \alpha p^i \\
  &= e\cdot \left( \sum_{j=1}^{\alpha p^i} \nu_p\left( \frac{n+1-j}{j}\right)\right) + \alpha p^i \\
  &= e\cdot \left( \sum_{j=1}^{p^i} \nu_p\left( \frac{n+1-j}{j}\right)\right) + p^i 
      + e\cdot \left( \sum_{j=p^i+1}^{\alpha p^i} \nu_p\left( \frac{n+1-j}{j}\right)\right) + (\alpha - 1) p^i \\
  &= u_{p,n,e}(p^i) + e\cdot \left( \sum_{j=p^i+1}^{\alpha p^i} \nu_p\left( \frac{n+1-j}{j}\right)\right) + (\alpha - 1) p^i , \mbox{\ so:}\\
 (1-\alpha ) p^i &\geq e\cdot \left( \sum_{j=p^i+1}^{\alpha p^i} \nu_p\left( \frac{n+1-j}{j}\right)\right) .
\end{align*}
Now if $\alpha \neq 1$, then $(1-\alpha)p^i < 0$, and consequently
$e\cdot \left( \sum_{j=p^i+1}^{\alpha p^i} \nu_p\left( \frac{n+1-j}{j}\right)\right)$ is negative. Hence:
\begin{align}
 0 
\nonumber  &> \sum_{j=p^i+1}^{\alpha p^i} \nu_p\left( \frac{n+1-j}{j}\right) \\
\nonumber  &= \left( \sum_{k=n+1-\alpha p^i}^{n-p^i} \nu_p(k)\right) - \left( \sum_{j=p^i+1}^{\alpha p^i}\nu_p(j)\right), \\
\label{inequality 11}  &= \left( \sum_{k\in N} \nu_p(k)\right) - \left( \sum_{j\in N^{\prime}} \nu_p(j)\right),
\end{align}
where $N$ is the set of integers $a$ satisfying $n+1-\alpha p^i \leq a \leq n-p^i$,
and $N^{\prime}$ is the set of integers $a$ satisfying $p^i+1 \leq a\leq \alpha p^i$.
If $u_{p,n,e}$ does not take its minimum value at $n$, then 
by inequality~\ref{inequality 10}, the set of integers $N$ is equal to the set of integers
$\{ p^i+1+m, p^i+2+m, p^i+3+m, \dots ,\alpha p^i\}$ for some nonnegative $m$.
By Lemma~\ref{counting valuations lemma}, 
the sum of the $p$-adic valuations of the elements of $N$ is at least as large as the sum of the $p$-adic valuations of the 
elements of $N^{\prime}$, i.e.,
\[ \sum_{k\in N} \nu_p(k) \geq  \sum_{j\in N^{\prime}} \nu_p(j) ,\]
contradicting inequality~\ref{inequality 11}.
So the assumption $\alpha \neq 1$ implies a contradiction unless $u_{p,n,e}$ takes its minimum value at $n$. 
So $\alpha=1$ and hence $u_{p,n,e}$ can only take its minimum values $n$ or at powers of $p$.

Now we need to know something about $u_{p,n,e}(p^i)$, for positive integers $i$. Write $n$ as $n=\beta p^m$ where $\beta$ is an integer prime to $p$.
First, notice that, if $i\leq m$ and $1<j\leq p^i$, then $\nu_p(\beta p^m + 1 - j) = \nu_p(j-1)$.
Consequently, if $i\leq m$, then:
\begin{align*} 
 u_{p,\beta p^m,e}(p^i) 
  &= e\cdot \nu_p\left( \binom{\beta p^m}{p^i} \right) + p^i \\
  &= e\cdot \nu_p\left( \prod_{j=1}^{p^i} \frac{\beta p^m +1-j}{j} \right) + p^i \\
  &= e\cdot \sum_{j=1}^{p^i} \left( \nu_p(\beta p^m+1-j) - \nu_p(j)\right)  + p^i \\
  &= e\cdot \left( \nu_p(\beta p^m)  - \nu_p(1) + \sum_{j=2}^{p^i} \left( \nu_p(j-1) - \nu_p(j)\right)\right)  + p^i \\
  &= e\cdot \left( \nu_p(\beta p^m) - \nu_p(p^i) \right)  + p^i \\
  &= e\cdot \left( m-i \right)  + p^i .\end{align*}

On the other hand, suppose that $i>m$. Then:
\begin{align*}
 u_{p,\beta p^m,e}(p^i)
  &= e\cdot \nu_p\left( \binom{\beta p^m}{p^i}\right) + p^i \\
  &\geq p^i \\
  &> p^m \\
  &= e\cdot \left( m-m \right)  + p^m \\
  &= u_{p,\beta p^m,e}(p^m),\end{align*}
so the least value taken by $u_{p,\beta p^m,e}$ is either $u_{p,\beta p^m,e}(\beta p^m)$ or it is 
$u_{p,\beta p^m,e}(p^i)$ for some (possibly non-unique!) integer $i$ satisfying $i\leq m$. 
But
\begin{align*}
 u_{p,\beta p^m,e}(\beta p^m) & = \beta p^m\\
  & \geq p^m \\
  & = u_{p,\beta p^m,e}(p^m),\end{align*}
so the least value taken by $u_{p,n,e}$ is 
$u_{p,n,e}(p^i)$ for some (possibly non-unique!) integer $i$ satisfying $i\leq \nu_p(n)$. 

Now we use a little bit of elementary calculus!
The derivative
\[ \frac{d}{di} \left( e(\nu_p(n) - i ) + p^i\right) = -e + (\ln p) p^i \]
has a unique zero, namely, at $i = \log_p\left( \frac{e}{\ln p}\right)$, hence $u_{p,n,e}(p^i)$, as a function of a {\em real variable} $i$,
has at most one local extremum in the domain $(0,\nu_p(n))\subseteq \mathbb{R}$, namely, $i = \log_p\frac{e}{\ln p}$.
Consequently the function $u_{p,n,e}(p^i)$, as a function of an {\em integer} $i\in \{ 1, \dots ,p^n\}$, takes on its minimum value at most two times, at
\begin{equation} \label{two possible locations of minima} \floor{\log_p\frac{e}{\ln p}} \mbox{\ \ and\ at\ \ } \ceiling{\log_p\frac{e}{\ln p}},\end{equation} if $u_{p,n,e}(p^i)$ does indeed take on its minimum value two times and not just once.

Now we need to decide under what circumstances $u_{p,n,e}$ takes on its minimum value twice. We already know that, if
$u_{p,n,e}$ takes on its minimum value twice, then it takes on its minimum value at two consecutive powers of $p$, that is,
$u_{p,n,e}(p^i) = u_{p,n,e}(p^{i+1})$ for some $i < n$. Consequently:
\begin{align*}
 0 
  &= u_{p,n,e}(p^i) - u_{p,n,e}(p^{i+1}) \\ 
  &= e(\nu_p(n) - i) + p^i - e(\nu_p(n) - (i+1)) - p^{i+1} \\
  &= e + p^i(1-p) , \mbox{\ \ equivalently,} \\
 i &= \log_p\frac{e}{p-1}.
\end{align*}
So $u_{p,n,e}$ takes its minimum value twice if $\log_p\frac{e}{p-1}$ is an integer {\em and} $p^{\log_p\frac{e}{p-1}}$ and $p^{1+\log_p\frac{e}{p-1}}$ are both less than $p^{\nu_p(n)}$.
Consequently, {\em the function $u_{p,n,e}: \{ 1, 2, \dots , n\} \rightarrow \mathbb{N}$ takes on its minimum value exactly once if $\log_p\frac{e}{p-1}$ is not an integer or if $\log_p\frac{e}{p-1} \geq \nu_p(n)$, and exactly twice if $\log_p\frac{e}{p-1}$ is an integer and $\log_p\frac{e}{p-1} < \nu_p(n)$.}

In the case that $\log_p\frac{e}{p-1}$ is an integer and $\log_p\frac{e}{p-1}< \nu_p(n)$, clearly the minimum
value of $u_{p,n,e}$ occurs as $u_{p,n,e}(p^{\floor{\log_p \frac{e}{\ln p}}}) =
u_{p,n,e}(p^{1+\floor{\log_p \frac{e}{\ln p}}})$, by~\ref{two possible locations of minima},
but one can give a cleaner description which does not involve the natural logarithm of $p$:
suppose that $u_{p,n,e}(p^i) = u_{p,n,e}(p^{i+1})$.
Then:
\begin{align}
 0 
\label{equality 100}  &= u_{p,n,e}(p^i) - u_{p,n,e}(p^{i+1}) \\
\label{equality 101}  &= e(\nu_p(n) - i) + p^i - e(\nu_p(n) - i - 1) - p^{i+1} \\
\label{equality 102}  &= e + (1-p)p^i, \mbox{\ equivalently,}\\
\label{equality 103} i &= \log_p\frac{e}{p-1}.
\end{align}
So {\em when $\log_p\frac{e}{p-1}$ is an integer and $\log_p\frac{e}{p-1}< \nu_p(n)$, the minimum value of $u_{p,n,e}$ occurs exactly at $\log_p \frac{e}{p-1}$ and at $1+\log_p\frac{e}{p-1}$, and this minimum value is}
\begin{align*}
 u_{p,n,e} (\log_p\frac{e}{p-1}) &= e(\nu_p(n) - \log_p \frac{e}{p-1}) + \frac{e}{p-1},\end{align*}
as claimed.

In the case that $\log_p\frac{e}{p-1}$ is an integer and $\log_p\frac{e}{p-1}\geq \nu_p(n)$, 
the function $u_{p,n,e}(p^i)$, as a function of $i$, is monotone decreasing (note that its domain is $\{ 0, 1, \dots, \nu_p(n)\}$).
Hence, {\em when $\log_p\frac{e}{p-1}$ is an integer and $\log_p\frac{e}{p-1}\geq \nu_p(n)$, the minimum value of $u_{p,n,e}$ occurs uniquely at $p^{\nu_p(n)}$, and this minimum value is}
\begin{align*}
 u_{p,n,e} (p^{\nu_p(n)}) &= p^{\nu_p(n)},\end{align*}
as claimed.

In the case that $\log_p\frac{e}{p-1}$ is not an integer,
one can ask whether the minimum value of 
$u_{p,n,e}$ occurs as $u_{p,n,e}(p^{\floor{\log_p \frac{e}{\ln p}}})$ or as
$u_{p,n,e}(p^{\ceiling{\log_p \frac{e}{\ln p}}})$ (it must be one of the other, 
by~\ref{two possible locations of minima}, unless both $p^{\floor{\log_p \frac{e}{\ln p}}}$ and $p^{\ceiling{\log_p \frac{e}{\ln p}}}$ are outside the domain of $u_{p,n,e}$).
It is simpler and cleaner (avoiding formulas involving $\ln p$), however, to 
simply check under what conditions on $i$ it is true that
$u_{p,n,e}(p^i) < u_{p,n,e}(p^{i+1})$.
By the same line of argument used in the equalities~\ref{equality 100},~\ref{equality 101},~\ref{equality 102}, and~\ref{equality 103},
the inequality $u_{p,n,e}(p^i) < u_{p,n,e}(p^{i+1})$
holds if and only if $i> \log_p\frac{e}{p-1}$. 
Consequently, {\em if $\log_p\frac{e}{p-1}$ is not an integer and $\ceiling{\log_p\frac{e}{p-1}} \leq \nu_p(n)$,
, then the minimum value of
$u_{p,n,e}$ occurs uniquely at $p^{\ceiling{\log_p\frac{e}{p-1}}}$, and this minimum value
is:}
\begin{align*}
 u_{p,n,e} (\ceiling{\log_p\frac{e}{p-1}}) &= e(\nu_p(n) - \ceiling{\log_p \frac{e}{p-1}}) + p^{\ceiling{\log_p\frac{e}{p-1}}},\end{align*}
as claimed.

Finally, the last case is the one in which  $\log_p\frac{e}{p-1}$ is not an integer but $\ceiling{\log_p\frac{e}{p-1}} > \nu_p(n)$.
Then it is again true that 
the function $u_{p,n,e}(p^i)$, as a function of $i$, is monotone decreasing.
Hence, {\em when $\log_p\frac{e}{p-1}$ is not an integer and $\ceiling{\log_p\frac{e}{p-1}}> \nu_p(n)$, the minimum value of $u_{p,n,e}$ occurs uniquely at $p^{\nu_p(n)}$, and this minimum value is again}
\begin{align*}
 u_{p,n,e} (p^{\nu_p(n)}) &= p^{\nu_p(n)},\end{align*}
as claimed.
\end{proof}

\begin{lemma}\label{valuations predicted by local conjecture}
Suppose that $K/\mathbb{Q}_p$ is a finite field extension of ramification degree $e$ and whose ring of integers $A$ has uniformizer $\pi$,
and suppose that $\log_p(\frac{e}{p-1})$ is not an integer, i.e.,
$\frac{e}{p-1}$ is not a power of $p$.
Let $n$ be a positive integer,
and let $I_n$ be the ideal in $A$ generated by all elements of the form
$a^n-1$ for elements $a\in A$ congruent to $1$ modulo $\pi$.
Then $A/I_n \cong A/\pi^i$,
where 
\[ i = 
 \left\{ \begin{array}{ll} 
  e\left(\nu_p(n) - \ceiling{\log_p \frac{e}{p-1}}\right) + p^{\ceiling{\log_p\frac{e}{p-1}}} &\mbox{\ if\ } \ceiling{\log_p\frac{e}{p-1}} \leq \nu_p(n) \\
  p^{\nu_p(n)} &\mbox{\ if\ } \ceiling{\log_p\frac{e}{p-1}} > \nu_p(n) .\end{array}\right.
\]
\end{lemma}
\begin{proof}
Let $a\in A$ be an element congruent to $1$ modulo $\pi$, i.e., 
$a \equiv 1+ a_1\pi$ modulo $\pi^2$, for some $a_1\in A$.
Then the monomial term of $a^n-1$
of smallest $\pi$-adic valuation is also the monomial term of
$\sum_{k=1}^n \binom{n}{k} \pi^k (a_1)^k$ of smallest $\pi$-adic valuation, i.e., it is
$\binom{n}{k} \pi^k (a_1)^k$ for some unique $k\in \{ 1, \dots , \nu_p(n)\}$,
by 
Lemma~\ref{valuations of binomial coefficients}
and the assumption that $\log_p\frac{e}{p-1}$ is not an integer.
Now observe that
\[ \nu_{\pi}\left( \binom{n}{k} \pi^k (a_1)^k \right) \geq u_{p,n,e}(k)\]
with equality if $a_1\in A^{\times}$,
with $u_{p,n,e}$ the function defined in Lemma~\ref{valuations of binomial coefficients},
and then the 
claim in the present lemma then follows from Lemma~\ref{valuations of binomial coefficients}.
\end{proof}

\begin{lemma}\label{forgetful preserves colimits}
Let $R$ be a commutative ring, and let $(A,\Gamma)$ be a commutative graded Hopf algebroid over $R$. Suppose that $\Gamma$ is flat over $A$.
Let $F$ be the forgetful functor $F : \gr\Comod(A,\Gamma) \rightarrow \gr\Mod(A)$ from the category of graded left 
$\Gamma$-comodules to the category of graded $A$-modules.
Let $e: \gr\Mod(A)\rightarrow \gr\Comod(A,\Gamma)$ be the extended comodule functor, i.e., $e(M) = \Gamma\otimes_A M$, with left $\Gamma$-coaction given by the map
\[ \Delta \otimes_A M: \Gamma\otimes_A M \rightarrow \Gamma\otimes_A \Gamma\otimes_A M,\]
where $\Delta: \Gamma\rightarrow \Gamma\otimes_A\Gamma$ is the coproduct map on $\Gamma$. 

Then both $e$ and $F$ preserve colimits.
\end{lemma}
\begin{proof} 
It is an easy exercise to check that $e$ is right adjoint to $\Gamma$.
Hence $F$ is a left adjoint, hence $F$ preserves colimits.
Since colimits in $\gr\Comod(A,\Gamma)$ are consequently computed in graded $A$-modules,
and since the functor $M \mapsto \Gamma\otimes_A M$ preserves colimits in graded $A$-modules,
$e$ also preserves colimits.
\end{proof}

\begin{lemma}\label{ext and compactness lemma}
Let $R$ be a commutative ring, and let $(A,\Gamma)$ be a commutative graded Hopf algebroid over $R$. Suppose that $\Gamma$ is flat over $A$.
Suppose $M$ is a graded $\Gamma$-comodule whose underlying $A$-module is finitely generated and projective,
and suppose that 
\begin{equation}\label{seq 10000} L_0 \rightarrow L_1 \rightarrow L_2 \rightarrow \dots\end{equation}
is a sequence of one-to-one graded left $\Gamma$-comodule homomorphisms.
Then, for any nonnegative integer $s$ and any integer $t$
the canonical comparison map
\begin{equation}\label{comparison map 22} \colim_i \Ext_{(A,\Gamma)}^{s,t}(M, L_i) \rightarrow \Ext_{(A,\Gamma)}^{s,t}(M, \colim_i L_i) \end{equation}
is an isomorphism of $R$-modules.
\end{lemma}
\begin{proof}
Since the underlying $A$-module of $M$ is assumed to be finitely-generated, any given morphism of 
$A$-modules from $M$ to $\colim_i L_i$ is determined by the image of finitely many $A$-module generators
of $M$, and each of these generators $g$ is in the image of the inclusion 
$L_{i_g}\hookrightarrow \colim_i L_i$ for some nonnegative integer $i_g$.
Consequently the maximum of the nonnegative integers $\{ i_g\}$, taken over a given finite set of 
generators $g$ for $M$ as a graded $A$-module, is some nonnegative integer $j$, and 
the given map $M \rightarrow \colim_i L_i$ factors through the inclusion map
$L_j \hookrightarrow \colim_i L_i$.
Consequently the canonical map
\[ \colim_i \hom_{\gr\Mod(A)}(FM, FL_i) \rightarrow \hom_{\gr\Mod(A)}(FM, \colim_i FL_i)\]
is an isomorphism of $A$-modules, where $F$ is the forgetful functor defined in Lemma~\ref{forgetful preserves colimits}.

Now we have the commutative square of morphisms of $A$-modules
\begin{equation}\label{comm sq 6} \xymatrix{
\colim_i \hom_{\gr\Comod(A,\Gamma)}(M, L_i) \ar[r]\ar[d] & \hom_{\gr\Comod(A,\Gamma)}(M, \colim_i L_i) \ar[d] \\
\colim_i \hom_{\gr\Mod(A)}(FM, FL_i) \ar[r] & \hom_{\gr\Mod(A)}(FM, \colim_i FL_i),
}\end{equation}
in which the vertical map on the right is a monomorphism by definition,
the vertical map on the left is a monomorphism because it is a directed union of 
monomorphisms and the category of graded $A$-modules satisfies Grothendieck's axiom AB5, 
and we just showed that the horizontal map on the bottom is an isomorphism, hence also a monomorphism.
Hence the top map is also a monomorphism.

I claim that the square~\ref{comm sq 6} is a pullback square in $A$-modules.
Since the maps involved are all monomorphisms, this amounts to the claim that 
$\colim_i \hom_{\gr\Comod(A,\Gamma)}(M, L_i)$ is equal to the intersection of
$\colim_i \hom_{\gr\Mod(A)}(FM, FL_i)$ with $\hom_{\gr\Comod(A,\Gamma)}(M, \colim_i L_i)$
in $\hom_{\gr\Mod(A)}(FM, \colim_i FL_i)$. It is clear that we have an inclusion
\[ \colim_i \hom_{\gr\Comod(A,\Gamma)}(M, L_i) 
    \subseteq \left( \colim_i \hom_{\gr\Mod(A)}(FM, FL_i)\right) \cap \left( \hom_{\gr\Comod(A,\Gamma)}(M, \colim_i L_i) \right) .\]
Suppose $f \in \left( \colim_i \hom_{\gr\Mod(A)}(FM, FL_i)\right) \cap \left( \hom_{\gr\Comod(A,\Gamma)}(M, \colim_i L_i) \right)$.
Then $f$ is a graded left $\Gamma$-comodule morphism $f: M \rightarrow \colim_i L_i$,
and there exists a graded $A$-module morphism
$\tilde{f}: M \rightarrow L_i$ such that $I\circ \tilde{f}  = f$,
where $I: L_i \hookrightarrow \colim_i L_i$ is the canonical inclusion.

Consequently we have the (not yet known to be commutative!) diagram:
\[\xymatrix{
 M \ar[r]^{\tilde{f}} \ar[d]^{\Psi_M} \ar@/^2pc/[rr]^{f} &
  L_i \ar[d]^{\Psi_{L_i}} \ar[r]^{I} &
  \colim_i L_i\ar[d]^{\Psi} \\
\Gamma\otimes_A M \ar[r]^{\Gamma\otimes_A \tilde{f}} \ar@/_2pc/[rr]_{\Gamma\otimes_A f} &
 \Gamma\otimes_A L_i \ar[r]^{\Gamma\otimes_A I} &
 \Gamma\otimes_A \colim_i L_i 
}\]
in which the vertical maps are the comodule structure maps, and we know the following 
four equalities:
\begin{align*}
 I\circ \tilde{f} &= f, \\
 (\Gamma\otimes_A I) \circ (\Gamma\otimes_A \tilde{f}) &= \Gamma\otimes_A f, \\
 (\Gamma\otimes_A f) \circ \Psi_M &= \Psi\circ f, \mbox{\ \ and} \\
 \Psi\circ I &= (\Gamma\otimes_A I) \circ \Psi_{L_i}.\end{align*}
We use those four equalities to get the following equalities:
\begin{align*}
 (\Gamma\otimes_A I) \circ (\Gamma\otimes_A \tilde{f})\circ \Psi_M 
  &= (\Gamma_A f) \circ \Psi_M \\
  &= \Psi \circ f \\
  &= \Psi \circ I \circ \tilde{f} \\
  &= (\Gamma\otimes_A I) \circ \Psi_{L_i} \circ \tilde{f} .\end{align*}
Now since $I$ is one-to-one and $\Gamma$ is assumed flat over $A$, 
the map $\Gamma\otimes_A I$ is injective, hence left-cancellable, so
\begin{align*}  (\Gamma\otimes_A I) \circ (\Gamma\otimes_A \tilde{f})\circ \Psi_M 
  &= (\Gamma\otimes_A I) \circ \Psi_{L_i} \circ \tilde{f} \end{align*}
implies $(\Gamma\otimes_A \tilde{f})\circ \Psi_M 
  = \Psi_{L_i} \circ \tilde{f}$, which is exactly the statement that $\tilde{f}$ is a morphism of
left $\Gamma$-comodules, and hence that $f\in \colim_i \hom_{\gr\Comod(A,\Gamma)}(M, L_i)$, and
consequently the square~\ref{comm sq 6} is a pullback square in $A$-modules, as claimed!

Consequently 
the map
\[ \colim_i \hom_{\gr\Comod(A,\Gamma)}(M, L_i) \rightarrow \hom_{\gr\Comod(A,\Gamma)}(M, \colim_i L_i) \]
is a pullback of an isomorphism, hence itself an isomorphism.
So the map~\ref{comparison map 22} is an isomorphism for $s=0$.


Now we handle $s>0$. Given a graded left $\Gamma$-comodule $N$, 
let $D_{\Gamma}(N)^{\bullet}$ be the cobar resolution of $N$ (see Appendix 1 of \cite{MR860042} for this construction and its basic properties, along with all the other 
fundamentals of homological algebra of comodules over a Hopf algebroid). Then $N\mapsto D_{\Gamma}(N)^{\bullet}$ is in fact a functor from 
graded left $\Gamma$-comodules to cochain complexes of relatively injective graded left $\Gamma$-comodules.
In particular, for each nonnegative integer $n$, the $n$-cochains functor $N\mapsto D_{\Gamma}(N)^{n}$ is a functor from graded left $\Gamma$-comodules to
relatively injective graded left $\Gamma$-comodules.
Furthermore, for each $n$, the functor $D_{\Gamma}(N)^{n}$ is a composite of the functors $e$ and $F$ 
(in fact, $D_{\Gamma}(N)^{\bullet}$, the cobar complex functor, is the cobar construction associated to
the adjunction $F\dashv e$), and consequently Lemma~\ref{forgetful preserves colimits} implies that 
the canonical map
\begin{equation}\label{comparison map 23} \colim_i D_{\Gamma}(L_i) \rightarrow D_{\Gamma}(\colim_i L_i) \end{equation}
is an isomorphism.

Since the map~\ref{comparison map 22} is a natural bijection for $s=0$
and since the map~\ref{comparison map 23} is a natural bijection as well, 
we have the isomorphism of
cochain complexes of $R$-modules
\begin{equation}\label{comm diag 11000}\xymatrix{
 0 \ar[r] \ar[d]^{\cong} & 
  \colim_i  \hom_{\gr\Comod(A,\Gamma)}(\Sigma^t M, D_{\Gamma}(L_i)^0) \ar[r] \ar[d]^{\cong} &
  \colim_i  \hom_{\gr\Comod(A,\Gamma)}(\Sigma^t M, D_{\Gamma}(L_i)^1) \ar[r] \ar[d]^{\cong} & \dots \\
 0 \ar[r]  & 
  \hom_{\gr\Comod(A,\Gamma)}(\Sigma^t M, D_{\Gamma}(\colim_i  L_i)^0) \ar[r]  &
  \hom_{\gr\Comod(A,\Gamma)}(\Sigma^t M, D_{\Gamma}(\colim_i  L_i)^1) \ar[r]  & \dots .}\end{equation}
The top row in diagram~\ref{comm diag 11000} has cohomology
$\colim_i \Ext_{(A,\Gamma)}^{*,t}(M, L_i)$ since cohomology of cochain complexes of modules over a commutative
ring commutes with filtered colimits (this is Grothendieck's axiom AB5 at work), in particular
sequential colimits.
The bottom row in diagram~\ref{comm diag 11000} has cohomology
$\Ext_{(A,\Gamma)}^{*,t}(M, \colim_i L_i)$,
and the isomorphism induced by~\ref{comm diag 11000} is the map~\ref{comparison map 22}.
\end{proof}

\begin{theorem}\label{local conjecture} {\bf ((Many cases of) Ravenel's Local Conjecture.)}
Suppose that $K/\mathbb{Q}_p$ is a finite field extension of ramification 
degree $e$ and residue degree $f$ and total degree $d=ef$, with uniformizer $\pi$, and suppose that $\log_p(\frac{e}{p-1})$ is not an integer, i.e.,
$\frac{e}{p-1}$ is not a power of $p$.
Let $A$ be the ring of integers in $K$.
Then Ravenel's Local Conjecture, Conjecture~\ref{local conj}, holds for $A$. That is,
for each $n\in\mathbb{N}$,
we have an isomorphism of $A$-modules
\[ \Ext_{(V^A,V^AT)}^{1,2(p^f-1)n}(V^A, V^A) \cong A/I_{n},\]
where $I_{n}$ is the ideal in $A$ generated by all elements of the form
$(a^n-1)$ with $a$ an element in $A$ congruent to $1$ modulo $\pi$.

Equivalently, in terms of the moduli stack $\mathcal{M}_{fmA}$ of one-dimensional formal $A$-modules over $\Spec A$, and with notation as in Conventions~\ref{running conventions}:
for each $n\in\mathbb{N}$,
we have an isomorphism of $A$-modules
\[ H^{1,2(p^f-1)n}_{fl}(\mathcal{M}_{fmA}; \mathcal{O}) \cong A/I_{n},\]
where $I_{m}$ is the ideal in $A$ generated by all elements of the form
$(a^m-1)$ with $a$ an element in $A$ congruent to $1$ modulo $\pi$.
\end{theorem}
\begin{proof}
Recall that Ravenel constructs, in~\cite{MR745362}, a ``formal $A$-module
chromatic spectral sequence':
\begin{align*} E_1^{s,t,u} \cong \Ext_{(V^A, V^AT)}^{s,u}\left(V^A, (v_t^A)^{-1}V^A/(\pi^{\infty}, (v_1^A)^{\infty}, \dots , (v_{t-1}^A)^{\infty})\right) & 
 \Rightarrow \Ext^{s+t,u}_{(V^A,V^AT)}(V^A,V^A) \\
d_r: E_r^{s,t,u} & \rightarrow E_r^{s+1-r,t+r,u} .
\end{align*}
For straightforward 
degree reasons, in the formal $A$-module chromatic spectral sequence,
the only terms that can contribute to 
$\Ext^{1,*}_{(V^A,V^AT)}(V^A,V^A)$ are those on the lines 
$t=0$ and $t=1$, and as the $t=0$ line is concentrated in tridegree
$s=0,t=0,u=0$, we have an isomorphism
\[  \Ext_{(V^A,V^AT)}^{1,u}(V^A, V^A)
    \cong\]\[ \left( \ker d_1^{0,1,u}: \left(\Ext_{(V^A,V^AT)}^{0,u}\left(V^A,(v_1^A)^{-1}V^A/\pi^{\infty}\right)\rightarrow \Ext_{(V^A,V^AT)}^{0,u}\left(V^A,(v_2^A)^{-1}V^A/\pi^{\infty},(v_1^A)^{\infty}\right)\right)\right) \]
for all $u>0$.

So we just need to compute the kernel of $d_1$ on the $\Ext$-group
\[ \Ext_{(V^A,V^AT)}^{0,u}\left(V^A,(v_1^A)^{-1}V^A/\pi^{\infty}\right) 
 \cong (V^A\Box_{V^AT} (v_1^A)^{-1}V^A/\pi^{\infty})^u.\]
By symmetry of the cotensor product and the identification of $\Ext$ in a comodule category 
with the derived functors $\Cotor$ of the cotensor product (see Appendix 1 of~\cite{MR860042}), this $\Ext$-group 
is isomorphic to 
\begin{equation}\label{side swap iso} \left( \left( (v_1^A)^{-1}V^A/\pi^{\infty}\right) \Box_{V^AT} V^A\right)^u,\end{equation}
which is slightly more convenient to work with, as it is
just the degree $u$ summand of the equalizer of the two maps
\[ 
\Psi, \left( \id_{(v_1^A)^{-1}V^A/\pi^{\infty}} \otimes \eta_L\right):
  (v_1^A)^{-1}V^A/\pi^{\infty} 
 \rightarrow  (v_1^A)^{-1}V^A/\pi^{\infty}\otimes_{V^A} V^AT \stackrel{\cong}{\longrightarrow} (v_1^A)^{-1}V^AT/p^{\infty} \]
where $\Psi$ is the left $V^AT$-comodule structure map on $(v_1^A)^{-1}V^A/\pi^{\infty}$,
i.e., 
\begin{align*} \Psi( (v_1^A)^n/{\pi^i}) 
 &= \eta_R( (v_1^A)^n)/{\pi^i} \\
 &= (v^1_A + (\pi - \pi^q))^n/\pi^i\end{align*}
for $n\geq 0$.
(The reason we swap the left coactions for the right coactions in~\ref{side swap iso}
is so that we can work with the isomorphism
\[ (v_1^A)^{-1}V^A/\pi^{\infty}\otimes_{V^A} V^AT \cong (\eta_L(v_1^A))^{-1}V^AT/\pi^{\infty} = (v_1^A)^{-1}V^AT/\pi^{\infty} \]
rather than the isomorphism
\[ V^AT\otimes_{V^A}(v_1^A)^{-1}V^A/\pi^{\infty}  \cong (\eta_R(v_1^A))^{-1}V^AT/\pi^{\infty} \]
which is less convenient.)

To compute $\Ext_{(V^A,V^AT)}^{*}(V^A,(v_1^A)^{-1}V^A/\pi^{\infty})$, 
we use the short exact sequence of graded $V^AT$-comodules
\[ 0 \rightarrow (v_1^A)^{-1} V^A/\pi 
 \rightarrow (v_1^A)^{-1} V^A/\pi^{\infty} 
 \stackrel{\pi}{\longrightarrow} (v_1^A)^{-1} V^A/\pi^{\infty} 
 \rightarrow 0\]
and its induced long exact sequence
of $\Ext$ groups
\begin{equation}\label{exact seq 30} 0 \rightarrow 
 \Ext^0_{(V^A,V^AT)}\left(V^A, (v_1^A)^{-1} V^A/\pi\right) \stackrel{i}{\longrightarrow}
 \Ext^0_{(V^A,V^AT)}\left(V^A, (v_1^A)^{-1} V^A/\pi^{\infty}\right) \end{equation} \[\stackrel{\pi}{\longrightarrow}
 \Ext^0_{(V^A,V^AT)}\left(V^A, (v_1^A)^{-1} V^A/\pi^{\infty}\right) \stackrel{\delta}{\longrightarrow}
 \Ext^1_{(V^A,V^AT)}\left(V^A, (v_1^A)^{-1} V^A/\pi\right) \rightarrow \dots .\]

Suppose $x\in \Ext^0_{(V^A,V^AT)}(V^A,(v_1^A)^{-1}V^A/\pi^{\infty})$
has the property that $\delta(x) \neq 0$.
Then I claim that there exists some positive integer $j$ such that $\pi^j x \in \im i$.
The proof of this claim is as follows: first, observe that
we have a sequence of one-to-one morphisms of 
graded left $V^AT$-comodules
\[ (v_1^A)^{-1}V^A/\pi \hookrightarrow (v_1^A)^{-1}V^A/\pi^2 \hookrightarrow (v_1^A)^{-1}V^A/\pi^3 \hookrightarrow \dots\]
and Lemma~\ref{ext and compactness lemma} tells us that the map
\[ \colim_m \Ext_{(V^A,V^AT)}^{s,t}(V^A,(v_1^A)^{-1}V^A/\pi^m)
 \rightarrow \Ext_{(V^A,V^AT)}^{s,t}(V^A,\colim_m (v_1^A)^{-1}V^A/\pi^m)g
  \cong \Ext_{(V^A,V^AT)}^{s,t}(V^A,(v_1^A)^{-1}V^A/\pi^{\infty}) \]
is an isomorphism for all $s$ and $t$.
Each of the $A$-modules $\Ext_{(V^A,V^AT)}^{s,t}(V^A,(v_1^A)^{-1}V^A/\pi^m)$ 
is a $\pi$-power-torsion $A$-module, that is, each element in 
$\Ext_{(V^A,V^AT)}^{s,t}(V^A,(v_1^A)^{-1}V^A/\pi^m)$ is killed by multiplication by some sufficiently large
power of $\pi$.
Consequently the colimit $\colim_m \Ext_{(V^A,V^AT)}^{s,t}(V^A,(v_1^A)^{-1}V^A/\pi^m)$
is also $\pi$-power-torsion.
So there exists some $j$ such that $\pi^j x = 0$, and now exactness of sequence~\ref{exact seq 30}
implies that $\pi^j x\in \im i$, as claimed.

As a consequence, if we give a description of $\im i$ as well as a description of the elements
in $\Ext^0_{(V^A,V^AT)}\left(V^A, (v_1^A)^{-1} V^A/\pi^{\infty}\right)$
whose multiples by $\pi^j$ are nonzero elements of $\im i$ for sufficiently large $i$,
then we will have a complete description of 
$\Ext^0_{(V^A,V^AT)}\left(V^A, (v_1^A)^{-1} V^A/\pi^{\infty}\right)$. 
(The above argument, which is the simplest way that I know of to
reduce the computation of $\Ext^0_{(V^A,V^AT)}\left(V^A, (v_1^A)^{-1} V^A/\pi^{\infty}\right)$
to the computation of $\im i$, was not given in Ravenel's paper~\cite{MR745362}; 
Ravenel simply computes $\im i$. Presumably something like the above line of argument
was clear to Ravenel, though not stated in his paper. The argument is involved enough,
requiring the two Lemmas~\ref{forgetful preserves colimits} and ~\ref{ext and compactness lemma},
that I think it is worth spelling out in the present paper.)
As Ravenel observes in~\cite{MR745362},
the image of $i$ is generated by the cohomology class of the $0$-cocycles
$v_1^i/\pi$ for $i\in \mathbb{Z}$, and to determine
$\Ext^0_{(V^A,V^AT)}(V^A, (v_1^A)^{-1} V^A/\pi^{\infty})$ we just need to determine
how many times each of these cohomology classes are divisible by $\pi$.
An element in 
$\Ext^0_{(V^A,V^AT)}(V^A, (v_1^A)^{-1} V^A/\pi^{\infty})$ is not divisible by $\pi$
if and only if its image under $\delta$ is nonzero.
Consequently, to compute the order of
$\Ext^{0,2i(q-1)}_{(V^A,V^AT)}(V^A, (v_1^A)^{-1} V^A/\pi^{\infty})$,
we just need to divide 
$v_1^i/\pi\in (v_1^A)^{-1} V^A/\pi^{\infty}$ by the largest possible power of $\pi$ 
such that the resulting quotient $x$ is still a $0$-cocycle and
such that $\delta(x) \neq 0$. 

Our next task is to determine how large that largest possible power of $\pi$ is.
Now let $n$ be a positive integer, and suppose that $\log_p\frac{p-1}{e}$ is not an integer, as in the statement of the theorem.
Then the $\pi$-adic valuation of 
\begin{equation}\label{cobar sum} (\eta_R - \eta_L)((v_1^A)^n) = \sum_{i=1}^{n} \binom{n}{i} (v_1^A)^{n-i} \left( \pi - \pi^q\right)^{i} (t_1^A)^i \end{equation}
is equal to 
\begin{align*} 
 \min\left\{ \nu_{\pi}\left(\binom{n}{i} (v_1^A)^{n-i} \left( \pi - \pi^q\right)^{i} (t_1^A)^i\right): i=1, \dots ,n\right\} 
  &=  \min\left\{ \nu_{\pi}\left( \binom{n}{i}\right) + i\nu_{\pi}\left( \pi - \pi^q\right): i=1, \dots ,n\right\} \\
  &=  \min\left\{ e\cdot\nu_p\left( \binom{n}{i}\right) + i : i=1,\dots ,n\right\}.
\end{align*}
(Here I am using the Araki-type generators $v_1^A, v_2^A, \dots $ for $V^A$, and consequently $\eta_R(v_1^A) = v_1^A + (\pi - \pi^q)t_1^A$. It works just as well to use the 
Hazewinkel-type generators $V_1^A, V_2^A, \dots $  for $V^A$, in which case $\eta_R(V_1^A) = V_1^A + \pi t_1^A$, and the rest of the argument goes through without
significant changes. See~\cite{MR2987372} for these formulas for $\eta_R$.)
Now by Lemma~\ref{valuations of binomial coefficients}, when $\log_p(\frac{e}{p-1})$ is not an integer, 
the quantity $e\cdot\nu_p\left( \binom{n}{i}\right) + i$, as a function
of $i$, takes its unique minimum value when $i=p^{\ceiling{\log_p\frac{e}{p-1}}}$
if $\log_p\frac{e}{p-1} < \nu_p(n)$
and when $i=p^{\nu_p(n)}$ if $\nu_p(n) < \log_p\frac{e}{p-1}$. 
Consequently there is a unique monomial of least $\pi$-adic valuation in the sum~\ref{cobar sum}, and 
it is 
\[ \binom{n}{p^{\ceiling{\log_p\frac{e}{p-1}}}} (v_1^A)^{n-p^{\ceiling{\log_p\frac{e}{p-1}}}}\left( \pi-\pi^q\right)^{p^{\ceiling{\log_p\frac{e}{p-1}}}} (t_1^A)^{p^{\ceiling{\log_p\frac{e}{p-1}}}},\]
with $\pi$-adic valuation 
\[  e(\nu_p(n) - \ceiling{\log_p\frac{e}{p-1}}) + p^{{\ceiling{\log_p\frac{e}{p-1}}}} \]
if $\nu_p(n) > \log_p\frac{e}{p-1}$; and, on the other hand, it is
\[ \binom{n}{p^{\nu_p(n)}} (v_1^A)^{n-p^{\nu_p(n)}}\left( \pi-\pi^q\right)^{p^{\nu_p(n)}} (t_1^A)^{p^{\nu_p(n)}},\]
with $\pi$-adic valuation $p^{\nu_p(n)}$,
if $\nu_p(n) < \log_p\frac{e}{p-1}$.
For brevity, we write
$a(n)$ for the $\pi$-adic valuation determined
above, that is,
\[ a(n) = \left\{ \begin{array}{ll}
 e(\nu_p(n) - \ceiling{\log_p\frac{e}{p-1}}) + p^{{\ceiling{\log_p\frac{e}{p-1}}}} &\mbox{\ if\ } \nu_p(n) > \log_p\frac{e}{p-1} \\
 p^{\nu_p(n)} &\mbox{\ if\ } \nu_p(n) < \log_p\frac{e}{p-1} ,\end{array}\right. \]
and we write 
$b(n)$ for the exponent of $t_1^A$
appearing in the monomial term of
$(\eta_R-\eta_L)((v_1^A)^n)$
of lowest $\pi$-adic valuation,
that is,
\begin{align*} 
b(n) 
 &= \left\{ \begin{array}{ll} 
  p^{\ceiling{\log_p\frac{e}{p-1}}} & \mbox{\ if\ } \log_p\frac{e}{p-1} < \nu_p(n), \\
  p^{\nu_p(n)} & \mbox{\ if\ } \nu_p(n) < \log_p\frac{e}{p-1},\end{array}\right. \\
(\eta_R -\eta_L)((v_1^A)^n) 
 &\equiv \binom{n}{b(n)} (v_1^A)^{n-b(n)} \pi^{b(n)} (t_1^A)^{b(n)} \mod \pi^{a(n)+1} .\end{align*}

Hence the $0$-cocycle 
$(v_1^A)^n/\pi\in (v_1^A)^{-1}V^A/\pi^{\infty}$
in the cobar complex computing
$\Ext_{(V^A,V^AT)}^*(V^A, (v_1^A)^{-1}V^A/\pi^{\infty})$
is divisible by $\pi^{a(n)-1}$, but not divisible 
by $\pi^{a(n)}$.

Finally, all we need to do is to check that the $1$-cocycle
\begin{align} 
\label{suspicious cocycle} (\eta_R-\eta_L)\left( (v_1^A)^n/\pi^{a(n)}\right) 
  &= \binom{n}{b(n)} (v_1^A)^{n-b(n)} \pi^{b(n)-a(n)}(t_1^A)^{b(n)} \\
\nonumber  &\in (v_1^A)^{-1}V^A/\pi \otimes_{V^A} V^AT,\end{align}
in the cobar complex for $(V^A,V^AT)$ with coefficients in $(v_1^A)^{-1}V^A/\pi$,
is not a $1$-coboundary.
From naturality of the cobar complex,
if the $1$-cocycle~\ref{suspicious cocycle} is a $1$-coboundary in the cobar complex for $(V^A,V^AT)$ with coefficients in $(v_1^A)^{-1}V^A/\pi$, then its
image in the cobar complex for $(\mathbb{F}_q[(v_1^A)^{\pm 1}],\Sigma^A(1))$ with coefficients in $\mathbb{F}_q[(v_1^A)^{\pm 1}]$ is a $1$-coboundary.
(In fact, the converse is also true, 
due to the Morava-Miller-Ravenel-type isomorphism
\[ \Ext_{(V^A,V^AT)}^{*,*}(V^A,(v_1^A)^{-1}V^A/\pi) \cong \Ext_{(\mathbb{F}_q[(v_1^A)^{\pm 1}],\Sigma^A(1))}^{*,*}(\mathbb{F}_q[(v_1^A)^{\pm 1}],\mathbb{F}_q[(v_1^A)^{\pm 1}]) ,\]
proved in~\cite{MR745362}.)

I claim that this $1$-cocycle $\binom{n}{b(n)} (v_1^A)^{n-b(n)} \pi^{b(n)-a(n)} (t_1^A)^{b(n)}$
cannot be a coboundary.
The proof is trivially easy: in the cobar complex for $(\mathbb{F}_q[(v_1^A)^{\pm 1}],\Sigma^A(1))$, the $0$-cochains are simply elements of $\mathbb{F}_q[(v_1^A)^{\pm 1}]$,
and given some element $x\in \mathbb{F}_q[(v_1^A)^{\pm 1}]$, its coboundary $\delta^0(x)$ is simply $\delta^0(x) = \eta_R(x) - \eta_L(x)$.
Since $\eta_R(v_1^A) = v_1^A - (\pi-\pi^q) t_1^A$, we have that
$\eta_R(v_1^A) = 0$ modulo $\pi$.
So $\delta^0\left( (v_1^A)^m\right) = \eta_R\left( (v_1^A)^m\right) - \eta_L\left( (v_1^A)^m\right) = 0$ for all $m$,
so $\binom{n}{b(n)} (v_1^A)^{n-b(n)} \pi^{b(n)-a(n)} (t_1^A)^{b(n)}$ is not a coboundary.

Consequently, for all $n$, we have the isomorphism of $A$-modules 
\[ \Ext_{(V^A,V^AT)}^{0,2n(q-1)}(V^A,(v_1^A)^{-1}V^A/\pi^{\infty}) \cong A/\pi^{a(n)},\]
generated as an $A$-module by the $0$-cocycle 
$(v_1^A)^n/\pi^{a(n)}$.

Now all that remains is to check that none of these $0$-cocycles supports a nonzero $d_1$-differential in the formal $A$-module chromatic spectral sequence,
since we already observed that, for dimensional reasons, no other differentials can interact with the height $1$ layer
in this spectral sequence. 
Computing $d_1$ differentials in this spectral sequence is straightforward: if $x$ is a class in 
\[ \Ext^{s,t}_{(V^A,V^AT)}\left(V_A,(v_h^A)^{-1}V^A/\left(\pi^{\infty}, (v_1^A)^{\infty}, \dots ,(v_{h-1}^A)^{\infty}\right)\right) \]
represented by a cobar complex $s$-cocycle
\[ \overline{x}\in (v_h^A)^{-1}V^A/\left(\pi^{\infty}, (v_1^A)^{\infty}, \dots ,(v_{h-1}^A)^{\infty}\right)\otimes_{V^A} V^AT \otimes_{V^A} \dots \otimes_{V^A} V^AT,\] 
then applying the chromatic complex coboundary 
\[ \delta : (v_h^A)^{-1}V^A/\left( \pi^{\infty}, (v_1^A)^{\infty}, \dots ,(v_{h-1}^A)^{\infty}\right) \rightarrow
  (v_{h+1}^A)^{-1}V^A/\left( \pi^{\infty}, (v_1^A)^{\infty}, \dots ,(v_{h-1}^A)^{\infty}, (v_h^A)^{\infty}\right)\]
to the leftmost tensor factor in $\overline{x}$ yields a cobar complex $s$-cocycle 
\[ d(\overline{x})\in (v_{h+1}^A)^{-1}V^A/\left( \pi^{\infty}, (v_1^A)^{\infty}, \dots ,(v_{h-1}^A)^{\infty}, (v_h^A)^{\infty}\right) \otimes_{V^A} V^AT \otimes_{V^A} \dots \otimes_{V^A} V^AT,\] 
and the cobar complex $d_1$-differential $d_1(x)$ is the cohomology class of this cocycle $d(\overline{x})$.

This process is very easy for our $0$-cocycles $(v_1^A)^n/\pi^{a(n)}$: 
the relevant chromatic complex coboundary is the composite 
\[ \delta : (v_1^A)^{-1}V^A/\pi^{\infty} \rightarrow
  (v_{2}^A)^{-1}V^A/\left( \pi^{\infty}, (v_1^A)^{\infty}\right)\]
of the projection 
\begin{equation}\label{projection 14} (v_1^A)^{-1}V^A/\pi^{\infty} \twoheadrightarrow V^A/\left(\pi^{\infty},(v_1^A)^{\infty}\right) \end{equation}
with the localization map 
\[ V^A/\left(\pi^{\infty},(v_1^A)^{\infty}\right) \hookrightarrow (v_2^A)^{-1} V^A/\left(\pi^{\infty},(v_1^A)^{\infty}\right).\]
If $n\geq 0$, then $(v_1^A)^n/\pi^{a(n)}$ maps to zero under the projection map~\ref{projection 14}.
Consequently, the $d_1$-differential is zero on $\Ext_{(V^A,V^AT)}^{0,2n(q-1)}(V^A,(v_1^A)^{-1}V^A/\pi^{\infty})$ as long as $n\geq 0$.

If $n<0$, the situation is even easier to understand: since $V^A$ and $V^AT$ are both concentrated in nonnegative grading degrees,
$\Ext^{s,t}_{(V^A,V^AT)}(V^A,V^A) \cong 0$ for $t<0$, since a nonzero class in $\Ext^{s,t}_{(V^A,V^AT)}(V^A,V^A)$ for $t<0$ would have to be represented by a cocycle
in the cobar complex of negative total grading degree.
Consequently, if $n<0$, then any nonzero class in $\Ext_{(V^A,V^AT)}^{0,2n(q-1)}(V^A,(v_1^A)^{-1}V^A/\pi^{\infty})$
must support a differential in the formal $A$-module chromatic spectral sequence. As we already observed, the only possible differentials interacting
with the height $1$ layer are possible $d_1$-differentials from the height $1$ layer to the height $2$ layer; hence every one of the classes
$(v_1^A)^n/\pi^{a(n)}$ for $n<0$ supports a $d_1$-differential, and these classes are all ``gone'' by the time we reach the formal $A$-module chromatic $E_2$-term.

Consequently,
\[ \Ext_{(V^A,V^AT)}^{0,2n(q-1)}(V^A,(v_1^A)^{-1}V^A/\pi^{\infty}) \cong
 A/\pi^{a(n)}\]
for all $n>0$.

Now by Lemma~\ref{valuations predicted by local conjecture},
the $\Ext$-groups predicted by Ravenel's Local Conjecture are the $\Ext$-groups we have just computed.
\end{proof}

The next result, Corollary~\ref{local computation}, is just a slightly different formulation of Theorem~\ref{local conjecture} which will be convenient to refer to when we begin proving the Global Conjecture.
\begin{corollary}\label{local computation} 
Suppose that $K/\mathbb{Q}_p$ is a finite field extension of ramification 
degree $e$, and suppose that $\log_p(\frac{e}{p-1})$ is not an integer, i.e.,
$\frac{e}{p-1}$ is not a power of $p$.
Let $A$ be the ring of integers in $K$,
then,
for each $n\in\mathbb{N}$,
we have isomorphisms of $A$-modules $\Ext_{(V^A,V^AT)}^{1,2n+1}(V^A, V^A)\cong 0$
and
\[ \Ext_{(V^A,V^AT)}^{1,2n}(V^A, V^A) \cong A/H_n,\]
where $H_{n}$ is the ideal in $A$ generated by all elements of the form
$(a^n-1)$ with $a\in A^{\times}$.

Equivalently, in terms of the moduli stack $\mathcal{M}_{fmA}$ of one-dimensional formal $A$-modules over $\Spec A$, and with notation as in Conventions~\ref{running conventions}:
for each $n\in\mathbb{N}$,
we have isomorphisms of $A$-modules $\Ext_{(V^A,V^AT)}^{1,2n+1}(V^A, V^A)\cong 0$
and
\[ H^{1,2n}_{fl}(\mathcal{M}_{fmA}; \mathcal{O}) \cong A/H_n,\]
where $H_{n}$ is the ideal in $A$ generated by all elements of the form
$(a^n-1)$ with $a\in A^{\times}$.
\end{corollary}
\begin{proof}
Let $f$ be the residue degree of $K/\mathbb{Q}_p$, and let $\pi$ be a uniformizer for $A$.
If $p^f-1$ does not divide $n$, then there exists some nonzero element $x$ in the residue field $A/\pi$
such that $x^{n} \neq 1$, and hence there exists some element $\overline{x} \in A^{\times}$ such that
$\overline{x}^{n} - 1$ has nonzero reduction modulo $\pi$, i.e.,
$\overline{x}^{n} - 1\in A^{\times}$.
Hence $A/H_n \cong 0$, which is the value of
$\Ext_{(V^A,V^AT)}^{1,2n}(V^A, V^A)$ predicted by
Theorem~\ref{local conjecture}.

On the other hand, if $p^f-1$ divides $n$, then for any unit $x\in A^{\times}$, we have that
$x^{p^f-1}$ is congruent to $1$ modulo $\pi$, and conversely, if $y\in A$ is congruent to $1$ modulo $\pi$,
then $y$ has a $(p^f-1)$st root $x\in A^{\times}$, by Hensel's Lemma; consequently
$A/H_n$ coincides with $A/I_{n/(p^f-1)}$, the value of
$\Ext_{(V^A,V^AT)}^{1,2n}(V^A, V^A)$ predicted by
Theorem~\ref{local conjecture}.
\end{proof}

\section{$n$-congruing ideals and a precise statement of the Global Conjecture.}

\subsection{The Hasse principle for $n$-congruing ideals.}

Ravenel's original statement of the Global Conjecture, which is included verbatim in the present paper as Conjecture~\ref{global conj},
is phrased loosely enough that one has some leeway in interpreting the conjecture so that some version of it can be proven:
the conjecture seems to be that, up to some small factor, $\Ext^{1,2m}_{(L^A,L^AB)}(L^A,L^A)$ is isomorphic to $A/I$ where $I$ is an ideal with the property that, if
$a\in A$, then $a^N(a^n-1)\in I$ for some $N\in\mathbb{N}$. Ravenel's statement does not make it clear exactly which ideal with that property $I$ ought to be, however,
or even if there {\em is} more than one such ideal. So at this point it is natural to make the following definition:
\begin{definition}\label{def of n-congruing ideal}
Let $A$ be a commutative ring and let $n$ be a positive integer. We say that an ideal $I$ in $A$ is
{\em $n$-congruing} if, for every $a\in A$, there exists some $N \in\mathbb{N}$ such that
$a^N(a^n-1)\in I$.
\end{definition}

Our first task, in proving the Global Conjecture, is to make the conjecture precise. The most convincing way to make it precise would be to show that there is in fact a 
``universal'' $n$-congruing ideal among the collection of all the $n$-congruing ideals in $A$. In that case, the ``universal'' $n$-congruing ideal would be the right one to appear in the statement of 
the Global Conjecture. Notice there always exists at least one (trivial) choice of $n$-congruing ideal
of $A$, namely $A$ itself. What is less clear is whether there are others, and in particular, whether there exists a {\em minimal} $n$-congruing ideal in $A$. (Note
that the zero ideal is not $n$-congruing unless $n=0$.)

In this section I prove that there is, in fact, a universal (specifically, {\em minimal}) ideal $I$ in $A$ 
with the property that if $a\in A$, then $a^N(a^n-1)\in I$ for some $N\in\mathbb{N}$; this is a consequence 
of Theorem~\ref{hasse principle}, which is also interesting in its own right, as it establishes a Hasse principle for
$n$-congruing ideals. 
Consequently Conjecture~\ref{weaker rigorous global conj} (and its stronger form, Conjecture~\ref{rigorous global conj})
is the desired precise form of Ravenel's Global Conjecture.
In the next section (specifically Corollaries~\ref{weak form of global conj holds} and \ref{strong form of global conj holds} ) the reader can find proofs that Conjecture~\ref{weaker rigorous global conj} is actually true,
and that Conjecture~\ref{rigorous global conj} is true when certain hypotheses are met.

\begin{lemma}\label{valuation inequality lemma}
Let $A$ be a Dedekind domain, 
let $t\in\mathbb{N}$, and let $\mathfrak{p}$ be a maximal ideal in $A$.
Let $H_t^{\mathfrak{p}}$ denote the 
ideal in $\hat{A}_{\mathfrak{p}}$ generated by all elements of the form $x^t-1$ for $x\in \hat{A}_{\mathfrak{p}}^{\times}$.
Let $u_{\mathfrak{p},t}$ denote the $\mathfrak{p}$-adic valuation of 
the ideal $H_t^{\mathfrak{p}}$, i.e., $u_{\mathfrak{p},t}$ is the greatest integer $u$
such that $H_t^{\mathfrak{p}} \subseteq \mathfrak{p}^{u}\subseteq \hat{A}_{\mathfrak{p}}$.

Then, for all $k\in A$ not contained in $\mathfrak{p}$, we have the inequality
\begin{align}\label{valuation inequality 1} \nu_{\mathfrak{p}}(k^t - 1) &\geq u_{\mathfrak{p},t},\end{align}
and furthermore, there exists some $k\in A$ not contained in $\mathfrak{p}$ such that
\begin{align}\label{valuation inequality 2} \nu_{\mathfrak{p}}(k^t - 1) = u_{\mathfrak{p},t}.\end{align}
\end{lemma}
\begin{proof}
To produce the inequality~\ref{valuation inequality 1}, fix $k\in A$ not contained in $\mathfrak{p}$, and observe that
the image $\overline{k}$ of $k$ under the completion map $A\rightarrow \hat{A}_{\mathfrak{p}}$ is a unit.
Consequently 
\[ \overline{k}^t-1 \in H_t^{\mathfrak{p}} \subseteq \mathfrak{p}^{u},\]
so $\overline{k}^t-1$ has $\mathfrak{p}$-adic valuation at least $u$.

To produce the equality~\ref{valuation inequality 2}, we choose an element $k\in A$ whose image 
$\overline{k}\in \hat{A}_{\mathfrak{p}}$ is a unit and makes $\nu_{\mathfrak{p}}(\overline{k}^t-1)$ as small as possible,
i.e., $\overline{k}^t-1$ is an element of minimal $\mathfrak{p}$-adic valuation in $H_t^{\mathfrak{p}}$.
(It is possible to make such a choice of $k$ because we can simply start with an element $x\in \hat{A}_{\mathfrak{p}}^{\times}$ such that $x^t-1$ is of minimal $\mathfrak{p}$-adic valuation in $H_t^{\mathfrak{p}}$, and then create an element $\overline{k}$ of $\hat{A}_{\mathfrak{p}}$ 
by truncating $x$ by setting all the $\mathfrak{p}^m$-coefficients to zero for $m>>0$, in the $\mathfrak{p}$-adic expansion
of $x$. Then $\overline{k}$ is still a unit in $\hat{A}_{\mathfrak{p}}$ since its mod $\mathfrak{p}$ reduction is still nonzero, and $\overline{k}^t-1$ 
has the same $\mathfrak{p}$-adic valuation as $x^t-1$ if $m$ was chosen to be large relative to $t$. Furthermore, $\overline{k}$ has only finitely many
nonzero coefficients in its $\mathfrak{p}$-adic expansion, so $\overline{k}$ is the image in $\hat{A}_{\mathfrak{p}}$ of some element
$k\in A$ with the desired properties.)
Then $\nu_{\mathfrak{p}}(k^t-1) = \nu_{\mathfrak{p}}(\overline{k}^t-1) = u_{\mathfrak{p},t}$, as desired.
\end{proof}

\begin{lemma}\label{adic congruing ideals}
Let $\hat{A}_{\mathfrak{p}}$ be a complete (that is, complete in the $\mathfrak{p}$-adic topology) 
discrete valuation ring with maximal ideal $\mathfrak{p}$, and let $n$ be a positive integer.
Let $H_n^{\mathfrak{p}}$ denote the 
ideal in $\hat{A}_{\mathfrak{p}}$ generated by all elements of the form $x^n-1$ for $x\in \hat{A}_{\mathfrak{p}}^{\times}$.
Then an ideal $I$ in $\hat{A}_{\mathfrak{p}}$ is $n$-congruing if and only if $I$ contains $H_n^{\mathfrak{p}}$.
\end{lemma}
\begin{proof}
Suppose $I$ is an ideal in $\hat{A}_{\mathfrak{p}}$ containing $H_n^{\mathfrak{p}}$, and let $a\in A$. 
If $a\notin \mathfrak{p}$,
then $a$ is a unit in $\hat{A}_{\mathfrak{p}}$, so $a^n-1 \in H_n^{\mathfrak{p}}\subseteq I$.
If, on the other hand, $a\in \mathfrak{p}$,
then some power of $a$ is contained in $I$, since every ideal in a discrete valuation ring is a power of the maximal ideal,
so $I = \mathfrak{p}^m$ for some $m$, and then $a^m \in \mathfrak{p}^m = I$.
Hence every ideal in $\hat{A}_{\mathfrak{p}}$ containing $H_n^{\mathfrak{p}}$ is $n$-congruing.

On the other hand, suppose that $I$ is an $n$-congruing ideal in $\hat{A}_{\mathfrak{p}}$,
and suppose that $x\in \hat{A}_{\mathfrak{p}}^{\times}$.
Then there exists some integer $N$ such that $x^N(x^n-1)\in I$, since $I$ is $n$-congruing.
But $x^N$ is a unit for all $N$, so $x^N(x^n-1)\in I$ if and only if $x^n-1\in I$.
So $x^n-1\in I$.
Hence $H_n^{\mathfrak{p}}\subseteq I$. Hence every $n$-congruing ideal in $\hat{A}_{\mathfrak{p}}$ contains $H_n^{\mathfrak{p}}$.

\end{proof}

\begin{lemma}\label{valuation powers convergence lemma}
Let $A$ be a Dedekind domain, and let $n$ be a positive integer.
Let $k$ be a nonzero element of $\hat{A}_{\mathfrak{p}}$, and let 
$k(j)$ denote the truncation of $k$ in which we replace the $\mathfrak{p}^m$-coefficient in the $\mathfrak{p}$-adic expansion of $k$ by zero, for all
$m > j$.
Put another way: the ring $\hat{A}_{\mathfrak{p}}$ is defined as the limit $\hat{A}_{\mathfrak{p}} = \lim_{i\rightarrow \infty} A/\mathfrak{p}^i$,
so to specify an element $k$ in $\hat{A}_{\mathfrak{p}}$, we can specify a sequence of elements
\[ ( k_0, k_1, k_2, \dots : \forall i\ \ k_i\in A/\mathfrak{p}^i,\mbox{\ and\ \ } k_{i+1} \equiv k_i \mod \mathfrak{p}^i).\]
By $k(j)$ we mean the element of $\hat{A}_{\mathfrak{p}}$ given by the sequence
\[ ( k_0, k_1, k_2, \dots , k_{j-2}, k_{j-1}, k_j, k_j, k_j , \dots ).\]

Let $I$ be an ideal in $\hat{A}_{\mathfrak{p}}$.
Suppose that, for each positive integer $j$, there exists some $N\in\mathbb{N}$ such that
$k(j)^N(k(j)^n-1)\in I$. Let $N(j)$ be the smallest natural number $N$ such that
$k(j)^N(k(j)^n-1)\in I$.

Then the sequence of natural numbers
\begin{equation}\label{sequence 3} N(1), N(2), N(3), \dots \end{equation}
is eventually constant.
\end{lemma}
\begin{proof}
Since $\hat{A}_{\mathfrak{p}}$ is a discrete valuation ring, whether 
$k(j)^N(k(j)^n-1)$ is contained in $I$ is determined entirely by the $\mathfrak{p}$-adic valuations of
$k(j)^N(k(j)^n-1)$ and $I$, i.e., $N(j)$ is the least $N$ such that
\begin{align*}
 \nu_{\mathfrak{p}}( k(j)^N (k(j)^n-1)) &\geq \nu_{\mathfrak{p}}(I),
\end{align*}
where we write $\nu_{\mathfrak{p}}(I)$ for the least $\mathfrak{p}$-adic valuation of an element in $I$,
i.e.,
\begin{align*} \nu_{\mathfrak{p}}(I) &= \min\{ \nu_{\mathfrak{p}}(i) : i\in I\}.\end{align*}
Consequently, $N(j)$ is the least $N$ such that
\begin{align*}
 N\nu_{\mathfrak{p}}( k(j)) + \nu_{\mathfrak{p}}(k(j)^n-1) &\geq \nu_{\mathfrak{p}}(I),
\end{align*}
i.e., $N(j)$ is the natural number ceiling 
\begin{align}
\label{N(j) formula} N(j) &= \ceiling{\frac{\nu_{\mathfrak{p}}(I) - \nu_{\mathfrak{p}}(k(j)^n-1)}{\nu_{\mathfrak{p}}(k(j))}} .
\end{align}
(By the ``natural number ceiling'' of a real number $x$ I mean the least natural number which is greater than or equal to $x$.)
By basic properties of analysis in a complete discrete valuation ring,
$\lim_{j\rightarrow \infty} k(j) = k$,
so the sequence of natural numbers
\begin{equation}\label{sequence 1} (\nu_{\mathfrak{p}}(k(1)), \nu_{\mathfrak{p}}(k(2)), \nu_{\mathfrak{p}}(k(3)), \dots )\end{equation}
(which is indeed a sequence of natural numbers, never taking the value $\infty$, since we assumed that $k\neq 0$),
is eventually constant, converging to $\nu_{\mathfrak{p}}(k)$.

Continuity of addition and the $n$th power function on $\hat{A}_{\mathfrak{p}}$ also gives us that
\begin{align} \label{valuation limit equality 1} \lim_{j\rightarrow \infty} (k(j)^n - 1) &= \left( \lim_{j\rightarrow \infty} k(j)\right)^n - 1 \\
\nonumber &= k^n - 1.\end{align}
The sequence
\begin{equation}\label{sequence 2} (\nu_{\mathfrak{p}}(k(1)^n-1), \nu_{\mathfrak{p}}(k(2)^n-1), \nu_{\mathfrak{p}}(k(3)^n-1), \dots )\end{equation}
is a sequence of extended natural numbers, i.e., a sequence of elements of the set $\mathbb{N}\cup \{ \infty\}$,
and it is again (due to~\ref{valuation limit equality 1}) eventually constant, converging to
$\nu_{\mathfrak{p}}(k^n-1)$.

Consequently one of two things happens: 
\begin{itemize}
\item If $\nu_{\mathfrak{p}}(k^n-1) = \infty$ (i.e., if $k^n-1 = 0$),
then the sequence~\ref{sequence 2} is an eventually constant sequence of extended natural numbers converging to $\infty$,
and so $N(j) = 0$ for all sufficiently large $j$. Consequently sequence~\ref{sequence 3} is eventually constant (and converges to zero).
\item If $\nu_{\mathfrak{p}}(k^n-1) \neq \infty$ (i.e., if $k^n-1\neq 0$),
then the sequences~\ref{sequence 1} and~\ref{sequence 2} are both eventually constant, hence from equation~\ref{N(j) formula} we see that
the sequence~\ref{sequence 3} is eventually constant.
\end{itemize}
\end{proof}

\begin{definition}\label{def of minkowski}
We will say that a Dedekind domain $A$ is {\em Minkowski} if $A$ satisfies the following two conditions:
\begin{itemize}
\item All the residue fields of $A$ are finite, i.e., $A/\mathfrak{p}$ is finite for all maximal ideals $\mathfrak{p}$.
\item For each natural number $N$, there exist only finitely many maximal ideals $\mathfrak{p}$ of $A$ such that $N(\mathfrak{p}) < N$.
(Recall that the ``absolute norm'' $N(\mathfrak{p})$ of a maximal ideal $\mathfrak{p}$ is simply the cardinality of the residue field $A/\mathfrak{p}$.)
\end{itemize}
\end{definition}
For example, the ring of integers in any number field is Minkowski, and any localization of a Minkowski Dedekind domain is also Minkowski.

\begin{lemma}\label{finiteness of A/pr}
Let $A$ be a Minkowski Dedekind domain, and let $n$ be a positive integer.
For each maximal ideal $\mathfrak{p}$ in $A$, let $H_n^{\mathfrak{p}}$ denote the 
ideal in $\hat{A}_{\mathfrak{p}}$ generated by all elements of the form $x^n-1$ for $x\in \hat{A}_{\mathfrak{p}}^{\times}$.
Let 
\[ pr_n: A  \rightarrow \prod_{\mbox{maximal}\ \mathfrak{p}\subseteq A} \hat{A}_{\mathfrak{p}}/H_n^{\mathfrak{p}} \]
be the ring homomorphism whose $\mathfrak{p}$-component $A\rightarrow \hat{A}_{\mathfrak{p}}/H_n^{\mathfrak{p}}$ 
is the composite of the completion map $A\rightarrow \hat{A}_{\mathfrak{p}}$ with the modulo-$H_n^{\mathfrak{p}}$-reduction
map $\hat{A}_{\mathfrak{p}}\rightarrow \hat{A}_{\mathfrak{p}}/H_n^{\mathfrak{p}}$.

Then $pr_n$ is surjective, and the quotient ring $A/\ker pr_n$ is finite.
\end{lemma}
\begin{proof}
I claim that, for all but finitely many maximal ideals $\mathfrak{p}$ in $A$, the quotient $\hat{A}_p/H_n^{\mathfrak{p}}$ is the zero ring.
(Of course, which of the maximal ideals $\mathfrak{p}$ have nonzero quotient $\hat{A}_p/H_n^{\mathfrak{p}}$ depends on the choice of $n$; but
for any given choice of $n$, there are only finitely many such maximal ideals $\mathfrak{p}$.)
Suppose $\mathfrak{p}$ is a maximal ideal in $A$ such that $n$ is not divisible by $N(\mathfrak{p})-1 = \#\left( (A/\mathfrak{p})^{\times}\right)$.
Then there exists an element $x$ in $\hat{A}_p^{\times}$ such that $x^n$ is not congruent to $1$ modulo $\mathfrak{p}$,
hence $x^n-1$ is a unit in $\hat{A}_p$, hence the ideal $H_n^{\mathfrak{p}}\subseteq \hat{A}_p$ contains a unit.
Consequently $\hat{A}_p/H_n^{\mathfrak{p}}$ is only nonzero when $n$ is divisible by $\#(A/\mathfrak{p})-1$, which only happens for
finitely many $\mathfrak{p}$, since $A$ is Minkowski.

Clearly the ideal $H_n^{\mathfrak{p}}$ is also never the zero ideal, hence $A/H_n^{\mathfrak{p}}$ is always finite. 
Now the Chinese Remainder Theorem implies surjectivity of the map $pr_n$,
and hence $\prod_{\mbox{maximal}\ \mathfrak{p}\subseteq A} \hat{A}_{\mathfrak{p}}/H_n^{\mathfrak{p}} \cong A/\ker pr_n$ is finite.
\end{proof}

\begin{theorem}{\bf (Hasse principle for $n$-congruing ideals.)}\label{hasse principle}
Let $A$ be a Minkowski Dedekind domain, and let $n$ be a positive integer.
Let $pr_n$ be as in Lemma~\ref{finiteness of A/pr}.

Then the following conditions on an ideal $I$ of $A$ are equivalent:
\begin{enumerate}
\item \label{condition 1} $I$ is $n$-congruing.
\item \label{condition 4} For all maximal ideals $\mathfrak{p}$ of $A$, the ideal $\hat{I}_{\mathfrak{p}}$ of $\hat{A}_{\mathfrak{p}}$ is $n$-congruing.
\item \label{condition 2} $\ker pr_n \subseteq I$.
\end{enumerate}

Consequently, there exists a unique minimal $n$-congruing ideal of $A$, namely, $\ker pr_n$.
\end{theorem}
\begin{proof}
\begin{itemize}
\item That condition~\ref{condition 4} is equivalent to condition~\ref{condition 2} is Lemma~\ref{adic congruing ideals}.
\item That condition~\ref{condition 2} implies condition~\ref{condition 1} is as follows: suppose that $I$ is an ideal of $A$, and
suppose that $\ker pr_n\subseteq I$. For each maximal ideal $\mathfrak{p}$ of $A$, we write $c_{\mathfrak{p}}$ for the completion map
$c_{\mathfrak{p}}: A \rightarrow \hat{A}_{\mathfrak{p}}$.
Then $I$ contains every element $a$ of $A$ with the property that, for all maximal ideals $\mathfrak{p}$ of $A$,
$c_{\mathfrak{p}}(a)$ is divisible by $x^n-1$ for some unit $x\in \hat{A}_{\mathfrak{p}}^{\times}$.

Consequently, if $a\in A^{\times}$, then for each maximal ideal $\mathfrak{p}$ of $A$, 
$c_{\mathfrak{p}}(a)^n\in \hat{A}_{\mathfrak{p}}^{\times}$ and hence $c_{\mathfrak{p}}(a)^n - 1$ is divisible by $x^n-1$ for some unit $x\in \hat{A}_{\mathfrak{p}}^{\times}$
(namely, $x = c_{\mathfrak{p}}(a)$). This is the case for all maximal ideals $\mathfrak{p}$ of $A$, so $a^n-1\in I$.

On the other hand, if $a\notin A^{\times}$, then either $a^N(a^n-1)\in I$ for some $N\in\mathbb{N}$, or the quotient ring
$A/I$ has an element $a$ of infinite order. But this second case is impossible, 
since $I\supseteq \ker pr_n$ and so $A/I$ is a quotient of $A/\ker pr_n$, which is finite 
by Lemma~\ref{finiteness of A/pr}.
So, if $a\notin A^{\times}$, then $a^N(a^n-1)\in I$.

Consequently, for all $a\in A$, $a^N(a^n-1)\in I$ for some $N\in\mathbb{N}$.
Hence $I$ is $n$-congruing.
\item That condition~\ref{condition 1} implies condition~\ref{condition 4} is as follows:
suppose that $I$ is an $n$-congruing ideal in $A$.
Choose a maximal ideal $\mathfrak{p}$ of $A$. We continue to write $c_{\mathfrak{p}}$ for the completion map
$c_{\mathfrak{p}}: A \rightarrow \hat{A}_{\mathfrak{p}}$.
Let $k\in \hat{A}_{\mathfrak{p}}$ be nonzero. Then we apply Lemma~\ref{valuation powers convergence lemma} to $k$:
each $k(j)\in \hat{A}_{\mathfrak{p}}$ for $j\in \mathbb{N}$ is in the image of the completion map $c_{\mathfrak{p}}$, 
consequently for each $j$ there exists some $N\in\mathbb{N}$ such that $k(j)^N(k(j)^n-1)\in I$, 
since $I$ is $n$-congruing. So the assumptions made in the statement of Lemma~\ref{valuation powers convergence lemma} are satisfied in our case.

Now Lemma~\ref{valuation powers convergence lemma} tells us that,
if we write $N(j)$ for the least natural number $N$ such that 
$k(j)^N(k(j)^n-1)\in I$, then the sequence $(N(1), N(2), N(3), \dots)$ is eventually constant.
So we can choose some natural number $j$ such that $N(j) = N(j+c)$ for all $c\geq 0$.
Then $k(j+c)^{N(j)}(k(j+c)^n-1)\in \hat{I}_{\mathfrak{p}}$ for all $c\geq 0$.
Taking the limit over $c$, we have that
\begin{align*} \lim_{c\rightarrow \infty} \left( k(j+c)^{N(j)}(k(j+c)^n-1)\right)  
 &= \left(\lim_{c\rightarrow \infty} k(j+c)\right)^{N(j)}\left((\lim_{c\rightarrow\infty}k(j+c))^n-1\right) \\
 &= k^{N(j)}(k^n-1) \in \hat{I}_{\mathfrak{p}},\end{align*}
with the containment in $\hat{I}_{\mathfrak{p}}$ because $\hat{I}_{\mathfrak{p}}$ is closed (again, a standard fact about nonzero ideals in a complete discrete valuation ring with finite residue field, or more generally, finite-index subgroups of a profinite group: they are both open and closed) 
and hence contains all limits of sequences contained in itself.

Consequently, for any $k\in \hat{A}_{\mathfrak{p}}$, there exists some integer $N$ (namely, the $N(j)$ produced above) such that
$k^{N(j)}(k^n-1) \in \hat{I}_{\mathfrak{p}}\subseteq \hat{A}_{\mathfrak{p}}$. So $\hat{I}_{\mathfrak{p}}$ is an $n$-congruing ideal in $\hat{A}_{\mathfrak{p}}$.
\end{itemize}
\end{proof}

Consequently, ``universal'' (in the sense that the $n$-congruing ideal involved in the conjecture is the ``universal'' one, i.e., the minimal one) precise versions of Ravenel's Global Conjecture, Conjecture~\ref{global conj}, 
are Conjectures~\ref{weaker rigorous global conj} and~\ref{rigorous global conj}. I prove many cases of Conjecture~\ref{rigorous global conj} in Corollary~\ref{strong form of global conj holds}, and 
I prove completely (handling all cases) the weaker Conjecture~\ref{weaker rigorous global conj}, in Corollary~\ref{weak form of global conj holds}.

\subsection{Relation to Adams's $h(f,t)$.}

In~\cite{MR0198468}, Adams makes the following definition:
\begin{definition}
Let $f: \mathbb{Z}\rightarrow\mathbb{N}$ be a function. 
Then we define $h(f,t)$ to be the greatest common divisor
of the set of integers
\[ \{ k^{f(k)}(k^n-1): k\in\mathbb{Z}\}.\]
\end{definition}

Adams then uses his numbers $h(f,n)$ to prove essentially the $K=\mathbb{Q}$ case of our Theorem~\ref{global computation}, that is,
the original computation which motivated Ravenel to make the Global Conjecture
(and indeed, the only case of the Global Conjecture which was known to hold,
before the results of the present paper were obtained). 

In our setting, where $K$ is a finite extension of $\mathbb{Q}$, it is somewhat problematic to try to use Adams's numbers $h(f,n)$, because really what we are after is the ideal that they generate, and when the class number of the ring of integers of $K$ is greater than one, the ideal one really wants to work with may not be principal.
So it is reasonable to instead make the definition of the ideals $H(f,t)$ in the following proposition:
\begin{prop}\label{valuation inequality lemma 2}
Let $K/\mathbb{Q}$ be a finite extension with ring of integers $A$. 
Let $f: A \rightarrow \mathbb{N}$ be a function.
If $t\in \mathbb{N}$, we let $H(f,t)$ denote the ideal of $A$ which is the sum,
over all $k\in A$, of the principal ideals $(k^{f(k)}(k^t-1))$, i.e.,
\begin{align*} H(f,t) 
 &= \sum_{k\in A} \left( k^{f(k)}(k^t-1)\right) \end{align*}
Let 
\[ pr_t: A \rightarrow \prod_{\mbox{maximal}\ \mathfrak{p}\subseteq A} \hat{A}_{\mathfrak{p}}/H_t^{\mathfrak{p}} \]
be the ring homomorphism whose $\mathfrak{p}$-component $A\rightarrow \hat{A}_{\mathfrak{p}}/H_t^{\mathfrak{p}}$ 
is the composite of the completion map $A\rightarrow \hat{A}_{\mathfrak{p}}$ with the modulo-$H_t^{\mathfrak{p}}$-reduction
map $\hat{A}_{\mathfrak{p}}\rightarrow \hat{A}_{\mathfrak{p}}/H_t^{\mathfrak{p}}$.

Then $\ker pr_t \subseteq H(f,t)$.
\end{prop}
\begin{proof}
Choose a function $f: A \rightarrow\mathbb{N}$, and let $\mathfrak{p}$ be a maximal ideal in $A$.
Let $w_{\mathfrak{p}}$ be the $\mathfrak{p}$-adic valuation of the ideal $H(f,t)$, i.e., 
$w_{\mathfrak{p}}$ is the greatest integer $w$ such that $H(f,t)\subseteq \mathfrak{p}^w$.
By Lemma~\ref{valuation inequality lemma}, there exists some $k$ not in $\mathfrak{p}$ such that
\[
\nu_{\mathfrak{p}}(k^{f(t)}(k^t-1)) = \nu_{\mathfrak{p}}(k^t-1) = u_{\mathfrak{p},t},\] and all other choices of $k_0$ not in $\mathfrak{p}$
yield 
\[ \nu_{\mathfrak{p}}(k_0^{f(t)}(k_0^t-1))= \nu_{\mathfrak{p}}(k_0^t-1) \geq u_{\mathfrak{p},t}.\] 
The valuation behavior of elements $k_0\in \mathfrak{p}$
is irrelevant now: since there exists some $k$ such that
$\nu_{\mathfrak{p}}(k^{f(t)}(k^t-1))  = u_{\mathfrak{p},t}$,
we have
\begin{equation}\label{valuation inequality 3} \nu_{\mathfrak{p}}(H(f,t)) = w_{\mathfrak{p},t}\leq u_{\mathfrak{p},t} = \nu_{\mathfrak{p}}(H_t^{\mathfrak{p}}),\end{equation} which is what we want.

Since we have inequality~\ref{valuation inequality 3} for all maximal primes $\mathfrak{p}$ in $A$, we then have
that $\ker pr_t\subseteq H(f,t)$.
\end{proof}

The ideals $H(f,n)$ still are not central to the ideas in the present paper, but I include Proposition~\ref{valuation inequality lemma 2} and the following
Corollary~\ref{adams counting condition implies n-congruing} to make more clear what the relationship is between our $n$-congruing ideals and 
the numbers $h(f,n)$ defined by Adams.

\begin{corollary}\label{adams counting condition implies n-congruing}
Let $K/\mathbb{Q}$ be a finite extension with ring of integers $A$, and let $n$ be a positive integer.
For every function $f: A \rightarrow\mathbb{N}$, the ideal $H(f,n)$ satisfies the equivalent conditions of Theorem~\ref{hasse principle}.
In particular, $H(f,n)$ is $n$-congruing.
\end{corollary}
\begin{proof}
By Proposition~\ref{valuation inequality lemma 2}, $H(f,n)$ satisfies condition~\ref{condition 2} of Theorem~\ref{hasse principle}.
\end{proof}

I do not know if the converse to Corollary~\ref{adams counting condition implies n-congruing} holds. That is,
I do not know if every $n$-congruing ideal in $A$ is of the form $H(f,n)$ for some function $f: A \rightarrow\mathbb{N}$.
It seems likely that this is so (Adams proves it for $A = \mathbb{Z}$, in~\cite{MR0198468}), 
but it is not necessary for the arguments in the present paper, so I do not pursue the question here.

\section{Proof of Ravenel's Global Conjecture.}

In order to prove the Global Conjecture, I need to invoke 
A. Pearlman's proof of Ravenel's Local-Global Conjecture (originally conjectured
in~\cite{MR745362}, but note that the statement below is slightly different from Ravenel's statement in~\cite{MR745362}, which seems to have a typographical error in the coefficients), from Pearlman's unpublished 
thesis~\cite{pearlmanthesis}:
\begin{theorem}{\bf (Ravenel's Local-Global conjecture.)}\label{local-to-global conj}
Suppose $K/\mathbb{Q}$ is a finite Galois extension, and let $A$ be the ring of
integers of $K$.
Then, for any $L^AB$-comodule $M$ and any choice of prime $\mathfrak{p}$ in $A$, we have an isomorphism of graded $A_{\mathfrak{p}}$-modules:
\[ A_{\mathfrak{p}} \otimes_{A} \Ext^{*}_{(L^A,L^AB)}(L^A, M) 
 \cong \Ext^*_{(V^{A_{\mathfrak{p}}},V^{A_{\mathfrak{p}}}T)}(V^{A_{\mathfrak{p}}}, V^{A_{\mathfrak{p}}} \otimes_{L^A} M).\]
\end{theorem}
\begin{proof}
See~\cite{pearlmanthesis}.
\end{proof}

Now we define a useful invariant of an extension $K/\mathbb{Q}$:
\begin{definition}\label{def of prime power disc}
Let $K/\mathbb{Q}$ be a finite field extension with ring of integers $A$.
Then by the {\em prime-power-ramification discriminant} $\underline{\Delta}_{K/\mathbb{Q}}$ we mean the
number $\underline{\Delta}_{K/\mathbb{Q}}\in\mathbb{N}$ which is the product of all the prime numbers $p$
with the property that there exists a prime ideal $\mathfrak{p}$ of $A$ over $p$
with ramification degree $e_{\mathfrak{p}}$ such that $\log_p(\frac{e_{\mathfrak{p}}}{p-1})$ is an integer.
\end{definition}

\begin{example}\label{examples of ppr discriminants}
\begin{itemize}
\item $\underline{\Delta}_{\mathbb{Q}/\mathbb{Q}} = 2$, since $\log_2(\frac{1}{2-1}) = 0$, an integer, while $\log_p(\frac{1}{p-1})$ is not an integer if $p>2$.
\item For quadratic $K/\mathbb{Q}$, $\underline{\Delta}_{K/\mathbb{Q}} = 6$ if $3$ ramifies,
and $\underline{\Delta}_{K/\mathbb{Q}} = 2$ if $3$ does not ramify.
\item For cubic $K/\mathbb{Q}$, we have $2\mid \underline{\Delta}_{K/\mathbb{Q}}$ if and only if $2$ does not ramify totally in $K$, and 
$3\mid \underline{\Delta}_{K/\mathbb{Q}}$ if and only if some prime over $3$ has ramification degree $2$ over $3$. No other primes divide $\underline{\Delta}_{K/\mathbb{Q}}$.
\item In particular, for cubic Galois $K/\mathbb{Q}$, we have $\underline{\Delta}_{K/\mathbb{Q}} = 1$ if $2$ ramifies in $K$, and
$\underline{\Delta}_{K/\mathbb{Q}} = 2$ otherwise.
\item More generally, if $\ell$ is an odd prime number and $K/\mathbb{Q}$ is a Galois extension of degree $\ell$, then we have $\underline{\Delta}_{K/\mathbb{Q}} = 1$ if $2$ ramifies in $K$, and
$\underline{\Delta}_{K/\mathbb{Q}} = 2$ otherwise.
\item Let $K$ be the splitting field of $x^5 + x + 2$ over $\mathbb{Q}$. Then $K/\mathbb{Q}$ is Galois and $[K:\mathbb{Q}] = 24$.
The only primes of $\mathbb{Z}$ which ramify in the ring of integers of $K$ are $2$ and $349$; the primes over $2$ have 
ramification degree $3$, and the primes over $349$ have ramification degree $2$, so the prime-power-ramification discriminant of $K/\mathbb{Q}$ is $1$. Furthermore, the Galois group $G_{K/\mathbb{Q}}$ is an order $24$ subgroup of the symmetric group on $5$ letters, hence is nonabelian. 
So $K$ is
an example of a nonabelian Galois extension of $\mathbb{Q}$ with trivial prime-power-ramification discriminant.
\end{itemize}
\end{example}

\begin{observation}\label{discriminant divisibility}
If $K\neq \mathbb{Q}$, then the prime-power-ramification discriminant divides two times the classical discriminant, that is, $\underline{\Delta}_{K/\mathbb{Q}}\mid 2\Delta_{K/\mathbb{Q}}$.
\end{observation}
\begin{proof} 
This follows from three easy observations: the primes that ramify in $K$ divide the classical discriminant, no odd prime divides the prime-power-ramification discriminant unless it ramifies in $K$, and the prime-power-ramification discriminant is square-free.
\end{proof}

One can generalize Definition~\ref{def of prime power disc} by defining a ``relative prime-power-ramification discriminant'' for an extension of number fields $L/K$, 
which is an ideal in the ring of integers of $K$, but I do not know any examples of things you can do with the extra generality in the definition.

\begin{theorem}\label{global computation}
Suppose $K/\mathbb{Q}$ is a finite Galois extension, and let $A$ be the ring of
integers of $K$.
Let $\underline{\Delta}_{K/\mathbb{Q}}$ denote the prime-power-ramification discriminant of $K/\mathbb{Q}$ (defined in Definition~\ref{def of prime power disc}).
Then, for each $n\in\mathbb{N}$,
\[ \Ext^{1,2n+1}_{(L^A,L^AB)}(L^A,L^A) \cong 0,\]
and
the $A[\underline{\Delta}^{-1}]$-module 
\[ \Ext^{1,2n}_{(L^A,L^AB)}(L^A,L^A)[\underline{\Delta}^{-1}]\]
is isomorphic to the localized direct sum
\[ \left( \oplus_{\mbox{max'l\ } \mathfrak{p} \subseteq A} \hat{A}_{\mathfrak{p}}/H^{\mathfrak{p}}_n\right)  [\underline{\Delta}^{-1}],\]
i.e., the localization (inverting $\underline{\Delta}$) of the direct sum, over all maximal ideals $\mathfrak{p}$ of $A$,
of the modules $\hat{A}_{\mathfrak{p}}/H^{\mathfrak{p}}_{n}$, where $H^{\mathfrak{p}}_{n}$ is defined as in Lemma~\ref{valuation inequality lemma}.
\end{theorem}
\begin{proof}
First, a few easy observations:
\begin{enumerate}
\item $\Ext_{(L^A,L^AB)}^{n,i}(L^A, L^A)$ vanishes for $i$ odd, since $L^A$ and $L^AB$ are concentrated in even grading degrees.
In particular, $\Ext_{(L^A,L^AB)}^{1,i}(L^A, L^A) \cong 0$ for odd $i$, as claimed.
\item By Theorem~3.2 of~\cite{MR745362}, 
$\Ext_{(L^A,L^AB)}^{i}(L^A, L^A)$ is a torsion $A$-module for all $i>0$.
In particular, $\Ext_{(L^A,L^AB)}^{1}(L^A, L^A)$ is a torsion $A$-module, hence each of its summands $\Ext_{(L^A,L^AB)}^{1,n}(L^A, L^A)$ is a torsion $A$-module.
\item For all $i$, the grading degree $i$ summand $(L^AB)^i$
of $L^AB$ is a finitely generated $A$-module.
\item For all $i$, $\Ext^{1,i}_{(L^A,L^AB)}(L^A, L^A)$ is a subquotient
of $(L^AB)^i$, since the cobar complex of $(L^A,L^AB)$ computes
$\Ext^{1,i}_{(L^A,L^AB)}(L^A, L^A)$.
\end{enumerate}
Consequently $\Ext^{1,i}_{(L^A,L^AB)}(L^A, L^A)$ is a finitely generated
torsion $A$-module for all $i$.
Since $A$ is the ring of integers in a number field, $A$ is a Dedekind domain,
and now the classification of finitely generated modules over a Dedekind domain
tells us that $\Ext^{1,i}_{(L^A,L^AB)}(L^A, L^A)$
splits as a direct sum of finitely many cyclic submodules of the form 
$A/\mathfrak{p}^n$, for various prime ideals $\mathfrak{p}$ of $A$
and various positive integers $n$.

Now we use Pearlman's result. As a consequence of Theorem~\ref{local-to-global conj},
for any prime ideal $\mathfrak{p}$ of $A$, we have an isomorphism of graded $\hat{A}_{\mathfrak{p}}$-modules
\[ \Ext^{*}_{(L^A,L^AB)}(L^A, M)^{\widehat{}}_{\mathfrak{p}}
 \cong \Ext^*_{(V^{\hat{A}_{\mathfrak{p}}},V^{\hat{A}_{\mathfrak{p}}}T)}(V^{\hat{A}_{\mathfrak{p}}}, V^{\hat{A}_{\mathfrak{p}}} \otimes_{L^A} M).\]
(We do not distinguish between $\hat{A}_{{\mathfrak{p}}} \otimes_{A} \Ext^{*}_{(L^A,L^AB)}(L^A, M)$ and
$\Ext^{*}_{(L^A,L^AB)}(L^A, M)^{\widehat{}}_{\mathfrak{p}}$ because $A$ is Noetherian and $\Ext^{*}_{(L^A,L^AB)}(L^A, M)$
is a finitely generated $A$-module in each bidegree $\Ext^{i,j}$, for the same reasons
described above for $\Ext^{1,i}_{(V^A,V^AT)}(V^A, V^A)$;
consequently ${\mathfrak{p}}$-adic completion on $\Ext^{*}_{(L^A,L^AB)}(L^A, M)$ coincides with 
taking the tensor product $\hat{A}_{\mathfrak{p}}\otimes_A \Ext^{*}_{(L^A,L^AB)}(L^A, M)$.
See Proposition~10.13 of~\cite{MR0242802} for the classical result on comparison of adic completion of modules to
the base-change to the adic completion of the base ring.)

Now each cyclic summand $A/\underline{q}^n$ of $\Ext^{1,2i}_{(V^A,V^AT)}(V^A, V^A)$ is 
annihilated by $\mathfrak{p}$-adic completion for all maximal ideals $\mathfrak{p}$ of $A$ except for one, namely, the prime $\mathfrak{p} = \underline{q}$.
Furthermore $A/\underline{q}^n$ is isomorphic to its
own $\underline{q}$-adic completion.
Consequently, for all $i$, the $A$-module $\Ext^{1,2i}_{(L^A,L^AB)}(L^A, L^A)$ splits as a direct sum of its own $\mathfrak{p}$-adic completions:
\begin{align}\label{splitting into completions} \Ext^{1,2i}_{(L^A,L^AB)}(L^A, L^A) &\cong \oplus_{{\mathfrak{p}}} \Ext^{1,2i}_{(L^A,L^AB)}(L^A, L^A)^{\widehat{}}_{\mathfrak{p}}\\
 \label{consequence of pearlmans thm} &\cong \oplus_{{\mathfrak{p}}} \Ext^{1,2i}_{(V^{\hat{A}_{\mathfrak{p}}},V^{\hat{A}_{\mathfrak{p}}}T)}(V^{\hat{A}_{\mathfrak{p}}}, V^{\hat{A}_{\mathfrak{p}}}),
\end{align}
with~\ref{consequence of pearlmans thm} the consequence of Theorem~\ref{local-to-global conj}, as described above.

Consequently, 
since $\Ext^{1,2i}_{(L^A,L^AB)}(L^A, L^A)$ splits as a direct sum of
summands, each of which is $\Ext^{1,2i}_{(V^{\hat{A}_{\mathfrak{p}}},V^{\hat{A}_{\mathfrak{p}}}T)}(V^{\hat{A}_{\mathfrak{p}}}, V^{\hat{A}_{\mathfrak{p}}})$ and hence
$\mathfrak{p}$-adically complete for some
maximal ideal $\mathfrak{p}$ in $A$, the effect
of inverting the prime-power-ramification discriminant $\underline{\Delta}_{K/\mathbb{Q}}$ 
is that the localized $\Ext$ $A$-module
$\Ext^{1,2i}_{(L^A,L^AB)}(L^A, L^A)[\underline{\Delta}_{K/\mathbb{Q}}^{-1}]$
splits as a direct sum 
\[ \Ext^{1,2i}_{(L^A,L^AB)}(L^A, L^A)[\underline{\Delta}_{K/\mathbb{Q}}^{-1}] \cong
 \oplus_{\underline{\Delta}_{K/\mathbb{Q}}\notin {\mathfrak{p}}} \Ext^{1,2i}_{(V^{\hat{A}_{\mathfrak{p}}},V^{\hat{A}_{\mathfrak{p}}}T)}(V^{\hat{A}_{\mathfrak{p}}}, V^{\hat{A}_{\mathfrak{p}}}),\]
with the direct sum taken over all maximal ideals $\mathfrak{p}$ in $A$
which do not contain the prime-power-ramification discriminant $\underline{\Delta}_{K/\mathbb{Q}}$.

Finally, by the cases of the Local Conjecture proved in Theorem~\ref{local conjecture} (and, in particular, the formulation in Corollary~\ref{local computation})
for each $i\in\mathbb{Z}$ and for each maximal ideal $\mathfrak{p}$ in $A$ which sits over a prime $p$ in $\mathbb{Z}$ not in $\underline{\Delta}_{K/\mathbb{Q}}$, 
we have the isomorphism of $A$-modules
\begin{align*} \Ext^{1,2i}_{(V^{\hat{A}_{\mathfrak{p}}},V^{\hat{A}_{\mathfrak{p}}}T)}(V^{\hat{A}_{\mathfrak{p}}}, V^{\hat{A}_{\mathfrak{p}}}) &\cong A/H^{\mathfrak{p}}_i
 \end{align*}

Consequently $\Ext^{1,2i}_{(L^A,L^AB)}(L^A, L^A)[\underline{\Delta}_{K/\mathbb{Q}}^{-1}]$ is isomorphic, as an $A$-module, to the direct product of the quotients
$\prod_{\mathfrak{p}} A/H^{\mathfrak{p}}_{i}$, for all maximal ideals $\mathfrak{p}$ in $A$ which do not contain $\underline{\Delta}_{K/\mathbb{Q}}$.
Inverting $\underline{\Delta}_{K/\mathbb{Q}}$ then yields the claimed isomorphism.
\end{proof}

\begin{corollary}\label{weak form of global conj holds}
The weaker form of Ravenel's Global Conjecture, Conjecture~\ref{weaker rigorous global conj}, is true.

Specifically: let $K/\mathbb{Q}$ be a finite field extension with ring of integers $A$.
Then there exists some number $c\in \mathbb{N}$ such that,
for all $m\in \mathbb{N}$,
\[ \Ext_{(L^A,L^AB)}^{1,2m}(L^A,L^A)[c^{-1}] \cong A/(J_m)[c^{-1}],\]
where $J_m$ is the minimal $m$-congruing ideal of $A$.
The number $c$ can be taken to be the prime-power-ramification discriminant,
$c = \underline{\Delta}_{K/\mathbb{Q}}$.

Equivalently, in terms of the moduli stack $\mathcal{M}_{fmA}$ of one-dimensional formal $A$-modules over $\Spec A$, and with notation as in Conventions~\ref{running conventions}:
there exists some number $c\in \mathbb{N}$ such that,
for all $m\in \mathbb{N}$,
\[ H^{1,2m}_{fl}(\mathcal{M}_{fmA}; \mathcal{O})[c^{-1}] \cong A/(J_m)[c^{-1}],\]
where $J_m$ is the minimal $m$-congruing ideal of $A$.
The number $c$ can be taken to be the prime-power-ramification discriminant,
$c = \underline{\Delta}_{K/\mathbb{Q}}$.
\end{corollary}
\begin{proof}
To get this result from Theorem~\ref{global computation}, 
just apply Theorem~\ref{hasse principle} to the Minkowski Dedekind domain
$A[\Delta_{K/\mathbb{Q}}^{-1}]$ (the fact that we have to localize $A$ is the reason
we phrased Theorem~\ref{hasse principle} in terms of Minkowski Dedekind domains,
e.g. localized number rings, rather than just number rings!), to get
that $\prod_{\mathfrak{p}} A/H^{\mathfrak{p}}_{i}$ is isomorphic to 
$A/J_i$, $A$ modulo the maximal $i$-congruing ideal in $A$.
\end{proof}

\begin{corollary}\label{strong form of global conj holds}
The strong form of Ravenel's Global Conjecture, Conjecture~\ref{rigorous global conj}, is true for number fields $K$ whose prime-power-ramification discriminant $\underline{\Delta}_{K/\mathbb{Q}}$ is equal to one. The ``correcting factor''
in the statement of Conjecture~\ref{rigorous global conj} can be taken to be one (the strongest possible, that is, no correcting factor is necessary!).

Specifically: 
let $K/\mathbb{Q}$ be a finite field extension with ring of integers $A$.
Suppose that $\underline{\Delta}_{K/\mathbb{Q}} = 1$, i.e.,
suppose that, for every maximal ideal $\mathfrak{p}$ in $A$,
the number $\log_p(\frac{e}{p-1})$ is not an integer, where $p$ is the prime of $\mathbb{Z}$ under $\mathfrak{p}$, and $e$ is the ramification degree of $\mathfrak{p}$.
Then, for all $m\in \mathbb{N}$,
\[\Ext_{(L^A,L^AB)}^{1,2m}(L^A,L^A) \cong A/J_m,\]
where $J_m$ is the minimal $m$-congruing ideal of $A$.
\end{corollary}

The next two corollaries recall the classes of examples, from Example~\ref{examples of ppr discriminants}, in which the prime-power-ramification discriminant is one, and hence Corollary~\ref{strong form of global conj holds} applies.
\begin{corollary}
Suppose $K/\mathbb{Q}$ is a Galois extension of odd primary degree,
and suppose that $2$ ramifies in $K$. Let $A$ be the ring of integers of $K$.
Then, for all $m\in \mathbb{N}$,
\[\Ext_{(L^A,L^AB)}^{1,2m}(L^A,L^A) \cong A/J_m,\]
where $J_m$ is the minimal $m$-congruing ideal of $A$.

Equivalently, in terms of the moduli stack $\mathcal{M}_{fmA}$ of one-dimensional formal $A$-modules over $\Spec A$, and with notation as in Conventions~\ref{running conventions}:
for all $m\in \mathbb{N}$,
\[ H^{1,2m}_{fl}(\mathcal{M}_{fmA}; \mathcal{O}) \cong A/J_m,\]
where $J_m$ is the minimal $m$-congruing ideal of $A$.
\end{corollary}

\begin{corollary}
A nonabelian example: let $K$ be the splitting field of $x^5+x+2$ over $\mathbb{Q}$.
Let $A$ be the ring of integers of $K$.
Then, for all $m\in \mathbb{N}$,
\[\Ext_{(L^A,L^AB)}^{1,2m}(L^A,L^A) \cong A/J_m,\]
where $J_m$ is the minimal $m$-congruing ideal of $A$.

Equivalently,
for all $m\in \mathbb{N}$,
\[ H^{1,2m}_{fl}(\mathcal{M}_{fmA}; \mathcal{O}) \cong A/J_m,\]
where $J_m$ is the minimal $m$-congruing ideal of $A$.
\end{corollary}

\section{Zeta-functions.}

\subsection{Motivation.}

In~\cite{MR745362}, Ravenel writes: ``The numbers $j_m$ of 3.8 are also related to Bernoulli numbers and the values of
the Riemann zeta function at negative integers, but these properties do not appear
to generalize to other number fields. For example if the field is not totally real its
Dedekind zeta function vanishes at all negative integers.''
Ravenel's observation, that a direct connection between the order of $\Ext_{(L^A,L^AB)}^{1,2m}(L^A,L^A)$ and $\zeta_K(1-m)$ is impossible
unless $K$ is totally real since otherwise $\zeta_K$ vanishes at negative integers, is quite correct.

However, a good relationship between numerical invariants of $\zeta_K$ and the order of $\Ext_{(L^A,L^AB)}^{1,2m}(L^A,L^A)$
is quite possible, if one is willing to work with invariants of $\zeta$-functions other than special values.
The idea is as follows: in an Euler product admitting an analytic continuation, 
the $p$-local Euler factor typically controls prime-to-$p$ denominators of special values at negative integers.
One wants to apply some kind of transform to such an Euler product which ``straightens out'' the relationship between the $p$-local Euler factors and the
$p$-primary factors in the special values of the function, causing the $p$-local Euler factor to contribute to 
$p$-primary factors, rather than prime-to-$p$ factors, in the values of the transformed function. 

In the remaining sections of this paper, I give two distinct ways of doing this,
which I call ``unramified straightening'' and ``Galois-Dedekind straightening.''
\begin{itemize}
\item The ``unramified straightening transform'' $\mathbb{S}(L(s))(n)$ of an
$L$-function $L(s)$ has the appeal of having an extremely simple
and natural definition. On the other hand, it has the disadvantage that,
when $A$ is the ring of integers of a finite extension $K/\mathbb{Q}$
with Dedekind $\zeta$-function $\zeta_K(s)$, the unramified straightening
transform
$\mathbb{S}(\zeta_K(s))(n)$ only describes the order
of $\Ext^{1,2n}_{(L^A,L^AB)}(L^A,L^A)$ after inverting two times 
the classical discriminant $\Delta_{K/\mathbb{Q}}$.
\item On the other hand, the ``Galois-Dedekind straightening transform'' $\mathbb{S}_{GD}(L(s))(n)$ 
is only defined for certain (sufficiently ``nice'') $L$-functions
$L(s)$, and its definition is more complicated and less natural-seeming
than the unramified straightening transform. On the other hand, the Galois-Dedekind straightening transform has the advantage that,
when $A$ is the ring of integers of a finite {\em Galois} extension $K/\mathbb{Q}$
with Dedekind $\zeta$-function $\zeta_K(s)$, the Galois-Dedekind straightening
transform
$\mathbb{S}_{GD}(\zeta_K(s))(n)$ is defined, and it describes the order
of $\Ext^{1,2n}_{(L^A,L^AB)}(L^A,L^A)$ after inverting the prime-power-ramification discriminant $\underline{\Delta}_{K/\mathbb{Q}}$, which is a less destructive operation
than inverting two times the classical discriminant.
\end{itemize}

As I mentioned in the introduction to this paper, these straightening 
transforms also
have some applications elsewhere, in the theory of $L$-functions associated
to Bousfield-localized stable homotopy types of finite CW-complexes, which I 
will say much more about in a later paper, currently in preparation.
The definitions of these straightening transforms are chosen in order to have 
convenient properties which relate them to $\Ext^{1}_{(L^A,L^AB)}(L^A,L^A)$
and to stable homotopy types of finite CW-complexes; I do not know
if these straightening transforms are of any interest or use for number theorists.

\subsection{The unramified straightening transform.}

\begin{definition}\label{def of unramified straightening transform}
If $p$ is a prime number, then by an {\em elementary $p$-local Euler factor} I mean an expression of the form 
\[ \frac{1}{1 - p^{-as}}\]
for some positive integer $a$. By an {\em elementary Euler factor} I mean elementary $p$-local Euler factor for some prime number $p$.

Suppose $L(s)$ is an $L$-function which is equal (for $s$ with sufficiently large real part) to an Euler product which is a product of elementary Euler factors.
By the {\em unramified straightening transform of $L(s)$} I mean the function $\mathbb{S}(L(s)): \mathbb{N} \rightarrow \mathbb{N}$
given as follows: 
\[\mathbb{S}\left( \frac{1}{1-p^{-as}}\right)(n) = \left\{ 
 \begin{array}{ll} 
 1 &\mbox{\ if\ } p^a-1\nmid n \\
 p^{(1+\nu_p(n))a} & \mbox{\ if\ } p^a-1\mid n,\end{array}\right. \]
and $\mathbb{S}$ of a product of elementary Euler factors is the product of $\mathbb{S}$ applied to each Euler factor.
\end{definition}
Obviously this definition is limited to only a very restricted class of $L$-functions, namely, those 
with an Euler product which is a product of elementary Euler factors.
In a later paper (currently in preparation) I extend the definition of this straightening transform to 
a larger class of Euler products, and I demonstrate that this slightly more general unramified $\mathbb{S}$-transform 
has some agreeable properties (e.g. recovering the orders of $KU$-local stable homotopy groups)
when evaluated on certain $L$-functions associated to finite CW-complexes. 

The following proposition, Proposition~\ref{riemann special value denominators},
is equivalent to a computation of the denominators of the numbers
$\zeta(1-n)$ for positive integers $n$. This computation is
very well-known: in the terms most familiar to a topologist, it is the fact that 
the orders of the homotopy groups $\pi_{*}(j)$ in the image of the $J$-homomorphism in degrees congruent to $-1$ mod $4$ (which Adams identified, in~\cite{MR0198470}, with the denominators of the special values at negative integers of the Riemann $\zeta$-function) have the correct $p$-adic valuations to coincide, after $p$-adic
completion for $p>2$,
with the familiar pattern one sees in the homotopy groups of the $K(1)$-local (or $E(1)$-local) sphere, computed by Adams and Baird (see~\cite{MR737778}).
But of course Proposition~\ref{riemann special value denominators} was known to number theorists much, much earlier. I do not know where to find a proof in the literature, so I supply one, below.
\begin{prop}\label{riemann special value denominators}
Let $\zeta$ be the Riemann $\zeta$-function.
Then, for all positive integers $n$, the denominator of $\zeta(1-2n)$ is equal to $\mathbb{S}(\zeta(s))(2n)$.

Suppose further that $m$ is an odd positive integer. Then $\zeta(1-m) = 0$, so it is a matter of convention what the denominator of $\zeta(1-m)$ is.
However, $\mathbb{S}(\zeta(s))(m) = 2$ for all odd positive integers $m$, 
so if we adopt the convention that the denominator of $\zeta(1-m) = 0$ is $2$, then the denominator of $\zeta(1-n)$ is equal to 
$\mathbb{S}(\zeta(s))(n)$ for all positive integers $n$.
\end{prop}
\begin{proof}
First, recall that the Euler product for $\zeta$ is:
\[ \zeta(s) = \prod_{p} \frac{1}{1-p^{-s}}\]
when $\Re s > 1$.
Consequently, for each prime $p$, the $p$-adic valuation of $\mathbb{S}(\zeta(s))(n)$ is given by
\begin{align} \nu_p(\mathbb{S}(\zeta(s))(n)) = 
\label{valuation formula for S-transform} \left\{ \begin{array}{ll} 
  0 & \mbox{\ if\ } (p-1) \nmid n \\
  1 + \nu_p(n) & \mbox{\ if\ } (p-1)\mid n.\end{array}\right. \end{align}

Now we will show that the $p$-adic valuation of $\mathbb{S}(\zeta(s))(2n)$ agrees with the $p$-adic valuation of 
the denominator of $\zeta(1-2n)$. 
It is classical that 
\begin{equation}\label{zeta and bernoullis} \zeta(1-n) = \frac{-B_n}{n}\end{equation}
for positive integers $n$, and that $B_n = 0$ if $n$ is odd and $n>1$.
By the von Staudt-Clausen theorem, the denominator of the Bernoulli number $B_{2n}$ is equal to the product of all primes
$p$ such that $p-1$ divides $2n$. 
Suppose that $p$ is a prime and $n$ a positive integer such that $(p-1)\nmid 2n$.
Let $N_{2n}$ denote the numerator and $D_{2n}$ the denominator of $B_{2n}$.
Then it is known that 
\[\nu_p(2n)\leq \nu_p(N_{2n}),\] that is, the $p$-adic valuation
of $N_{2n}$ is at least as large as the $p$-adic valuation of $2n$;
\cite{MR1041890}
provides, in the introduction, a brief history of this result,
which goes back to von Staudt but has been independently rediscovered many
times. 
Consequently, if $p-1$ does not divide $2n$,
then $D_{2n}$ is not divisible by $p$, and 
$N_{2n}$ is divisible by at least as large a power of $p$
as $2n$ is. Hence the denominator of $\frac{B_{2n}}{2n}$ is not divisible by $p$.

On the other hand, suppose that $p$ is a prime and $n$ a positive integer
such that $(p-1)\mid 2n$. Then $D_{2n}$ {\em is} divisible
by $p$, and hence $N_{2n}$ {\em isn't} divisible by $p$;
consequently we have the equality of $p$-adic valuations
\begin{align*} \nu_p\left(\frac{B_{2n}}{2n}\right) &= -\nu_p\left(D_{2n}\right) - \nu_p(2n) \\
 &= -\nu_p(2n) - 1.\end{align*}

Consequently, using the equation~\ref{zeta and bernoullis}, 
if we let $D_{2n}^{\prime}$ denote the denominator of
$\zeta(1-2n)$, then 
for all primes $p$ and positive integers $n$ we get the formula:
\begin{align*}
 \nu_p\left( D_{2n} \right) &= 
 \left\{ \begin{array}{ll} 
   0 & \mbox{\ if\ } (p-1)\nmid 2n \\
   \nu_p(2n) +1 \mbox{\ if\ } (p-1)\mid 2n.\end{array} \right. 
\end{align*}
This is equal to $\mathbb{S}(\zeta(s))(2n)$ by equation~\ref{valuation formula for S-transform}.

If $m$ is odd, we already observed that $\zeta(1-m) = 0$ since the odd Bernolli numbers $B_m$ are zero. 
On the other hand, for odd $m$, there exists a (unique) prime number $p$ such that $(p-1)\mid m$, namely, $p=2$.
Furthermore, $1 + \nu_2(m) = 1$ for odd $m$,
so by equation~\ref{valuation formula for S-transform}, we have that $\mathbb{S}(\zeta(s))(m) = 2$ for all positive odd integers $m$,
as claimed.
\end{proof}

Now we need to review a few facts about Hecke $L$-functions. 
We will only need the Hecke $L$-functions of trivial Gro{\ss}encharakters (with varying conductors, however),
so the treatment I give here is much more limited than what one can find in a good number theory textbook, such as~\cite{MR1697859}.
\begin{definition-proposition}{\bf (Classical.)}
Let $K/\mathbb{Q}$ be a number field with ring of integers $A$, and let $m$ be a proper ideal of $A$. 
We write $\chi$ for the trivial Gro{\ss}encharakter of $K$ with conductor $m$. 
By the {\em Hecke $L$-function of $\chi$} we mean the complex-valued function
\[ L(s,\chi) = \sum_{I} \frac{1}{N(I)^{-s}},\]
for all complex numbers $s$ with real part $\Re s > 1$, 
where $I$ ranges across all nonzero ideals of $A$ which are coprime to $m$,
and $N(I)$ is the absolute norm of $I$, i.e., $N(I)$ is the cardinality of the residue ring $A/I$.

Then $L(s,\chi)$ admits analytic continuation to a meromorphic function on the complex plane,
and
$L(s,\chi)$ admits an Euler product 
\[ L(s,\chi) = \prod_{\mathfrak{p}} \frac{1}{1 - N(\mathfrak{p})^{-s}}\]
for all complex numbers $s$ with real part $\Re s > 1$,
where the product is taken over all maximal ideals $\mathfrak{p}$ of $A$ which are coprime to $m$.

Hence 
\[ L(s,\chi) = \left(\prod_{\mathfrak{p}\supseteq m}\left(1 - N(\mathfrak{p})^{-s}\right)\right) \zeta_K(s),\]
where $\zeta_K(s)$ is the Dedekind $\zeta$-function of $K$, and the product is taken over all
maximal ideals $\mathfrak{p}$ of $A$ which contain $m$.
In particular, if $m = 1$,
then $L(s,\chi) = \zeta_K(s)$.
\end{definition-proposition}
So Hecke $L$-functions of trivial Gro{\ss}encharakters are a convenient ``language'' for working with 
Dedekind $\zeta$-functions in which certain Euler factors---those associated to the primes containing the conductor ideal---have been omitted.

From Proposition~\ref{riemann special value denominators}, one knows that the unramified $\mathbb{S}$-transform of the Riemann $\zeta$-function recovers
the denominators of the special values of $\zeta(s)$ at negative integers, and consequently (by the classical $K=\mathbb{Q}$ case of Theorem~\ref{global computation}
which motivated Ravenel's Global Conjecture), it recovers the order of $\Ext_{(L,LB)}^{1,*}(L,L)$ up to multiplication by a power of $2$. 
A more remarkable fact is that, for the Dedekind $\zeta$-function $\zeta_K(s)$ of
a number field $K/\mathbb{Q}$, the unramified $\mathbb{S}$-transform $\mathbb{S}(\zeta_K(s))$ {\em does not} recover the 
denominators of the special values of $\zeta(s)$ at negative integers, but instead,
if the discriminant {\em and} the prime-power-ramification discriminant of $K/\mathbb{Q}$ are both trivial,
then $\mathbb{S}(\zeta_K(s))$ recovers the order of $\Ext_{(L^A,L^AB)}^{1,*}(L^A,L^A)$, where $A$ is the ring of integers in $K$.
In fact a stronger statement is true, and
here is a precise statement (and proof):
\begin{theorem}\label{main thm for unramified straightening}
Let $K/\mathbb{Q}$ be a finite field extension with ring of integers $A$.
Let $\chi_{2\Delta_{K/\mathbb{Q}}}$ be the trivial Gro\ss{e}ncharakter of $K$ of conductor equal to $2\Delta_{K/\mathbb{Q}}$, two times the classical discriminant
of $K/\mathbb{Q}$, and let
$L(s,\chi_{2\Delta_{K/\mathbb{Q}}})$ be its associated Hecke $L$-function.

Then, for all positive $n\in\mathbb{N}$, the following numbers are all equal:
\begin{itemize}
\item the order of the abelian group $\Ext^{1,2n}_{(L^A,L^AB)}(L^A,L^A)[(2\Delta_{K/\mathbb{Q}})^{-1}]$,
\item the order of the abelian group $H^{1,2n}_{fl}(\mathcal{M}_{fmA}; \mathcal{O})[(2\Delta_{K/\mathbb{Q}})^{-1}]$,
\item the order of the abelian group $A[(2\Delta_{K/\mathbb{Q}})^{-1}]/J_{n}$, where $J_n$ is the minimal $n$-congruing ideal in $A[(2\Delta_{K/\mathbb{Q}})^{-1}]$, and
\item the number $\mathbb{S}(L(s,\chi_{2\Delta_{K/\mathbb{Q}}}))(n)$.
\end{itemize}
\end{theorem}
\begin{proof}
First, since $(\underline{\Delta}_{K/\mathbb{Q}}))\supseteq ({2\Delta_{K/\mathbb{Q}}})$ by Observation~\ref{discriminant divisibility}, 
inverting ${2\Delta_{K/\mathbb{Q}}}$ implies also inverting $\underline{\Delta}_{K/\mathbb{Q}}$,
so Corollary~\ref{weak form of global conj holds} implies that we have an isomorphism of $A$-modules
\[ \Ext^{1,2n}_{(L^A,L^AB)}(L^A,L^A)[(2\Delta_{K/\mathbb{Q}})^{-1}] \cong A[(2\Delta_{K/\mathbb{Q}})^{-1}]/J_{n},\]
consequently these two abelian groups have equal order. So for the rest of the proof we will concern ourselves with the question of why
the order of $\Ext^{1,2n}_{(L^A,L^AB)}(L^A,L^A)[(2\Delta_{K/\mathbb{Q}})^{-1}]$ is equal to $\mathbb{S}(L(s,\chi_{2\Delta_{K/\mathbb{Q}}}))(n)$.

Now we need to see why $\Ext^{1,2n}_{(L^A,L^AB)}(L^A,L^A)[(2\Delta_{K/\mathbb{Q}})^{-1}]$ has finite order!
We know that $\Ext^{1,2n}_{(L^A,L^AB)}(L^A,L^A)$ is a finitely generated torsion $A$-module (see the beginning of the proof of Theorem~\ref{global computation}),
hence has finite order. Furthermore, by the fundamental theorem of finitely generated modules over a Dedekind domain, 
$\Ext^{1,2n}_{(L^A,L^AB)}(L^A,L^A)$ is a direct sum of finitely many $A$-modules of the form $A/\mathfrak{p}^i$,
for various prime ideals $\mathfrak{p}$ of $A$ and various positive integers $i$.
Consequently, the effect of inverting ${2\Delta_{K/\mathbb{Q}}}$ on $\Ext^{1,2n}_{(L^A,L^AB)}(L^A,L^A)$ is the same as quotienting out those summands of $A$ of the form
$A/\mathfrak{p}^i$ where ${2\Delta_{K/\mathbb{Q}}}\in \mathfrak{p}$. So $\Ext^{1,2n}_{(L^A,L^AB)}(L^A,L^A)[(2\Delta_{K/\mathbb{Q}})^{-1}]$ is a quotient $A$-module of the finite
$A$-module $\Ext^{1,2n}_{(L^A,L^AB)}(L^A,L^A)$. So $\Ext^{1,2n}_{(L^A,L^AB)}(L^A,L^A)[(2\Delta_{K/\mathbb{Q}})^{-1}]$ is itself finite.

Now clearly $\mathbb{S}(L(s,\chi_{2\Delta_{K/\mathbb{Q}}}))(n)$ and the order of $\Ext^{1,2n}_{(L^A,L^AB)}(L^A,L^A)[(2\Delta_{K/\mathbb{Q}})^{-1}]$ are both positive integers,
so if we show that the $p$-adic valuation of $\mathbb{S}(L(s,\chi_{2\Delta_{K/\mathbb{Q}}}))(n)$ agrees with the $p$-adic valuation of the order of
$\Ext^{1,2n}_{(L^A,L^AB)}(L^A,L^A)[(2\Delta_{K/\mathbb{Q}})^{-1}]$ for all prime numbers $p$, then we are done.

Suppose $p$ is a prime number such that some prime $\mathfrak{p}$ of $A$ over $p$ is ramified. 
Then $\mathfrak{p} \supseteq (\Delta_{K/\mathbb{Q}})\supseteq ({2\Delta_{K/\mathbb{Q}}})$,
and consequently the Euler product for $L(s, \chi_{2\Delta_{K/\mathbb{Q}}})$ contains no $p$-local Euler factors. Hence $\nu_p\left(\mathbb{S}(L(s,\chi_{2\Delta_{K/\mathbb{Q}}}))(n)\right) = 0$.
On the other hand, 
for each prime $\underline{\ell}$ of $A$ over a given prime number $\ell$, the completion $\left(\Ext^{1,2n}_{(L^A,L^AB)}(L^A,L^A)\right)^{\hat{}}_{\underline{\ell}}$ is 
a finite $\hat{A}_{\underline{\ell}}$-module, hence has order equal to some power of $\ell$, and $\Ext^{1,2n}_{(L^A,L^AB)}(L^A,L^A)$
splits as a direct sum of its completions at the various maximal ideals $\mathfrak{p}$ of $A$ (this is isomorphism~\ref{splitting into completions} from the
proof of Theorem~\ref{global computation}).
Consequently $\Ext^{1,2n}_{(L^A,L^AB)}(L^A,L^A)[(2\Delta_{K/\mathbb{Q}})^{-1}]$ has no summands of order a power of $p$.
Consequently the $p$-adic valuation of the order of $\Ext^{1,2n}_{(L^A,L^AB)}(L^A,L^A)[(2\Delta_{K/\mathbb{Q}})^{-1}]$ is zero.

The more important situation is when $p$ is a prime number such that no primes of $A$ over $p$ are ramified. Suppose $p$ is such a prime.
Then the $p$-adic valuation of the order of $\Ext^{1,2n}_{(L^A,L^AB)}(L^A,L^A)[(2\Delta_{K/\mathbb{Q}})^{-1}]$
is equal to the $p$-adic valuation of the order of the completion $(\Ext^{1,2n}_{(L^A,L^AB)}(L^A,L^A))^{\hat{}}_p$, i.e.,
the product of the $p$-adic valuations of the orders of the completions $(\Ext^{1,2n}_{(L^A,L^AB)}(L^A,L^A))^{\hat{}}_{\mathfrak{p}}$ for all primes
$\mathfrak{p}$ of $A$ over $p$.
By Theorem~\ref{local-to-global conj}, the completion $(\Ext^{1,2n}_{(L^A,L^AB)}(L^A,L^A))^{\hat{}}_{\mathfrak{p}}$ is isomorphic as a $\hat{A}_{\mathfrak{p}}$-module
to $\Ext^{1,2n}_{(V^{\hat{A}_{\mathfrak{p}}},V^{\hat{A}_{\mathfrak{p}}}T)}(V^{\hat{A}_{\mathfrak{p}}},V^{\hat{A}_{\mathfrak{p}}})$,
which is in turn isomorphic, by Corollary~\ref{local computation}, to $\hat{A}_{\mathfrak{p}}/H_n$, where $H_n$ is the ideal of $A_{\mathfrak{p}}$
generated by all elements of the form $x^n-1$, where $x\in \hat{A}_{\mathfrak{p}}^{\times}$.
As we showed in Theorem~\ref{local conjecture}, $\hat{A}_{\mathfrak{p}}/H_n \cong 0$ if $p^f-1 \nmid n$ where $p^f$ is the cardinality of the residue field
of $\hat{A}_{\mathfrak{p}}$, and if $p^f-1\mid n$, then $H_n = I_{n/(p^f-1)}$, where $I_{n/(p^f-1)}$ is the ideal in $\hat{A}_{\mathfrak{p}}$ generated by all
elements of the form $a^{n/(p^f-1)}-1$ for elements $a\in \hat{A}_{\mathfrak{p}}$ congruent to $1$ modulo $\mathfrak{p}$ (this is by the Hensel's Lemma
argument at the end of the proof of Corollary~\ref{local computation}).

Finally, Lemma~\ref{valuations predicted by local conjecture} gives us that the order of
$\hat{A}_{\mathfrak{p}}/I_{n/(p^f-1)}$ is equal to $p^{fj}$, where
\begin{align*} j 
 &= 
  \left\{ \begin{array}{ll} 
   e\left(\nu_p(n/(p^f-1)) - \ceiling{\log_p \frac{e}{p-1}}\right) + p^{\ceiling{\log_p\frac{e}{p-1}}} &\mbox{\ if\ } \ceiling{\log_p\frac{e}{p-1}} \leq \nu_p(n/(p^f-1)) \\
   p^{\nu_p(n/(p^f-1))} &\mbox{\ if\ } \ceiling{\log_p\frac{e}{p-1}} > \nu_p(n/(p^f-1)) \end{array}\right. \\
 &= 
   \nu_p(n)  + 1,
\end{align*}
since by assumption $\mathfrak{p}$ is not ramified and hence $e=1$.
Consequently our formula for the $p$-adic valuation of $\Ext^{1,2n}_{(L^A,L^AB)}(L^A,L^A)[(2\Delta_{K/\mathbb{Q}})^{-1}]$ is:
\begin{equation*}\label{p-adic valuation of order of Ext1} \nu_p\left( \#\left( \Ext^{1,2n}_{(L^A,L^AB)}(L^A,L^A)[(2\Delta_{K/\mathbb{Q}})^{-1}] \right) \right) = \sum_{\mathfrak{p}} \left( (\nu_p(n)+1)f_{\mathfrak{p}}\right),\end{equation*}
where the sum is taken over all primes $\mathfrak{p}$ of $A$ over $p$ such that
the cardinality $p^{f_{\mathfrak{p}}}$ of the residue field $A/\mathfrak{p}$ satisfies $(p^{f_{\mathfrak{p}}}-1) \mid n$.

Now we turn to the $p$-adic valuation of $\mathbb{S}(L(s,\chi_{2\Delta_{K/\mathbb{Q}}}))(n)$.
This follows easily from the definition of $\mathbb{S}$ and the Euler product for $L(s,\chi_{2\Delta_{K/\mathbb{Q}}})$:
\begin{equation*}\label{p-adic valuation of order of Ext1 2} \nu_p\left( \mathbb{S}(L(s, \chi_{2\Delta_{K/\mathbb{Q}}}))(n) \right) = \sum_{\mathfrak{p}} \left( (\nu_p(n)+1)f_{\mathfrak{p}}\right),\end{equation*}
where the sum is taken over all primes $\mathfrak{p}$ of $A$ over $p$ such that
the cardinality $p^{f_{\mathfrak{p}}}$ of the residue field $A/\mathfrak{p}$ satisfies $(p^{f_{\mathfrak{p}}}-1) \mid n$,
i.e., 
\[ \nu_p\left( \mathbb{S}(L(s, \chi_{2\Delta_{K/\mathbb{Q}}}))(n) \right) = \nu_p\left( \#\left( \Ext^{1,2n}_{(L^A,L^AB)}(L^A,L^A)[(2\Delta_{K/\mathbb{Q}})^{-1}] \right) \right),\]
as desired.
\end{proof}

\begin{definition}\label{def of arithmetic equivalence}
Recall that two number fields $K_1/\mathbb{Q}$ and $K_2/\mathbb{Q}$ are said to be {\em arithmetically equivalent} if their Dedekind $\zeta$-functions
$\zeta_{K_1}(s)$ and $\zeta_{K_2}(s)$ are equal. 
\end{definition}
See~\cite{MR2041074}
for some nice examples of arithmetically equivalent number fields; in particular, Bosma and de Smit show that the smallest $n$ such that there exist distinct arithmetically equivalent number fields $K_1,K_2$ of degree $n$ over $\mathbb{Q}$ is $n=7$.

If $K_1/\mathbb{Q}$ and $K_2/\mathbb{Q}$ are number fields, a weaker notion than arithmetic equivalence is to ask whether the Dedekind $\zeta$-functions of $K_1$ and of $K_2$ agree when certain Euler factors are ignored. In other words:
\begin{definition}\label{def of arithmetic equivalence mod m}
Let $K_1,K_2$ be number fields, and let $m$ be a positive integer. We will say that {\em $K_1$ and $K_2$ are arithmetically equivalent modulo $m$} if 
the Hecke $L$-function of the trivial Gro{\ss}encharakter on $K_1$ with conductor $m$ is equal to the 
the Hecke $L$-function of the trivial Gro{\ss}encharakter on $K_2$ with conductor $m$.
\end{definition}

Now we have the following corollary of Theorem~\ref{main thm for unramified straightening}:
\begin{corollary}
Let $K_1/\mathbb{Q}$ and $K_2/\mathbb{Q}$ be finite field extensions with ring of integers $A_1$ and $A_2$, respectively.
Let $m$ be any integer which is divisible by $2, \Delta_{K_1/\mathbb{Q}}$, and $\Delta_{K_2/\mathbb{Q}}$.
If $K_1$ and $K_2$ are arithmetically equivalent modulo $m$, then
for all positive $n\in\mathbb{N}$, 
the order of the abelian group
\[ \Ext^{1,2n}_{(L^{A_1},L^{A_1}B)}(L^{A_1},L^{A_1})[m^{-1}]\]
is equal to the order of the abelian group
\[ \Ext^{1,2n}_{(L^{A_2},L^{A_2}B)}(L^{A_2},L^{A_2})[m^{-1}].\]

Equivalently, in terms of the moduli stack $\mathcal{M}_{fmA}$ of one-dimensional formal $A$-modules over $\Spec A$, and with notation as in Conventions~\ref{running conventions}:
the order of the abelian group
\[ H^{1,2n}_{fl}(\mathcal{M}_{fmA_1}; \mathcal{O})[m^{-1}]\]
is equal to the order of the abelian group
\[ H^{1,2n}_{fl}(\mathcal{M}_{fmA_2}; \mathcal{O})[m^{-1}].\]
\end{corollary}

\begin{remark}\label{remark on S-transform and tl}
I do not know if the $\mathbb{S}$-transform is of any independent interest for number theorists (it is a very naive construction, so perhaps not!), but I do know of another topological situation in which the $\mathbb{S}$-transform is very useful, namely, in the problem of  
associating an $L$-function $L(s,X)$ to a finite CW-complex $X$ in such a way that the special values $L(-n,X)$ recover the orders of the homotopy groups of
various Bousfield localizations (say, the $KU$-localization) of $X$: this problem is in general not solvable (except for very particular choices of $X$) if one
insists on working with special values, but if one instead asks that the values $\mathbb{S}(L(s,X))(n)$ of the $\mathbb{S}$-transform coincide
with the orders of the homotopy groups of the $KU$-localization of $X$, then this problem is very tractable, and I will give my solution to this problem in a paper
which is currently in preparation.
\end{remark}

\subsection{The Galois-Dedekind straightening transform.}

Theorem~\ref{main thm for unramified straightening}, identifying the values of the unramified straightening transform of a Hecke $L$-function of $K/\mathbb{Q}$
in terms of orders of $\Ext^1_{(L^A,L^AB)}$ groups after inverting the conductor of the Gro{\ss}encharakter, is an interesting result but there is something
not totally satisfying about it: the conductor of the Gro{\ss}encharakter is required to be divisible by twice the (classical) discriminant of $K/\mathbb{Q}$, meaning
that the result always requires that we invert at least one prime (2) as well as any primes that ramify in $K$. It would be nice to have a generalization of
Theorem~\ref{main thm for unramified straightening} which holds without having to invert so much, in particular, one which holds for at least some number fields
$K$ without having to invert {\em any} primes at all.

In this section I state and prove Theorem~\ref{main thm for g-d straightening}, 
which is just such a generalization of Theorem~\ref{main thm for unramified straightening}. Theorem~\ref{main thm for g-d straightening} comes at a price, however: 
it is phrased in terms of the Galois-Dedekind straightening transform (also defined in this section), 
which is significantly less simple and natural-looking than the unramified straightening transform.

Recall that elementary Euler factors were defined in Definition~\ref{def of unramified straightening transform}.
\begin{definition}\label{def of g-d l-function}
Suppose $L(s)$ is an $L$-function which is equal (for $s$ with sufficiently large real part) to an Euler product which is a product of elementary Euler factors.
I will say that the Euler product is {\em of Galois-Dedekind $d$-type} if the following two conditions hold:
\begin{itemize}
\item {\em (Galois axiom:)} For each prime number $p$, 
if 
\[ \frac{1}{1 - p^{-as}} \mbox{\ \ and\ \ } \frac{1}{1 - p^{-bs}} \]
are factors appearing in the Euler product for $L(s)$, then $a=b$.
\item {\em (Degree axiom:)} For all but finitely many prime numbers $p$,
there exists at least one $p$-local elementary Euler factor in the Euler product for $L(s)$,
and, writing
$\frac{1}{1 - p^{as}}$ for any such $p$-local elementary Euler factor,
we have the equation $an = d$, where $n$ is the number of factors $\frac{1}{1 - p^{as}}$ appearing in the Euler product for $L(s)$.
\end{itemize}

I will say that an Euler product is {\em of Galois-Dedekind type} if it is of Galois-Dedekind $d$-type for some positive integer $d$.
\end{definition}

\begin{example}\label{examples of g-d l-functions}
If $K/\mathbb{Q}$ is a finite extension, then the Dedekind $\zeta$-function $\zeta_K(s)$ is not necessarily of Galois-Dedekind type.
However, if $K/\mathbb{Q}$ is also Galois, then the Dedekind $\zeta$-function {\em is} of Galois-Dedekind type, specifically of
Galois-Dedekind $[K: \mathbb{Q}]$-type.

In fact, a more general statement is true:
suppose $K/\mathbb{Q}$ is a finite field extension which is Galois.
Let $A$ be the ring of integers of $K$, and let $m\in \mathbb{Z}$. Then the Hecke $L$-function $L(s, \chi_m)$ of the 
trivial Gro{\ss}encharakter of $A$ with conductor $m$ is of Galois-Dedekind $[K:\mathbb{Q}]$-type.
The reason for this is as follows: clearly all the factors in the Euler product for $L(s, \chi_m)$ are elementary.
If $p$ is a prime number and $p\mid m$, then $L(s, \chi_m)$ has no $p$-local Euler factors,
so the Galois axiom is automatically satisfied for such $p$.

If $p\nmid m$, then $L(s, \chi_m)$ has $p$-local Euler factors, one for each prime $\mathfrak{p}$ of $A$ over $p$.
Since $K/\mathbb{Q}$ is Galois, the Galois group $G_{K/\mathbb{Q}}$ acts transitively on the set of primes of $A$ over $p$, 
each of which has the same residue degree and ramification degree. Consequently, for fixed $p$, the elementary $p$-local Euler factors
of $L(s, \chi_m)$ are all equal. So the Galois axiom is satisfied for such $p$. 
Furthermore there will be exactly one elementary $p$-local Euler factor $\frac{1}{1 - p^{-f_ps}}$
for each prime $\mathfrak{p}$ of $A$ over $p$, and consequently the number of such factors will be equal to $\frac{[K : \mathbb{Q}]}{e_pf_p}$,
where $e_p$ is the ramification degree and $f_p$ the residue degree of any of the primes $\mathfrak{p}$ over $p$.

However, only finitely many primes $p$ will divide the classical discriminant $\Delta_{K/\mathbb{Q}}$, hence only finitely many primes $p$
will have $e_p> 1$. Hence, for all but finitely many prime numbers $p$ (those finitely many prime numbers $p$ are the ones which either ramify in $A$ 
or which divide the conductor $m$), the number of elementary $p$-local Euler factors $\frac{1}{1 - p^{-f_ps}}$ in $L(s, \chi_m)$ is equal to $\frac{[K: \mathbb{Q}]}{f_p}$.
Hence the degree axiom is satisfied.

In the case $m=1$, this shows that the Dedekind $\zeta$-function $\zeta_K(s)$ of a finite Galois extension $K/\mathbb{Q}$ is
of Galois-Dedekind type $[K: \mathbb{Q}]$.
\end{example}

\begin{definition}\label{def of g-d transform}
Let $d$ be a positive integer and let $L(s)$ be an $L$-function of Galois-Dedekind $d$-type.
By the {\em Galois-Dedekind straightening transform of $L(s)$} I mean the function $\mathbb{S}_{GD}(L(s)): \mathbb{N} \rightarrow \mathbb{N}$
given as follows: 
\[\mathbb{S}_{GD}\left( \frac{1}{1-p^{-as}}\right)(n) = \left\{ 
 \begin{array}{ll} 
 1 &\mbox{\ if\ } p^a-1\nmid n \\
 p^{\frac{d}{g_p}\left( \nu_p(n) - \ceiling{\log_p \frac{d}{ag_p(p-1)}}\right) + ap^{\ceiling{\log_p \frac{d}{ag_p(p-1)}}}} & \mbox{\ if\ } p^a-1\mid n\mbox{\ and\ } \ceiling{\log_p\frac{d}{ag_p(p-1)}} \leq \nu_p(n), \\
 p^{ap^{\nu_p(n)}} & \mbox{\ if\ } p^a-1\mid n\mbox{\ and\ } \ceiling{\log_p\frac{d}{ag_p(p-1)}} > \nu_p(n),
,\end{array}\right. \]
where $g_p$ is the number of elementary $p$-local Euler factors in $L(s)$,
and $\mathbb{S}$ of a product of Euler factors is the product of $\mathbb{S}$ applied to each Euler factor.
\end{definition}

\begin{example}\label{g-d transform of riemann zeta}
Let $\zeta(s)$ be the Riemann $\zeta$-function. Certainly its Galois-Dedekind straightening transform $\mathbb{S}_{GD}(\zeta(s))$ is defined, since
$\zeta(s)$ is a special case of a Dedekind $\zeta$-function, so it is of Galois-Dedekind $1$-type by Example~\ref{examples of g-d l-functions}.

I claim that $\mathbb{S}(\zeta(s)) = \mathbb{S}_{GD}(\zeta(s))$. For each prime number $p$, we will have $g_p = 1$, i.e., $\zeta(s)$ has 
a unique $p$-local elementary Euler factor, and that elementary Euler factor is
\[ \frac{1}{1- p^{-s}},\]
i.e., the number $a$ is $1$ in Definition~\ref{def of g-d transform}. The number $d$ is also $1$ since $\zeta(s)$ is of Galois-Dedekind $1$-type.
Consequently
\begin{equation*} \ceiling{\log_p\frac{d}{ag_p(p-1)}} = \ceiling{\log_p\frac{1}{p-1}} = 0,\end{equation*}
so the condition $\ceiling{\log_p\frac{d}{ag_p(p-1)}} \leq \nu_p(n)$ is always satisfied, and 
we have that
\[ \mathbb{S}_{GD}\left( \frac{1}{1-p^{-s}}\right)(n) = p^{a\nu_p(n) + a}\]
if $(p-1)\mid n$; finally, this agrees with Definition~\ref{def of unramified straightening transform}.

Consequently, by Proposition~\ref{riemann special value denominators}, the denominator of $\zeta(1-2n)$ is equal to $\mathbb{S}_{GD}(\zeta(s))(2n)$ as well as
$\mathbb{S}_{GD}(\zeta(s))(2n)$.
\end{example}

\begin{theorem}\label{main thm for g-d straightening}
Let $K/\mathbb{Q}$ be a finite field extension with ring of integers $A$. Suppose $K/\mathbb{Q}$ is Galois.
Let $\chi_{\underline{\Delta}_{K/\mathbb{Q}}}$ be the trivial Gro\ss{e}ncharakter of $K$ of conductor equal to the prime-power-ramification discriminant $\underline{\Delta}_{K/\mathbb{Q}}$, and let 
$L(s,\chi_{\underline{\Delta}_{K/\mathbb{Q}}})$ be its associated Hecke $L$-function.

Then, for all positive $n\in\mathbb{N}$, the following numbers are all equal:
\begin{itemize}
\item the order of the abelian group $\Ext^{1,2n}_{(L^A,L^AB)}(L^A,L^A)[{\underline{\Delta}_{K/\mathbb{Q}}}^{-1}]$,
\item the order of the abelian group $H^{1,2n}_{fl}(\mathcal{M}_{fmA}; \mathcal{O})[{\underline{\Delta}_{K/\mathbb{Q}}}^{-1}]$,
\item the order of the abelian group $A[{\underline{\Delta}_{K/\mathbb{Q}}}^{-1}]/J_{n}$, where $J_n$ is the minimal $n$-congruing ideal in $A[{\underline{\Delta}_{K/\mathbb{Q}}}^{-1}]$, and
\item the number $\mathbb{S}_{GD}(L(s,\chi_{\underline{\Delta}_{K/\mathbb{Q}}}))(n)$.
\end{itemize}
\end{theorem}
\begin{proof}
This proof is much the same as that of Theorem~\ref{main thm for unramified straightening}.
Corollary~\ref{weak form of global conj holds} implies that we have an isomorphism of $A$-modules
\[ \Ext^{1,2n}_{(L^A,L^AB)}(L^A,L^A)[{\underline{\Delta}_{K/\mathbb{Q}}}^{-1}] \cong A[{\underline{\Delta}_{K/\mathbb{Q}}}^{-1}]/J_{n},\]
consequently these two abelian groups have equal order. So for the rest of the proof we will concern ourselves with the question of why
the order of $\Ext^{1,2n}_{(L^A,L^AB)}(L^A,L^A)[{\underline{\Delta}_{K/\mathbb{Q}}}^{-1}]$ is equal to $\mathbb{S}_{GD}(L(s,\chi_{\underline{\Delta}_{K/\mathbb{Q}}}))(n)$.

That $\Ext^{1,2n}_{(L^A,L^AB)}(L^A,L^A)[{\underline{\Delta}_{K/\mathbb{Q}}}^{-1}]$ has finite order is proven by the same argument as in the proof of Theorem~\ref{main thm for unramified straightening}.
Now clearly $\mathbb{S}_{GD}(L(s,\chi_{\underline{\Delta}_{K/\mathbb{Q}}}))(n)$ and the order of $\Ext^{1,2n}_{(L^A,L^AB)}(L^A,L^A)[{\underline{\Delta}_{K/\mathbb{Q}}}^{-1}]$ are both positive integers,
so if we show that the $p$-adic valuation of $\mathbb{S}_{GD}(L(s,\chi_{\underline{\Delta}_{K/\mathbb{Q}}}))(n)$ agrees with the $p$-adic valuation of the order of
$\Ext^{1,2n}_{(L^A,L^AB)}(L^A,L^A)[{\underline{\Delta}_{K/\mathbb{Q}}}^{-1}]$ for all prime numbers $p$, then we are done.

Suppose $p$ is a prime number such that some prime $\mathfrak{p}$ of $A$ over $p$ has ramification degree $e$ with the property that
$\log_p\frac{e}{p-1}$ is an integer.
Then $\mathfrak{p} \supseteq (p) \supseteq ({\underline{\Delta}_{K/\mathbb{Q}}})$,
and consequently the Euler product for $L(s, \chi_{\underline{\Delta}_{K/\mathbb{Q}}})$ contains no $p$-local Euler factors. Hence $\nu_p\left(\mathbb{S}_{GD}(L(s,\chi_{\underline{\Delta}_{K/\mathbb{Q}}}))(n)\right) = 0$.
On the other hand, 
for each prime $\underline{\ell}$ of $A$ over a given prime number $\ell$, the completion $\left(\Ext^{1,2n}_{(L^A,L^AB)}(L^A,L^A)\right)^{\hat{}}_{\underline{\ell}}$ is 
a finite $\hat{A}_{\underline{\ell}}$-module, hence has order equal to some power of $\ell$, and $\Ext^{1,2n}_{(L^A,L^AB)}(L^A,L^A)$
splits as a direct sum of its completions at the various maximal ideals $\mathfrak{p}$ of $A$ (this is isomorphism~\ref{splitting into completions} from the
proof of Theorem~\ref{global computation}).
Consequently $\Ext^{1,2n}_{(L^A,L^AB)}(L^A,L^A)[{\underline{\Delta}_{K/\mathbb{Q}}}^{-1}]$ has no summands of order a power of $p$.
Consequently the $p$-adic valuation of the order of $\Ext^{1,2n}_{(L^A,L^AB)}(L^A,L^A)[{\underline{\Delta}_{K/\mathbb{Q}}}^{-1}]$ is zero.

The more important situation is when $p$ is a prime number such that no primes of $A$ over $p$ have ramification degree $e$ such that
$\log_p\frac{e}{p-1}$ is an integer. Suppose $p$ is such a prime.
Then the $p$-adic valuation of the order of $\Ext^{1,2n}_{(L^A,L^AB)}(L^A,L^A)[{\underline{\Delta}_{K/\mathbb{Q}}}^{-1}]$
is equal to the $p$-adic valuation of the order of the completion $(\Ext^{1,2n}_{(L^A,L^AB)}(L^A,L^A))^{\hat{}}_p$, i.e.,
the product of the $p$-adic valuations of the orders of the completions $(\Ext^{1,2n}_{(L^A,L^AB)}(L^A,L^A))^{\hat{}}_{\mathfrak{p}}$ for all primes
$\mathfrak{p}$ of $A$ over $p$.
By Theorem~\ref{local-to-global conj}, the completion $(\Ext^{1,2n}_{(L^A,L^AB)}(L^A,L^A))^{\hat{}}_{\mathfrak{p}}$ is isomorphic as a $\hat{A}_{\mathfrak{p}}$-module
to $\Ext^{1,2n}_{(V^{\hat{A}_{\mathfrak{p}}},V^{\hat{A}_{\mathfrak{p}}}T)}(V^{\hat{A}_{\mathfrak{p}}},V^{\hat{A}_{\mathfrak{p}}})$,
which is in turn isomorphic, by Corollary~\ref{local computation}, to $\hat{A}_{\mathfrak{p}}/H_n$, where $H_n$ is the ideal of $A_{\mathfrak{p}}$
generated by all elements of the form $x^n-1$, where $x\in \hat{A}_{\mathfrak{p}}^{\times}$.
As we showed in Theorem~\ref{local conjecture}, $\hat{A}_{\mathfrak{p}}/H_n \cong 0$ if $p^f-1 \nmid n$ where $p^f$ is the cardinality of the residue field
of $\hat{A}_{\mathfrak{p}}$, and if $p^f-1\mid n$, then $H_n = I_{n/(p^f-1)}$, where $I_{n/(p^f-1)}$ is the ideal in $\hat{A}_{\mathfrak{p}}$ generated by all
elements of the form $a^{n/(p^f-1)}-1$ for elements $a\in \hat{A}_{\mathfrak{p}}$ congruent to $1$ modulo $\mathfrak{p}$ (this is by the Hensel's Lemma
argument at the end of the proof of Corollary~\ref{local computation}).

Finally, Lemma~\ref{valuations predicted by local conjecture} gives us that the order of
$\hat{A}_{\mathfrak{p}}/I_{n/(p^f-1)}$ is equal to $p^{fj}$, where
\begin{align*} j 
 &= 
  \left\{ \begin{array}{ll} 
   e\left(\nu_p(n/(p^f-1)) - \ceiling{\log_p \frac{e}{p-1}}\right) + p^{\ceiling{\log_p\frac{e}{p-1}}} &\mbox{\ if\ } \ceiling{\log_p\frac{e}{p-1}} \leq \nu_p(n/(p^f-1)) \\
   p^{\nu_p(n/(p^f-1))} &\mbox{\ if\ } \ceiling{\log_p\frac{e}{p-1}} > \nu_p(n/(p^f-1)). \end{array}\right. 
\end{align*}
Consequently our formula for the $p$-adic valuation of $\Ext^{1,2n}_{(L^A,L^AB)}(L^A,L^A)[{\underline{\Delta}_{K/\mathbb{Q}}}^{-1}]$ is:
\begin{equation*}\label{p-adic valuation of order of Ext1 3} \nu_p\left( \#\left( \Ext^{1,2n}_{(L^A,L^AB)}(L^A,L^A)[{\underline{\Delta}_{K/\mathbb{Q}}}^{-1}] \right) \right) = \sum_{\mathfrak{p}} (f_{\mathfrak{p}}j_{\mathfrak{p}}(n)),
\end{equation*}
where the sum is taken over all primes $\mathfrak{p}$ of $A$ over $p$ such that
the cardinality $p^{f_{\mathfrak{p}}}$ of the residue field $A/\mathfrak{p}$ satisfies $(p^{f_{\mathfrak{p}}}-1) \mid n$,
and where $j_{\mathfrak{p}}(n)$ is defined by the formula
\begin{equation}\label{def of j} j_{\mathfrak{p}}(n) = \left\{ \begin{array}{ll} 
   e_{\mathfrak{p}}\left(\nu_p(n) - \ceiling{\log_p \frac{e_{\mathfrak{p}}}{p-1}}\right) + p^{\ceiling{\log_p\frac{e_{\mathfrak{p}}}{p-1}}} &\mbox{\ if\ } \ceiling{\log_p\frac{e_{\mathfrak{p}}}{p-1}} \leq \nu_p(n)  \\
   p^{\nu_p(n)} &\mbox{\ if\ } \ceiling{\log_p\frac{e_{\mathfrak{p}}}{p-1}} > \nu_p(n). \end{array}\right. \end{equation}
But $K/\mathbb{Q}$ is assumed Galois, so
the Galois group $G_{K/\mathbb{Q}}$ acts transitively on the set of primes $\mathfrak{p}$ of $A$ over $p$,
and any two such primes $\mathfrak{p}$ have equal ramification degrees and equal residue degrees, and finally, 
the number of distinct primes $\mathfrak{p}$ of $A$ over $p$ 
is equal to $\frac{[K:\mathbb{Q}]}{e_pf_p}$, where $e_p$ and $f_p$ are the ramification degree and residue degree, respectively, of any such prime $\mathfrak{p}$. 
Consequently,
\begin{equation*} \label{p-adic valuation of order of Ext1 5} \nu_p\left( \#\left( \Ext^{1,2n}_{(L^A,L^AB)}(L^A,L^A)[{\underline{\Delta}_{K/\mathbb{Q}}}^{-1}] \right) \right) =  \frac{[K: \mathbb{Q}]}{e_p} j_{p}(n),\end{equation*}
where $j_p(n)$ is equal to the number $j_{\mathfrak{p}}(n)$ defined in~\ref{def of j} for any prime $\mathfrak{p}$ of $A$
over $p$.

Now we turn to the $p$-adic valuation of $\mathbb{S}_{GD}(L(s,\chi_{\underline{\Delta}_{K/\mathbb{Q}}}))(n)$. We are still assuming that $p$ does not divide $\underline{\Delta}_{K/\mathbb{Q}}$.
Again, since $K/\mathbb{Q}$ is Galois, every prime $\mathfrak{p}$ of $A$ over $p$ has the same residue degree, so 
the $p$-local Euler factors in $L(s,\chi_{\underline{\Delta}_{K/\mathbb{Q}}})$, taken all together, are of the form
\[ \left( \frac{1}{1-p^{-f_p s}}\right)^{[K: \mathbb{Q}]/(e_pf_p)},\]
where $e_p,f_p$ are (still) the ramification degree and the residue degree, respectively, of any prime $\mathfrak{p}$ of $A$ over $p$.
Consequently:
\begin{equation*}\label{p-adic valuation of order of Ext1 4} \nu_p\left( \mathbb{S}_{GD}(L(s, \chi_{\underline{\Delta}_{K/\mathbb{Q}}}))(n) \right) = 
 \frac{[K:\mathbb{Q}]}{e_pf_p} h_p(n),\end{equation*}
where $h_p(n)$ is defined by the formula
\begin{align} \nonumber h_p(n) &= \left\{ 
 \begin{array}{ll} 
 0 &\mbox{\ if\ } p^{f_p}-1\nmid n \\
 \frac{[K:\mathbb{Q}]}{g_p}\left( \nu_p(n) - \ceiling{\log_p \frac{[K:\mathbb{Q}]}{f_pg_p(p-1)}}\right) + f_pp^{\ceiling{\log_p \frac{[K:\mathbb{Q}]}{f_pg_p(p-1)}}} & \mbox{\ if\ } p^{f_p}-1\mid n\mbox{\ and\ } \ceiling{\log_p\frac{[K:\mathbb{Q}]}{f_pg_p(p-1)}} \leq \nu_p(n), \\
 f_pp^{\nu_p(n)} & \mbox{\ if\ } p^{f_p}-1\mid n\mbox{\ and\ } \ceiling{\log_p\frac{[K:\mathbb{Q}]}{f_pg_p(p-1)}} > \nu_p(n)
\end{array}\right. \\
\label{equation 1001}                     &= \left\{ 
 \begin{array}{ll} 
 0 &\mbox{\ if\ } p^{f_p}-1\nmid n \\
 e_pf_p\left( \nu_p(n) - \ceiling{\log_p \frac{e_p}{p-1}}\right) + f_pp^{\ceiling{\log_p \frac{e_p}{p-1}}} & \mbox{\ if\ } p^{f_p}-1\mid n\mbox{\ and\ } \ceiling{\log_p\frac{e_p}{p-1}} \leq \nu_p(n), \\
 f_pp^{\nu_p(n)} & \mbox{\ if\ } p^{f_p}-1\mid n\mbox{\ and\ } \ceiling{\log_p\frac{e_p}{p-1}} > \nu_p(n),
\end{array}\right. \\
\end{align}
where $g_p$ is the number of distinct primes of $A$ over $p$, i.e., $g_p = \frac{[K: \mathbb{Q}]}{e_pf_p}$.

Now it takes only a moment's worth of arithmetic to verify, using equations~\ref{def of j} and~\ref{equation 1001},
 that 
\[ \frac{[K: \mathbb{Q}]}{e_p} j_{p}(n) = \frac{[K:\mathbb{Q}]}{e_pf_p} h_p(n),\]
and hence that the $p$-adic valuation of $\#\left( \Ext^{1,2n}_{(L^A,L^AB)}(L^A,L^A)[{\underline{\Delta}_{K/\mathbb{Q}}}^{-1}] \right)$
agrees with the $p$-adic valuation of $\mathbb{S}_{GD}(L(s, \chi_{\underline{\Delta}_{K/\mathbb{Q}}}))(n)$.
\end{proof}

Theorem~\ref{main thm for g-d straightening} is, to me, much more satisfying than
Theorem~\ref{main thm for unramified straightening}, because Theorem~\ref{main thm for g-d straightening} implies really ``global'' statements that do not require any localization whatsoever, while Theorem~\ref{main thm for unramified straightening} always requires that, at the very least, $2$ be inverted.
So Theorem~\ref{main thm for g-d straightening} allows us (when the prime-power-ramification discriminant is trivial)
to work with the Dedekind $\zeta$-function and not just a Hecke $L$-function for some Gro{\ss}encharakter with restrictions on its conductor.
For example:
\begin{corollary}\label{global corollary on dedekind zeta}
Let $K/\mathbb{Q}$ be a finite field extension with ring of integers $A$. Suppose $K/\mathbb{Q}$ is Galois
and suppose that the prime-power-ramification discriminant $\underline{\Delta}_{K/\mathbb{Q}}$ is equal to one.

Then, for all positive $n\in\mathbb{N}$, the following numbers are all equal:
\begin{itemize}
\item the order of the abelian group $\Ext^{1,2n}_{(L^A,L^AB)}(L^A,L^A)$,
\item the order of the abelian group $H^{1,2n}_{fl}(\mathcal{M}_{fmA}; \mathcal{O})$,
\item the order of the abelian group $A/J_{n}$, where $J_n$ is the minimal $n$-congruing ideal in $A$, and
\item the number $\mathbb{S}_{GD}(\zeta_K(s))(n)$.
\end{itemize}
\end{corollary}

Arithmetic equivalence and arithmetic equivalence modulo $m$ were defined in Definitions~\ref{def of arithmetic equivalence}
and~\ref{def of arithmetic equivalence mod m}, respectively.

\begin{corollary}\label{mod m arithmetic equivalence corollary, g-d}
Let $K_1/\mathbb{Q}$ and $K_2/\mathbb{Q}$ be finite field extensions with ring of integers $A_1$ and $A_2$, respectively.
Suppose that $K_1/\mathbb{Q}$ and $K_2/\mathbb{Q}$ are both Galois.
Let $m$ be any integer which is divisible by both $\underline{\Delta}_{K_1/\mathbb{Q}}$ and $\underline{\Delta}_{K_2/\mathbb{Q}}$.
If $K_1$ and $K_2$ arithmetically equivalent modulo $m$, then
for all positive $n\in\mathbb{N}$, 
the order of the abelian group
\[ \Ext^{1,2n}_{(L^{A_1},L^{A_1}B)}(L^{A_1},L^{A_1})[m^{-1}]\]
is equal to the order of the abelian group
\[ \Ext^{1,2n}_{(L^{A_2},L^{A_2}B)}(L^{A_2},L^{A_2})[m^{-1}].\]
\end{corollary}

\begin{corollary}\label{arithmetic equivalence corollary, g-d}
Let $K_1/\mathbb{Q}$ and $K_2/\mathbb{Q}$ be finite field extensions with ring of integers $A_1$ and $A_2$, respectively.
Suppose that $K_1/\mathbb{Q}$ and $K_2/\mathbb{Q}$ are both Galois, and suppose that
$\underline{\Delta}_{K_1/\mathbb{Q}}$ and $\underline{\Delta}_{K_2/\mathbb{Q}}$ are both equal to $1$.
If $K_1$ and $K_2$ arithmetically equivalent, then
for all positive $n\in\mathbb{N}$, 
the order of the abelian group
\[ \Ext^{1,2n}_{(L^{A_1},L^{A_1}B)}(L^{A_1},L^{A_1})\]
is equal to the order of the abelian group
\[ \Ext^{1,2n}_{(L^{A_2},L^{A_2}B)}(L^{A_2},L^{A_2}).\]

Equivalently, in terms of the moduli stack $\mathcal{M}_{fmA}$ of one-dimensional formal $A$-modules over $\Spec A$, and with notation as in Conventions~\ref{running conventions}:
the order of the abelian group
\[ H^{1,2n}_{fl}(\mathcal{M}_{fmA_1}; \mathcal{O})\]
is equal to the order of the abelian group
\[ H^{1,2n}_{fl}(\mathcal{M}_{fmA_2}; \mathcal{O}).\]
\end{corollary}

In Corollaries~\ref{mod m arithmetic equivalence corollary, g-d} and~\ref{arithmetic equivalence corollary, g-d},
the phrase ``for all positive $n\in\mathbb{N}$'' can be replaced with ``for all $n\in\mathbb{Z}$'' without affecting the truth of the statements, since
$\Ext^{1,2n}_{(L^{A},L^{A}B)}(L^{A},L^{A})$ vanishes for $n\leq 0$.

The converse of Corollary~\ref{mod m arithmetic equivalence corollary, g-d} is proven in Theorem~\ref{main equivalence thm after localization}, and the
converse of Corollary~\ref{arithmetic equivalence corollary, g-d} is in Corollary~\ref{main equivalence thm when ppr disc is trivial}.
Theorem~\ref{main equivalence thm after localization} and Corollary~\ref{main equivalence thm when ppr disc is trivial} also imply strengthened versions of
Corollaries~\ref{mod m arithmetic equivalence corollary, g-d} and~\ref{arithmetic equivalence corollary, g-d}:
specifically, in the conclusions of Corollaries~\ref{mod m arithmetic equivalence corollary, g-d} and~\ref{arithmetic equivalence corollary, g-d} we get more 
than just that the two abelian groups are of equal order: we even get that they are {\em isomorphic.}

\subsection{The inverse Galois-Dedekind straightening transform.}

In this section I construct an inverse to the Galois-Dedekind straightening transform, and I use it to prove the main results of this section, Theorem~\ref{main equivalence thm after localization} and Corollary~\ref{main equivalence thm when ppr disc is trivial}.

First I have to say what kind of functions the inverse Galois-Dedekind straightening transform can be applied to. This requires a preliminary definition:
\begin{definition}
Let $h: \mathbb{N} \rightarrow\mathbb{N}$ be a function, and let $p$ be a prime number. I will write $\Xi_p(h)$ for the least positive integer $n$
such that $\nu_p(h(n))>0$, or $\Xi_p(h) = \infty$ if $\nu_p(h(n)) = 0$ for all $n$.
\end{definition}

\begin{definition}
Let $h: \mathbb{N} \rightarrow \mathbb{N}$ be a function. I will say that $h$ is of {\em Dedekind type}
if $h$ satisfies all of the following conditions:
\begin{itemize}
\item For all $n\in\mathbb{N}$, the integer $h(n)$ is positive.
\item For each prime number $p$, either $\Xi_p(h) = \infty$ or the number $\log_p\left( 1+ \Xi_p(h) \right)$ is an integer which divides
$\nu_p\left(h(\Xi_p(h))\right)$.

\end{itemize}
%
\end{definition}

\begin{example}
Let $K/\mathbb{Q}$ be a finite Galois extension with ring of integers $A$, let $m$ be a positive integer, and let $L(s, \chi_m)$ be the Hecke $L$-function of the trivial
Gro{\ss}encharakter of $K$ of conductor $m$. Then the Galois-Dedekind straightening transform $\mathbb{S}_{GD}(L(s,\chi_m))$ of $L(s,\chi_m)$ is
of Dedekind type: clearly $\mathbb{S}_{GD}(L(s,\chi_m))(n) >0$ for all $n$, and for any prime number $p$, the smallest integer $n$
such that $\mathbb{S}_{GD}(L(s,\chi_m))(n)$ is divisible by $p$ is $p^{f_p}-1$, where $f_p$ is the residue degree of any prime of $A$ over $p$.
Hence $\Xi_p\left( \mathbb{S}_{GD}(L(s,\chi_m))\right) + 1$ is a power of $p$, hence 
$\log_p (\Xi_p\left( \mathbb{S}_{GD}(L(s,\chi_m))\right) + 1)$ is an integer, namely, $f_p$.
One can easily use Definition~\ref{def of g-d transform} to check that
\begin{align} 
\nonumber \nu_p\left( \mathbb{S}_{GD}(L(s,\chi_m))\left(\Xi_p\left( \mathbb{S}_{GD}(L(s,\chi_m)) \right)\right)\right) 
  &= \nu_p\left( \mathbb{S}_{GD}(L(s,\chi_m)\right)\left( p^{f_p} -1\right) \\
\label{valuation equation 16}  &= f_pg_p,\end{align}
where $g_p$ is the number of primes of $A$ over $p$.
Hence $\log_p (\Xi_p\left( \mathbb{S}_{GD}(L(s,\chi_m))\right) + 1)$ is indeed an integer which divides
$\nu_p\left( \mathbb{S}_{GD}(L(s,\chi_m))\left(\Xi_p\left( \mathbb{S}_{GD}(L(s,\chi_m)) \right)\right)\right)$.
 
As the special case when $m=1$, the Galois-Dedekind straightening transform $\mathbb{S}_{GD}(\zeta_K(s))$ of any finite Galois extension $K/\mathbb{Q}$ is
of Dedekind type.
\end{example}

\begin{definition}\label{def of inverse g-d s-transform}
Let $h: \mathbb{N}\rightarrow\mathbb{N}$ be a function of Dedekind type.
By the {\em inverse Galois-Dedekind straightening transform of $h$} I mean the formal product
\begin{align*} 
 \mathbb{S}^{-1}_{GD}(h)(s) &=  \prod_{p} \frac{1}{(1-p^{-f_ps})^{g_p}},
 \end{align*}
where $f_p$ and $g_p$ are defined by the formulas
\begin{align*}
 f_p &= \log_p\left(1+ \Xi_p(h)\right) ,\\
 g_p &= \frac{\nu_p\left(h(\Xi_p(h))\right)}{f_p}.\end{align*}
\end{definition}
It is not at all clear that the inverse Galois-Dedekind straightening transform of a function of Dedekind type will converge anywhere in the complex plane at all.
However, inverse Galois-Dedekind straightening transform of a function which is itself the Galois-Dedekind straightening transform of a Hecke $L$-function of
a trivial Gro{\ss}encharakter of some conductor recovers that Hecke $L$-function, hence converges:
\begin{prop}\label{inverse g-d s-transform really is inverse}
Let $K/\mathbb{Q}$ be a finite Galois extension, let $m$ be a positive integer, and let $L(s, \chi_m)$ be the Hecke $L$-function of the trivial
Gro{\ss}encharakter of $K$ of conductor $m$. 
Then the formal product $\mathbb{S}^{-1}_{GD}\left( \mathbb{S}_{GD}L(s,\chi_m)\right)(t)$ converges for all complex numbers $t$ with real part $\Re t>1$, and
furthermore,
\[ \mathbb{S}^{-1}_{GD}\left( \mathbb{S}_{GD}L(s,\chi_m)\right)(t) = L(t,\chi_m).\]
\end{prop}
\begin{proof}
Both $\mathbb{S}^{-1}_{GD}\left( \mathbb{S}_{GD}L(s,\chi_m)\right)(t)$ and $L(t,\chi_m)$ are products of elementary Euler factors, so if we can show
that the two have the same elementary $p$-local Euler factors for all prime numbers $p$, then we will be done.

Let $p$ be a prime number dividing $m$. Then $L(s,\chi_m)$ has no $p$-local elementary Euler factors,
hence $\mathbb{S}_{GD}(L(s,\chi_m))(n)$ is prime to $p$ for all $n$, hence
$\Xi_p\left(\mathbb{S}_{GD}L(s,\chi_m)\right) = \infty$. Hence $\mathbb{S}^{-1}_{GD}\left( \mathbb{S}_{GD}L(s,\chi_m)\right)$
also has no $p$-local elementary Euler factors.

Now let $p$ be a prime number not dividing $m$, and write $A$ for the ring of integers of $K$.
Then there is one $p$-local elementary Euler factors of $L(s,\chi_m)$
for each distinct prime $\mathfrak{p}$ of $A$ over $p$, and that $p$-local elementary Euler factor is equal to
$\frac{1}{1-p^{-f_{\mathfrak{p}}s}}$, where $f_{\mathfrak{p}}$ is the residue degree of $\mathfrak{p}$.
Since $K/\mathbb{Q}$ was assumed Galois, any two primes of $A$ over $p$ have equal residue degrees, so the $p$-local elementary Euler factors of $L(s,\chi_m)$,
taken all together, are equal to 
\begin{equation}\label{local euler factors}\frac{1}{\left( 1-p^{-f_ps}\right)^{g_p}},\end{equation} where $g_p$ is the number of primes of $A$ over $p$, and
$f_p$ is equal to $f_{\mathfrak{p}}$ for any such prime $\mathfrak{p}$.

Hence, from the definition of $\mathbb{S}_{GD}$ in Definition~\ref{def of g-d transform}, the smallest positive integer $n$ such that
$\mathbb{S}_{GD}(L(s,\chi_m))(n)$ is divisible by $p$ is $n = p^{f_p}-1$, and
\[ \mathbb{S}_{GD}(L(s,\chi_m))(p^{f_p}-1) = p^{f_pg_p},\] as one can easily check from Definition~\ref{def of g-d transform}.
Hence $\log_p\left(1+ \Xi_p\left(\mathbb{S}_{GD}(L(s,\chi_m))\right)\right) = f_p$, and 
\[ \nu_p\left(\mathbb{S}_{GD}(L(s,\chi_m))\left(\Xi_p\left(\mathbb{S}_{GD}(L(s,\chi_m))\right)\right)\right) = f_pg_p,\]
as already observed in equation~\ref{valuation equation 16}. 
Consequently the $f_p$ and $g_p$ appearing in the formula~\ref{local euler factors} are equal to the $f_p$ and $g_p$ appearing in 
Definition~\ref{def of inverse g-d s-transform}, i.e.,
the elementary $p$-local Euler factors in
\[ \mathbb{S}_{GD}^{-1}\left( \mathbb{S}_{GD}\left( L(s,\chi_m)\right)\right)(t),\]
taken all together, are equal to~\ref{local euler factors}.
\end{proof}

\begin{lemma}\label{dvr quotient description}
Let $K/\mathbb{Q}_p$ be a finite field extension with ring of integers $A$, ramification degree $e$, and residue degree $f$. Let $\pi$ be a uniformizer for $A$, and let $n$ be a positive integer, and write $n$ as $n = qe+r$, with $q$ a nonnegative integer and $r$ an integer satisfying $0\leq r<e$. Then we have an isomorphism of abelian groups:
\begin{equation}\label{ab grp iso 1} A/\pi^n \cong \left((\mathbb{Z}/p^{q+1}\mathbb{Z})^{\oplus f}\right)^{\oplus r} \oplus \left((\mathbb{Z}/p^{q}\mathbb{Z})^{\oplus f}\right)^{\oplus (e-r)}.\end{equation}
\end{lemma}
\begin{proof}
Since $A$ is the maximal $\hat{\mathbb{Z}}_p$-order in $K$ and $[K:\mathbb{Q}_p] = ef$, we have that $A$ is isomorphic to $\hat{\mathbb{Z}}_p^{\oplus ef}$
as an abelian group. 
Hence $A/(p^m)$ is isomorphic to $\left(\mathbb{Z}/p^m\mathbb{Z}\right)^{\oplus ef}$ as an abelian group. Since $K/\mathbb{Q}_p$ has ramification degree $e$,
we have that $(\pi^e) = (\pi)^e = (p)$, and consequently 
\[ A/\pi^{qe} \cong (\mathbb{Z}/p^{q}\mathbb{Z})^{\oplus ef}\]
as abelian groups, i.e., formula~\ref{ab grp iso 1} holds for $r=0$.

If $0<r<e$, then $A/\pi^n$ sits in a sequence of surjections of $A$-modules
\[ A/\pi^{(q+1)e} \twoheadrightarrow A/\pi^n \twoheadrightarrow A/\pi^{qe},\]
with $A/\pi^{(q+1)e} \cong (\mathbb{Z}/p^{q+1}\mathbb{Z})^{\oplus ef}$ and
$A/\pi^{qe} \cong (\mathbb{Z}/p^{q}\mathbb{Z})^{\oplus ef}$ as abelian groups. 
This observation is the initial step in an induction: suppose $s$ is an integer, $0<s<e$, and
isomorphism~\ref{ab grp iso 1} is known to hold for all $r<s$.

Now we make three observations about $A/\pi^{qe+s}$:
\begin{itemize}
\item $A/\pi^{qe+s}$ sits in a short exact sequence of $A$-modules
\[ 0\rightarrow \pi^{qe+s-1}/\pi^{qe+s} \rightarrow A/\pi^{qe+s} \rightarrow A/\pi^{qe+s-1} \rightarrow 0,\]
and $\pi^{qe+s-1}/\pi^{qe+s} \cong \mathbb{F}_{p^f} \cong (\mathbb{Z}/p\mathbb{Z})^{\oplus f}$ as abelian groups.
\item 
By the inductive hypothesis, we have the isomorphism of abelian groups
\[ A/\pi^{qe+s-1} \cong \left((\mathbb{Z}/p^{q+1}\mathbb{Z})^{\oplus f}\right)^{\oplus (s-1)} \oplus \left((\mathbb{Z}/p^{q}\mathbb{Z})^{\oplus f}\right)^{\oplus (e-s+1)}.\]
\item $A/\pi^{(q+1)e}\cong \left((\mathbb{Z}/p^{q+1}\mathbb{Z})^{\oplus f}\right)^{\oplus e}$ 
surjects on to $A/\pi^{qe+s}$. \end{itemize}
Up to isomorphism, there is only one abelian group $A/\pi^{qe+s}$ which satisfies all three of these properties, namely,
\[ A/\pi^{qe+s} \cong \left((\mathbb{Z}/p^{q+1}\mathbb{Z})^{\oplus f}\right)^{\oplus (s)} \oplus \left((\mathbb{Z}/p^{q}\mathbb{Z})^{\oplus f}\right)^{\oplus (e-s)}.\]
This completes the induction and the proof.
\end{proof}

Now Proposition~\ref{inverse g-d s-transform really is inverse} implies the converse of Corollaries~\ref{mod m arithmetic equivalence corollary, g-d}
and~\ref{arithmetic equivalence corollary, g-d}:
\begin{theorem}\label{main equivalence thm after localization}
Let $K_1/\mathbb{Q}$ and $K_2/\mathbb{Q}$ be finite field extensions with ring of integers $A_1$ and $A_2$, respectively.
Suppose that $K_1/\mathbb{Q}$ and $K_2/\mathbb{Q}$ are both Galois.
Let $m$ be any integer which is divisible by both $\underline{\Delta}_{K_1/\mathbb{Q}}$ and $\underline{\Delta}_{K_2/\mathbb{Q}}$.
Then the following statements are all equivalent:
\begin{enumerate}
\item \label{condition 1 1} $K_1$ and $K_2$ are arithmetically equivalent modulo $m$.
\item \label{condition 1 2} For all positive $n\in\mathbb{N}$, 
the order of the abelian group
\[ \Ext^{1,2n}_{(L^{A_1},L^{A_1}B)}(L^{A_1},L^{A_1})[m^{-1}]\]
is equal to the order of the abelian group
\[ \Ext^{1,2n}_{(L^{A_2},L^{A_2}B)}(L^{A_2},L^{A_2})[m^{-1}].\]
\item \label{condition 1 2stack} For all positive $n\in\mathbb{N}$, 
the order of the abelian group
\[ H^{1,2n}_{fl}(\mathcal{M}_{fmA_1}; \mathcal{O})[m^{-1}]\]
is equal to the order of the abelian group
\[ H^{1,2n}_{fl}(\mathcal{M}_{fmA_2}; \mathcal{O})[m^{-1}].\]
\item \label{condition 1 3} For all positive $n\in\mathbb{N}$, 
the abelian group
\[ \Ext^{1,2n}_{(L^{A_1},L^{A_1}B)}(L^{A_1},L^{A_1})[m^{-1}]\]
is isomorphic to the abelian group
\[ \Ext^{1,2n}_{(L^{A_2},L^{A_2}B)}(L^{A_2},L^{A_2})[m^{-1}].\]
\item \label{condition 1 3stack} For all positive $n\in\mathbb{N}$, 
the abelian group
\[ H^{1,2n}_{fl}(\mathcal{M}_{fmA_1}; \mathcal{O})[m^{-1}]\]
is isomorphic to the abelian group
\[ H^{1,2n}_{fl}(\mathcal{M}_{fmA_2}; \mathcal{O})[m^{-1}].\]
\item \label{condition 1 4} For all $n\in\mathbb{N}$, the order of the abelian group $A_1/(J_{n,1})[m^{-1}]$
is equal to the order of the abelian group $A_2/(J_{n,2})[m^{-1}]$
where $J_{n,1}$ is the minimal $n$-congruing ideal of $A_1[m^{-1}]$.
and $J_{n,2}$ is the minimal $n$-congruing ideal of $A_2[m^{-1}]$.
\item \label{condition 1 5} For all $n\in\mathbb{N}$, the abelian group $A_1/(J_{n,1})[m^{-1}]$
is isomorphic to the abelian group $A_2/(J_{n,2})[m^{-1}]$
where $J_{n,2}$ is the minimal $n$-congruing ideal of $A_1[m^{-1}]$.
and $J_{n,2}$ is the minimal $n$-congruing ideal of $A_2[m^{-1}]$.
\end{enumerate}
\end{theorem}
\begin{proof}
For the duration of this proof, I will write $L(s,\chi_{m,i})$ for the 
Hecke $L$-function of the trivial Gro{\ss}encharakter of conductor $m$ on $K_i$, where $i\in \{ 1,2\}$.

\begin{itemize}
\item {\em Condition~\ref{condition 1 2} is equivalent to condition~\ref{condition 1 2stack} and condition~\ref{condition 1 3} is equivalent to condition~\ref{condition 1 3stack}:} This is the usual cohomology-preserving equivalence of comodules over a Hopf algebroid with quasicoherent modules over the associated algebraic stack.
See Conventions~\ref{running conventions}.
\item {\em Condition~\ref{condition 1 1} implies condition~\ref{condition 1 2}:} If $K_1$ and $K_2$ are arithmetically equivalent modulo $m$, then 
\[ \#\left(\Ext^{1,2n}_{(L^{A_1},L^{A_1}B)}(L^{A_1},L^{A_1})[m^{-1}]\right) = \#\left(\Ext^{1,2n}_{(L^{A_2},L^{A_2}B)}(L^{A_2},L^{A_2})[m^{-1}]\right)\]
by Corollary~\ref{mod m arithmetic equivalence corollary, g-d}.
\item {\em Condition~\ref{condition 1 2} implies condition~\ref{condition 1 1}:} By Theorem~\ref{main thm for g-d straightening}, for $i\in \{ 1,2\}$, the order of $\#\left(\Ext^{1,2n}_{(L^{A_i},L^{A_i}B)}(L^{A_i},L^{A_i})[m^{-1}]\right)$
is equal to 
$\mathbb{S}_{GD}(L(s,\chi_{m,i}))(n)$. By Proposition~\ref{inverse g-d s-transform really is inverse}, we can recover
$L(s,\chi_{m,i})$ from $\mathbb{S}_{GD}(L(s,\chi_{m,i}))$ by applying the inverse Galois-Dedekind straightening transform.
Consequently, if
\[ \#\left(\Ext^{1,2n}_{(L^{A_1},L^{A_1}B)}(L^{A_1},L^{A_1})[m^{-1}]\right) = \#\left(\Ext^{1,2n}_{(L^{A_2},L^{A_2}B)}(L^{A_2},L^{A_2})[m^{-1}]\right),\]
then $L(s,\chi_{m,1}) = L(s,\chi_{m,2})$, so $K_1$ and $K_2$ are arithmetically equivalent modulo $m$.
\item {\em Condition~\ref{condition 1 2} is equivalent to condition~\ref{condition 1 3}:} If 
$\Ext^{1,2n}_{(L^{A_1},L^{A_1}B)}(L^{A_1},L^{A_1})[m^{-1}]$ is isomorphic as an abelian group to $\Ext^{1,2n}_{(L^{A_2},L^{A_2}B)}(L^{A_2},L^{A_2})[m^{-1}]$,
then clearly the two groups have the same order. 

For the converse: suppose
$\Ext^{1,2n}_{(L^{A_1},L^{A_1}B)}(L^{A_1},L^{A_1})[m^{-1}]$ has the same order as $\Ext^{1,2n}_{(L^{A_2},L^{A_2}B)}(L^{A_2},L^{A_2})[m^{-1}]$.
Choose a prime number $p$ not dividing $m$, and write $\mathfrak{p}_{1,1}, \dots ,\mathfrak{p}_{1, g}$ for the set of primes of $A_1$ over $p$, 
and $\mathfrak{p}_{2,1}, \dots ,\mathfrak{p}_{2, g}$ for the set of primes of $A_2$ over $p$. (These two sets do indeed have the same cardinality, namely, the set of 
$p$-local elementary Euler factors in $L(s,\chi_{m,1})$, equivalently, $L(s,\chi_{m,1})$.)
By Corollary~\ref{local computation} and Theorem~\ref{local-to-global conj}, 
the order of the abelian group 
\[ \left(\Ext^{1,2n}_{(L^{A_i},L^{A_i}B)}(L^{A_i},L^{A_i})[m^{-1}]\right)^{\hat{}}_{\mathfrak{p}_{i,j})}\cong \left(\Ext^{1,2n}_{(L^{A_i},L^{A_i}B)}(L^{A_i},L^{A_i})\right)^{\hat{}}_{\mathfrak{p}_{i,j}},\]
the $\mathfrak{p}_{i,j}$-adic completion of $\Ext^{1,2n}_{(L^{A_i},L^{A_i}B)}(L^{A_i},L^{A_i})[m^{-1}]$,
is given by the formula in Lemma~\ref{valuations predicted by local conjecture}, and this formula depends only on the ramification degree and the residue degree
of $\mathfrak{p}_{i,j}$ in the field extension $K_i/\mathbb{Q}$. That ramification degree and that residue degree are both independent of $j$, since
$K_i/\mathbb{Q}$ is Galois, and also independent of $i$, since $K_1,K_2$ are assumed arithmetically equivalent modulo $m$.

Now the $A_i[m^{-1}]$-module $\Ext^{1,2n}_{(L^{A_i},L^{A_i}B)}(L^{A_i},L^{A_i})[m^{-1}]$ is cyclic, by Corollary~\ref{local computation}, hence its $\mathfrak{p}$-adic completion 
$\left(\Ext^{1,2n}_{(L^{A_i},L^{A_i}B)}(L^{A_i},L^{A_i})[m^{-1}]\right)^{\hat{}}_{\mathfrak{p}}$ is a cyclic torsion $(A_i)^{\hat{}}_{\mathfrak{p}}$-module for all maximal ideals $\mathfrak{p}$ of $A_i[m^{-1}]$. Since $(A_i)^{\hat{}}_{\mathfrak{p}}$ is a discrete
valuation ring, every cyclic torsion module over it is isomorphic to
$(A_i)^{\hat{}}_{\mathfrak{p}}/\pi^n$ for some $n$, where $\pi$ is a 
uniformizer for $(A_i)^{\hat{}}_{\mathfrak{p}}$. 
Now by Lemma~\ref{dvr quotient description}, the structure of
$(A_i)^{\hat{}}_{\mathfrak{p}}/\pi^n$ as an abelian group is determined
entirely by $n$ and the ramification degree and residue degree of
the quotient field of $(A_i)^{\hat{}}_{\mathfrak{p}}$ as an extension of $\mathbb{Q}_p$. Since $K_1$ and $K_2$ are arithmetically equivalent modulo $m$ and are each Galois over $\mathbb{Q}$,
the $\mathfrak{p}_1$-adic completion of $K_1$ has the same residue degree and ramification degree over $\mathbb{Q}_p$ as the residue degree and ramification degree of the $\mathfrak{p}_2$-adic completion of $K_2$ over $\mathbb{Q}_p$,
for any choice of prime $\mathfrak{p}_1$ of $A_1$ and prime $\mathfrak{p}_2$ of $A_2$, both over $p$.
Consequently, knowing that 
\begin{align*} \left(\Ext^{1,2n}_{(L^{A_1},L^{A_1}B)}(L^{A_1},L^{A_1})[m^{-1}]\right)^{\hat{}}_{p}
 &\cong \oplus_{\mathfrak{p}_1/p}
  \left(\Ext^{1,2n}_{(L^{A_1},L^{A_1}B)}(L^{A_1},L^{A_1})[m^{-1}]\right)^{\hat{}}_{\mathfrak{p}_1} \end{align*}
and
\begin{align*} \left(\Ext^{1,2n}_{(L^{A_2},L^{A_2}B)}(L^{A_2},L^{A_2})[m^{-1}]\right)^{\hat{}}_{p}
 &\cong \oplus_{\mathfrak{p}_2/p}
  \left(\Ext^{1,2n}_{(L^{A_2},L^{A_2}B)}(L^{A_2},L^{A_2})[m^{-1}]\right)^{\hat{}}_{\mathfrak{p}_2} \end{align*}
have the same order implies that they are isomorphic as abelian groups.

So we have isomorphisms of abelian groups
\begin{align} \label{iso 10000} \left(\Ext^{1,2n}_{(L^{A_1},L^{A_1}B)}(L^{A_1},L^{A_1})[m^{-1}]\right)^{\hat{}}_{p} &\cong \left(\Ext^{1,2n}_{(L^{A_1},L^{A_1}B)}(L^{A_1},L^{A_1})[m^{-1}]\right)^{\hat{}}_{p} \\
\nonumber &\cong T^{\oplus g},\end{align}
for some abelian group $T$ of finite order (that order depends on $n$, and it 
is described by the formula in Lemma~\ref{valuations predicted by local conjecture},
but it is not important for this part of the proof).
The groups 
$\Ext^{1,2n}_{(L^{A_1},L^{A_1}B)}(L^{A_1},L^{A_1})[m^{-1}]$ and $\Ext^{1,2n}_{(L^{A_1},L^{A_1}B)}(L^{A_1},L^{A_1})[m^{-1}]$
are $\mathbb{Z}[m^{-1}]$-modules of finite order, hence are isomorphic as abelian groups if their $p$-adic completions
are isomorphic for all $p$ not dividing $m$.
This indeed happens, since we have the isomorphism~\ref{iso 10000} for all prime numbers $p$ not dividing $m$. 
\item {\em Condition~\ref{condition 1 2} is equivalent to condition~\ref{condition 1 4}, and 
condition~\ref{condition 1 3} is equivalent to condition~\ref{condition 1 5}:} This is the content of Corollary~\ref{weak form of global conj holds}.
\end{itemize}
\end{proof}

\begin{corollary}\label{main equivalence thm when ppr disc is trivial}
Let $K_1/\mathbb{Q}$ and $K_2/\mathbb{Q}$ be finite field extensions with ring of integers $A_1$ and $A_2$, respectively.
Suppose that $K_1/\mathbb{Q}$ and $K_2/\mathbb{Q}$ are both Galois, and suppose that the prime-power-ramification discriminants
$\underline{\Delta}_{K_1/\mathbb{Q}}$ and $\underline{\Delta}_{K_2/\mathbb{Q}}$ are both equal to $1$.
\begin{enumerate}
\item  $K_1$ and $K_2$ are arithmetically equivalent.
\item  For all positive $n\in\mathbb{N}$, 
the order of the abelian group
\[ \Ext^{1,2n}_{(L^{A_1},L^{A_1}B)}(L^{A_1},L^{A_1})\]
is equal to the order of the abelian group
\[ \Ext^{1,2n}_{(L^{A_2},L^{A_2}B)}(L^{A_2},L^{A_2}).\]
\item  For all positive $n\in\mathbb{N}$, 
the order of the abelian group
\[ H^{1,2n}_{fl}(\mathcal{M}_{fmA_1}; \mathcal{O})\]
is equal to the order of the abelian group
\[ H^{1,2n}_{fl}(\mathcal{M}_{fmA_2}; \mathcal{O}).\]
\item  For all positive $n\in\mathbb{N}$, 
the abelian group
\[ \Ext^{1,2n}_{(L^{A_1},L^{A_1}B)}(L^{A_1},L^{A_1})\]
is isomorphic to the abelian group
\[ \Ext^{1,2n}_{(L^{A_2},L^{A_2}B)}(L^{A_2},L^{A_2}).\]
\item  For all positive $n\in\mathbb{N}$, 
the abelian group
\[ H^{1,2n}_{fl}(\mathcal{M}_{fmA_1}; \mathcal{O})\]
is isomorphic to the abelian group
\[ H^{1,2n}_{fl}(\mathcal{M}_{fmA_2}; \mathcal{O}).\]
\item  For all $n\in\mathbb{N}$, the order of the abelian group $A_1/(J_{n,1})$
is equal to the order of the abelian group $A_2/(J_{n,2})$
where $J_{n,1}$ is the minimal $n$-congruing ideal of $A_1$.
and $J_{n,2}$ is the minimal $n$-congruing ideal of $A_2$.
\item For all $n\in\mathbb{N}$, the abelian group $A_1/(J_{n,1})$
is isomorphic to the abelian group $A_2/(J_{n,2})$
where $J_{n,2}$ is the minimal $n$-congruing ideal of $A_1$.
and $J_{n,2}$ is the minimal $n$-congruing ideal of $A_2$.
\end{enumerate}
\end{corollary}

In Theorem~\ref{main equivalence thm after localization} and Corollary~\ref{main equivalence thm when ppr disc is trivial},
the phrase ``for all positive $n\in\mathbb{N}$'' can be replaced with ``for all $n\in\mathbb{Z}$'' without affecting the truth of the statements, since
$\Ext^{1,2n}_{(L^{A},L^{A}B)}(L^{A},L^{A})$ vanishes for $n\leq 0$.

\begin{observation}\label{what degree 0 and 1 cohomology detect}
Suppose $K_1,K_2$ are finite extensions of $\mathbb{Q}$, with ring of integers $A_1,A_2$, respectively. Then 
the graded abelian group $\Ext^{0,*}_{(L^{A_1}, L^{A_1}B)}(L^{A_1}, L^{A_1})$ is isomorphic to 
the graded abelian group $\Ext^{0,*}_{(L^{A_2}, L^{A_2}B)}(L^{A_2}, L^{A_2})$ if and only if $[K_1: \mathbb{Q}] = [K_2: \mathbb{Q}]$.
This is easy to prove: it is because $\Ext^{0,*}_{(L^{A_i}, L^{A_i}B)}(L^{A_i}, L^{A_i}) \cong A_i$ concentrated in grading degree zero,
and, as an abelian group, $A_i$ is isomorphic to a $[K_i: \mathbb{Q}]$-fold direct sum of copies of $\mathbb{Z}$.
As a slogan: ``degree $0$ cohomology of the moduli stack of formal $A$-modules detects the degree of $K/\mathbb{Q}$.''

Corollary~\ref{main equivalence thm when ppr disc is trivial} then becomes, as a slogan,
``degree $1$ cohomology of the moduli stack of formal $A$-modules detects the arithmetic equivalence class of $K/\mathbb{Q}$,''
at least when $K/\mathbb{Q}$ is Galois and the prime-power-ramification discriminant $\overline{\Delta}_{K/\mathbb{Q}}$ is equal to $1$.
\end{observation}

I know of no reason why the restriction on the prime-power-ramification discriminant in Observation~\ref{what degree 0 and 1 cohomology detect}
cannot be lifted, aside from that the necessary computations (generalizing those in the proof of Theorem~\ref{local conjecture})
are simply much harder. I suspect that that restriction can indeed be lifted, i.e., that the following is true:
\begin{conjecture}{\bf (Arithmetic equivalence conjecture.)}\label{arithmetic equivalence conjecture}
Let $K_1/\mathbb{Q}$ and $K_2/\mathbb{Q}$ be finite field extensions with ring of integers $A_1$ and $A_2$, respectively.
Suppose that $K_1/\mathbb{Q}$ and $K_2/\mathbb{Q}$ are both Galois.
Then $K_1$ and $K_2$ are arithmetically equivalent if and only if, for all $n\in\mathbb{N}$, 
the abelian group
\[ \Ext^{1,2n}_{(L^{A_1},L^{A_1}B)}(L^{A_1},L^{A_1})\]
is isomorphic to the abelian group
\[ \Ext^{1,2n}_{(L^{A_2},L^{A_2}B)}(L^{A_2},L^{A_2}).\]

Equivalently, in terms of the moduli stack $\mathcal{M}_{fmA}$ of one-dimensional formal $A$-modules over $\Spec A$, and with notation as in Conventions~\ref{running conventions}:
$K_1$ and $K_2$ are arithmetically equivalent if and only if, for all $n\in\mathbb{N}$, 
the abelian group
\[ H^{1,2n}_{fl}(\mathcal{M}_{fmA_1}; \mathcal{O})\]
is isomorphic to the abelian group
\[ H^{1,2n}_{fl}(\mathcal{M}_{fmA_2}; \mathcal{O}).\]
\end{conjecture}

The assumptions in Conjecture~\ref{arithmetic equivalence conjecture} matter; in particular,
if $K_1,K_2$ are not both Galois, then I do not see any reason to expect anything like Conjecture~\ref{arithmetic equivalence conjecture}
to hold.

\bibliography{/home/asalch/texmf/tex/salch}{}
\bibliographystyle{plain}
\end{document}